\documentclass[a4paper]{article}
\usepackage{newunicodechar}
\newunicodechar{∶}{:} % 全角冒号映射为半角冒号
\usepackage[utf8]{inputenc}
\usepackage{amsfonts}
\usepackage{latexsym,amssymb}
\usepackage[fleqn]{amsmath}
\usepackage{amsbsy}
\usepackage{amsopn, amstext}
\usepackage{graphicx, color, epstopdf}
\usepackage{threeparttable}
\usepackage{authblk}
\usepackage{amsthm}
\usepackage{float}
\usepackage{subfig}
\usepackage{multirow}

\newtheorem{Lemma}{Lemma}
\newtheorem{theorem}{Theorem}
\newtheorem{remark}{Remark}
\newtheorem{example}{Example}

\numberwithin{equation}{section}
\numberwithin{Lem}{section}
\numberwithin{Defi}{section}
\numberwithin{Theo}{section}
\numberwithin{Rem}{section}
\numberwithin{Coro}{section}
\numberwithin{Fig}{section}

\title{A Fractional spectral method for weakly singular Volterra integro-differential equations with delays of the third-kind}
\author[1]{Borui Zhao\thanks{ m202270040@hust.edu.cn}}
\affil[1]{School of Mathematics and Statistics, Huazhong University of Science and Technology, Wuhan 430074, China}
\date{}

\begin{document}
\maketitle
\begin{abstract}
	In this paper, we present a fractional spectral collocation method for solving a class of weakly singular Volterra integro-differential equations (VDIEs) with proportional delays and cordial operators. Assuming the underlying solutions are in a specific function space, we derive error estimates in the $L^2_{\omega^{\alpha,\beta,\lambda}}$ and $L^{\infty}$-norms. A rigorous proof reveals that the numerical errors decay exponentially with the appropriate selections of parameters $\lambda$. Subsequently, numerical experiments are conducted to validate the effectiveness of the method.
\end{abstract}

\section{Introduction}\label{section_Introduction}
We consider the third-kind VIDEs with weakly singular kernels and proportional delays, which typically describe dynamic systems with memory and delay effects:
	\begin{align}
		t^{\gamma} y^{\prime}(t)= & p(t) y(t)+q(t) y(\varepsilon t)+g(t)+\int_0^t(t-s)^{-\mu}  s^{\mu+\gamma-1} K_1(t, s) y(s) ds \notag \\
		&+\varepsilon^{-\gamma} \int_0^{\varepsilon t}(\varepsilon t-\tau)^{-\mu} \tau^{\mu+\gamma-1} K_2(t, \tau) y(\tau) d \tau ,\  \
		t \in I_0:=[0,T],   \label{eq_tr_prime}\\
		y(0)=&  y_0,\label{eq_tr_origin}
	\end{align}
where$\quad 0 \leq \mu<1$, $\gamma>0$, $\mu+\gamma \geq 1$, $p{(t)}$, $q{(t)}$, $g{(t)}$, $K_1(t, s)$, $K_2(t, \tau)$ are given smooth functions,which equals to:
	\begin{align} 
		y^{\prime}(t)=&p_1(t) y(t)+q_1(t) y(\varepsilon t)+g_1(t)+\left(\mathcal{K}_1y\right)(t)+\left(\mathcal{K}_2y\right)(t),  \   \
		t \in I_0:=[0,T] ,\label{eq_prime}\\
		y(0)=& y_0.\label{eq_origin}
	\end{align}
where 
	 \begin{align}
	 	&p_1(t)=t^{-\gamma} p(t), q_1(t)=t^{-\gamma} q(t), g_1(t)=t^{-\gamma} g(t),\\
	 	&\left(\mathcal{K}_1y\right)(t)=t^{-\gamma} \int_0^t(t-s)^{-\mu} s^{\mu+\gamma-1} 
	 	K_1(t, s) y(s) d s,\\
	 	&\left(\mathcal{K}_2y\right)(t)=(\varepsilon t)^{-\gamma} \int_0^{\varepsilon t}(\varepsilon t-\tau)^{-\mu} \tau^{\mu+\gamma-1} K_2(t, \tau) y(\tau) d \tau.
	 \end{align}
\eqref{eq_prime} is called cordial $Volterra$ integro-differential equations (CVIDEs), $\left(\mathcal{K}_iy\right)(t)(i=1,2)$ are viewed as cordial operators\cite{operator3_1}\cite{operator3_2}, which demonstrate the characteristic as
	$$
	\left(\mathcal{K}_iy\right)(t):\left\{\begin{array}{l}
		K_i(0,0)=0: \text { is a compact operator,} \\
		K_i(0,0) \neq 0: \text { is a noncompact operator.}
	\end{array}\right.
	$$

Volterra integral equations (VIEs) and Volterra integro-differential equations (VIDEs) are widely used in mathematical models of biology and physics, with extremely important applications in population dynamics\cite{application_population} and viscoelastic phenomena\cite{application_粘弹性现象}. In recent years, research on the second kind of VIEs and VIDEs has been particularly extensive. Chen and Tang\cite{spectal_TangTao} have innovatively constructed a spectral collocation method for the second kind of weakly singular VIEs. Based on this, Chen and others have  proposed a series of numerical schemes\cite{spectral_Chen_WVIDEs}-\cite{spectral_Chen_WVIDEs_delay} for the second kind of VIDEs and with proportional delays. 
In addition, other methods are used to solve the second kind of VIEs and VIDEs, including collocation methods\cite{collocation_1}-\cite{collocation_8} and spectral methods\cite{spectral_1}-\cite{spectral_4}. 
The works\cite{VIDEs_delays_1}-\cite{VIDEs_delays_5} use continuous and discontinuous Petrov–Galerkin methods to solve delay-VIDEs, besides, $Sinc$ and $Tau$ methods. More importantly, \cite{WVIDEs_hp_MaZheng} and \cite{WVIDEs_hp_QinYu} have constructed hp-version methods to solve the second kind of weakly singular VIDEs and weakly singular VIDEs with delay terms, achieving the exponential rate of convergence.

However, research on the third-kind of VIEs and VIDEs is relatively scarce. Currently, numerical  methods for the third-kind of VIEs with the integral operator $\mathcal{K}_1$
include the Legendre-Galerkin method\cite{3VIEs_Legendre-Galerkin}, among others. Ma and Huang\cite{3VIEs_MaZheng} have proposed a collocation method that is suitable for non-smooth solutions, achieving the optimal order of global convergence through modified graded meshes. Ma\cite{3VIEs_MaXiaohua} and others have used the Chebyshev spectral collocation method after smooth transformation to solve the equations, obtaining spectral accuracy and providing a rigorous convergence analysis.For the third-kind of VIDEs, Jiang and Ma\cite{3VIDEs_mu0_Jiang}, Tang and Tohidi\cite{3VIDEs_mu0_Tang}, Shayanfard\cite{3VIDEs_mu0_Shayanfard} have considered the case where $\mu = 0$
, applying methods such as the Chebyshev and generalized mapping Laguerre spectral collocation to solve the equations. For CVIDEs with non-compact kernel cordial operators, \cite{CVIDEs_piecewise-polynomial} provides a numerical scheme using a piecewise polynomial collocation method. Ma and Huang\cite{CVIDEs_MaXiaohua} have considered non-smooth solutions for the equations after smooth transformation, applying spectral collocation methods to achieve high precision and low computational cost.

Hou and others \cite{Muntz} have applied fractional Jacobi polynomials to VIEs and obtained excellent numerical results. There
have been also some recent developments such as \cite{3VIEs_MaZheng} and \cite{CVIDEs_MaZheng}, Ma and Huang have proposed an hp-version error estimate for the third-kind of VIEs and VIDEs with non-compact operators, using the $H^1$-norm as a weighting factor, to address non-smooth solutions. Fractional Jacobi polynomials have been applied in the two airticles. The purpose of this paper is to design a fractional polynomial collocation method with a fractional coefficient $\lambda(0 < \lambda \leq 1)$ for solving the third-kind of VIDEs with proportional delays, and to provide a convergence analysis. The solutions of the equations \eqref{eq_prime} often exhibit weak singularity at the initial point $t=0$. We can construct numerical solutions in the form of $\sum_{m=0}^{N}a_mt^{m\lambda}$\cite{Ma毕业论文}, and choose an appropriate $\lambda$ in order to counteract this singularity.

This paper is organized as follows. In Section \ref{section_Regularity }, we will prove the regularity of the solution of \eqref{eq_prime}. In Section \ref{section_Preliminaries}, we will introduce specific function spaces and lemmas that will be used to construct convergence results in the following text. In Section \ref{section_Numerical Scheme}, a numerical scheme will be constructed for the equation \eqref{eq_prime} and apply fractional Jacobi polynomials to estimate the exact solutions. In Section \ref{section_Convergence Analysis}, we will demonstrate the convergence analysis of the proposed method under the weighted $L^{\infty}$ and $L^2_{\omega^{\alpha,\beta}}$-norm. In Section \ref{section_Numerical experiments},  numerical experiments will be presented to prove the theoretical analysis of Section \ref{section_Convergence Analysis}.

\section{Regularity of the Solution}\label{section_Regularity }
In \cite{Ma毕业论文}, the existence and uniqueness of solution in \eqref{eq_prime} with $\mu \neq 0$ and the regularity results have already been presented. Similarly, we extend the equation to the case with the delay term $\varepsilon$.

\begin{theorem}
	Let $D = {\lbrace(t,s),(t,\tau):0 \leq s,\tau \leq t \leq T\rbrace}$, $ K_1(t, s),K_2(t, \tau) \in C(D)$ and $p_1{(t)}$,$q_1{(t)}$,$g_1{(t)}$\\
	$ \in C[0, T]$, then\eqref{eq_prime} has a unique solution $y \in C^1{[0, T]}$ on the interval $[0, T]$. Furthermore, when $p_1{(t)},q_1{(t)},g_1{(t)} \in C^{m}[0, T], K_1(t, s),K_2(t, \tau) \in C^{m}(D)$, we can obtain the solution of \eqref{eq_prime} $y \in C^{m+1}[0, T]$. 
\end{theorem}\label{theo_regularity}
\begin{proof}
Let $A=\max_{t\in[0, T]}  p_1(t)$, $C=\max_{t\in[0, T]}  q_1(t)$, $ K_1=\max_{D} K_1(t,s)$, $K_2=\max_{D} K_2(t, \tau)$, For $v \in X:=\left\{\omega: \omega \in C[0, b], \omega_0=y_0\right\}$. Let b be a positive number such that $0 \leq t \leq b$ and \\ $b\{A+C+(K_1+K_2){B\left(1-\mu,\mu+\gamma)\right.}\}<1$, where $B$ is the Beta function. The mapping $y=\delta(v)$ is defined as the following form:
	\begin{align}\label{eq_regularity_map}
		\notag y^{\prime}(t)=&p_1(t) v(t)+q_1(t) v(\varepsilon t)+g_1(t)\\
		\notag&+t^{-\gamma} \int_0^t(t-s)^{-\mu} s^{\mu+\gamma-1} K_1(t, s) v(s) d s\\
		&+(\varepsilon t)^{-\gamma} \int_0^{\varepsilon t}(\varepsilon t-\tau)^{-\mu} \tau^{\mu+\gamma-1} K_2(t, \tau) v(\tau) d \tau,\quad y(0)=y_0,\quad
		0<t \leq b.
	\end{align}
By making the variable transformation $\tau=\varepsilon s $ and $s=t \xi$, \eqref{eq_regularity_map} can be transformed into
	\begin{align}\label{eq_regularity_trans}
		\notag y^{\prime}(t)=& p_1(t) v(t)+q_1(t) v(\varepsilon t) + g_1(t)\\
		\notag &+\int_0^1(1-\xi)^{-\mu} \xi^{\mu+\gamma-1} K_1(t, t \xi) v(t \xi) d \xi \\
		& +\int_0^1(1-\xi)^{-\mu} \xi^{\mu+\gamma-1} K_2(t, \varepsilon t \xi) v(\varepsilon t \xi) d \xi.
	\end{align}
For the integral of the above equation over $[0, t]$, we can obtain
		\begin{align}\label{eq_regularity_integral}
		y(t)=&y_0 +\int_0^t\left(p_1(s) v(s)+q_1(s) v(\varepsilon s)+g_1(s)\right) d s \notag \\
		& +\int_0^t \int_0^1(1-\xi)^{-\mu} \xi^{\mu+\gamma-1} K_1(s, s \xi) v(s \xi) d \xi d s \notag \\
		& +\int_0^t \int_0^1(1-\xi)^{-\mu} \xi^{\mu+\gamma-1} K_2(s, \varepsilon s \xi) v(\varepsilon s \xi) d \xi d s.
	\end{align}
Through \eqref{eq_regularity_trans} and \eqref{eq_regularity_integral}, it is known that for $0<t \leq b$, we have
	\begin{align}
	& |y(t)| \leq\left|y_0\right|+b(A+C)\|v\|_{[0, b]}+b\left\|g_1\right\|_{[0, b]}+b\left(K_1+K_2\right) B(1-\mu, \mu+\gamma)\|v\|_{[0, b]}, \\
	& \left|y^{\prime}(t)\right| \leq(A+C)\|v\|_{[0, b]}+\left\|g_1\right\|_{[0, b]}+\left(K_1+K_2\right) B\left(1-\mu, \mu+\gamma\right)\|v\|_{[0, b]},
	\end{align}
Therefore, through the above derivation, it can be known that the mapping $\delta: X \rightarrow X \cap C^{1}[0, b]$. Let $y_1 = \delta\left(v_1\right), y_2 = \delta(v_2)$:
	\begin{align}
		\notag \left(y_1-y_2\right)(t)= & \int_0^t p_1(s)\left(v_1-v_2\right)(s) d s+\int_0^t q_1(s)\left(v_1-v_2\right)(\varepsilon s) d s \\
		\notag& +\int_0^t \int_0^1(1-\xi)^{-\mu} \xi^{\mu+\gamma-1} K_1(s, s \xi)\left(v_1-v_2\right)(s \xi) d \xi d s \\
		& +\int_0^t \int_0^1(1-\xi)^{-\mu} \xi^{\mu+\gamma-1}K_2(s, \varepsilon s \xi)\left(v_1-v_2\right)(\varepsilon s \xi)d \xi d s ,
	\end{align}
Because of $t \leq b$, take the absolute value of both sides,
	\begin{align}
	|\left(y_1-y_2\right) (t)| \leq b(A+C)\left\|v_1-v_2\right\|_{[0, b]}+b\left(K_1+K_2\right) 
	B(1-\mu,\mu+\gamma)\left\|v_1-v_2\right\|_{[0, b]},
  	\end{align}
that is:
	\begin{align}
		\left\|y_1-y_2\right\|_{[0, b]} \leq
		\left(b(A+C)+b\left(K_1+K_2\right) 
		B(1-\mu,\mu+\gamma)\right)\left\|v_1-v_2\right\|_{[0, b]},
	\end{align}
It can be developed that the mapping $\delta$ is a contraction mapping from $X$ to $X \cap C^{1}[0, b]$, that is, \eqref{eq_prime} have a unique solution $y \in C^{1}[0, b]$. Similarly, we can prove by induction that the equation \eqref{eq_prime} have a unique solution $y \in C^{1}[0, T]$.

Next, we consider the case where $p_1{(t)}, q_1{(t)}, g_1{(t)} \in C^m[0, T]$, $K_1(t, s), K_2(t, \tau) \in C^m(D)$, and assume that $y \in C^j[0, T]$, where $1 \leq j \leq m$. According to \cite{Ma毕业论文}, the operator $\left(\mathcal{K}_{i}y\right)(t)$ maps $C^j[0, T]$ to $C^j[0, T]$. Note that the right-hand side of \eqref{eq_prime} belongs to $C^j[0, T]$, so $y \in C^{j+1}[0, T]$. As previously proved, $y \in C^{1}[0, T]$ can be verified. By mathematical induction, we obtain $y \in C^{m+1}[0, T]$.
\end{proof}

\section{Preliminaries}\label{section_Preliminaries}
In the following text, $C$ represents a generic positive constant that is independent of $N$, the number of selected $Jacobi-Gauss$ points. $C$ only depends on the $T$ of the interval $[0, T]$ and the bounds of the given functions. Let $I:=[0,1]$.

\subsection{Function spaces}
Let $r\geq 0$ and $\kappa \in[0,1]$. Define a space denoted by $C^{r, \kappa}(I)$, in which all functions whose $r$-th derivatives are Hölder continuous with exponent $\kappa$. Its norm is defined as:
	\begin{align}
	\|f(\theta)\|_{r, \kappa}=\max _{0 \leq i \leq r} \max _{\theta \in I}\left|f^{(i)}(\theta)\right|+\max _{0 \leq i \leq r} \sup _{\theta, \eta \in I, \theta \neq \eta} \frac{\left|f^{(i)}(\theta)-f^{(i)}(\eta)\right|}{|\theta-\eta|^\kappa} ,f(\theta) \in C^{r, \kappa}(I) .\label{eq_C_rk}
	\end{align}
When $\kappa = 0$, $C^{r,0}(I)$ is called the space of functions with $r$ continous derivatives on I,which equals to the common space $C^{r}(I)$.\\
Define the $\lambda$-polynomial space \cite{Muntz}:
	$$
	P_n^\lambda\left(\mathbb{R}^{+}\right):=\operatorname{span}\left\{1, \theta^\lambda, \theta^{2 \lambda}, \ldots, \theta^{n \lambda}\right\},
	$$
where $\mathbb{R}^{+}=[0,+\infty), 0<\lambda \leq 1$. $n$-th $\lambda$-polynomials can be presented as 
	$$
	p_n^\lambda(\theta):=k_n \theta^{n \lambda}+k_{n-1} \theta^{(n-1) \lambda}+\cdots+k_1 \theta^\lambda+k_0, \quad k_n \neq 0, \theta \in \mathbb{R}^{+}.
	$$
The sequence of $\{p_n^\lambda\}_{n=0}^{\infty}$ is called to be orthogonal in $L_{\omega}^{2}(I)$ if
	$$
		\left(p_m^\lambda, p_n^\lambda\right)_\omega=\int_0^1 p_n^\lambda(\theta) p_m^\lambda(\theta) \omega(\theta) d \theta=\gamma_m \delta_{m,n},
	$$
where $\gamma_n=\left\|p_n^\lambda\right\|_{0, \omega}^2:=\left(p_n^\lambda, p_n^\lambda\right)_\omega \text { and } \delta_{m, n} \text { is the } Kronecker\text{ delta}.$\\
Then we define $P_n^\lambda(I)$ - space which satisfies
	$$
		P_n^\lambda(I):=\operatorname{span}\left\{p_0^\lambda, p_1^\lambda, \ldots, p_n^\lambda\right\},
	$$
This is a special type of Müntz space\cite{Ma毕业论文}, specifically, $P_n^1(I)$ represents the space formed by polynomials of degree no more than $n$.\\
Now we introduce the non-uniformly Jacobi-weighted Sobolev space\cite{Muntz}
	$$
	B_{\alpha, \beta}^{m,1}(I)=\left\{u(\theta): \partial_\theta^k u(\theta) \in L_{\omega^{\alpha+k, \beta+k}}^2(I), 0 \leq k \leq m\right\},m \in \mathbb{N}
	$$
endowed with the inner product, semi-norm and norm:
	$$
	\begin{aligned}
		& (u, v)_{B_{\alpha, \beta}^{m,1}}=\sum_{k=0}^m\left(\partial_\theta^k u, \partial_\theta^k v\right)_{\omega^{\alpha+k, \beta+k,1}},\\
		& |u|_{B_{\alpha, \beta}^{m,1}}=\left\|\partial_\theta^m u\right\|_{0,\omega^{\alpha+m, \beta+m,1}}, \quad\|u\|_{B_{\alpha, \beta}^{m,1}}=(u, u)_{B_{\alpha, \beta}^{m,1}}^{\frac{1}{2}} .
	\end{aligned}
	$$

\subsection{Fractional Jacobi polynomials}
The fractional $Jacobi$ polynomials are derived from the standard $Jacobi$ polynomials via a variable transformation, which effectively maps the interval $[-1,1]$ to $[0,1]$. The relationship between the two is as follows\cite{Muntz}
	$$
	J_n^{\alpha, \beta, \lambda}(\theta)=J_n^{\alpha, \beta}\left(2 \theta^\lambda-1\right), \quad \forall \theta \in I,\alpha, \beta>-1,0<\lambda \leq 1.
	$$
The form of the classical $Jacobi$ polynomials $J_n^{\alpha, \beta}(\theta)$ is 
	$$
	J_n^{\alpha, \beta}(\theta)=\frac{\Gamma(n+\alpha+1)}{n!\Gamma(n+\alpha+\beta+1)} \sum_{k=0}^n\binom{n}{k} \frac{\Gamma(n+k+\alpha+\beta+1)}{\Gamma(k+\alpha+1)}\left(\frac{\theta-1}{2}\right)^k .
	$$
So fractional $Jacobi$ polynomials can be readily derived,
	$$
	J_n^{\alpha, \beta, \lambda}(\theta)=\frac{\Gamma(n+\alpha+1)}{n!\Gamma(n+\alpha+\beta+1)} \sum_{k=0}^n\binom{n}{k} \frac{\Gamma(n+k+\alpha+\beta+1)}{\Gamma(k+\alpha+1)}\left(\theta^\lambda-1\right)^k .
	$$
The weight function is as follows:
\begin{align}\label{eq_wight_equation}
	\qquad\qquad\qquad\qquad\qquad
	\omega^{\alpha, \beta, \lambda}(\theta):=\lambda\left(1-\theta^\lambda\right)^\alpha \theta^{(\beta+1) \lambda-1}.
\end{align}
The relationship between the two sets of $Jacobi$-$Gauss$ quadrature nodes and their corresponding weights also exist. Let the standard $Jacobi$-$Gauss$  quadrature nodes be $\lbrace t_i \rbrace_{j=0}^{N}$, with weights $\lbrace w_j \rbrace_{j=0}^{N}$. $\lbrace \theta_{j},\omega_{j} \rbrace_{j=0}^{N}$ are denoted to be the fractional $Jacobi$-$Gauss$  quadrature nodes and weights on $I_0 := [0,1]$. We demonstrate the relationship:
	$$
		\theta_{j}=(\frac{t_j+1}{2})^{\frac{1}{\lambda}},\omega_{j}=2^{-(\alpha+\beta+1)}w_j,\quad j=0,...,N.
	$$
Note that $\left\{F_{j,\lambda}\left(\theta\right)\right\}_{j=0}^N$ represent the generalized Lagrange interpolation basis functions
	$$
	F_{j,\lambda}\left(\theta\right)=\prod_{i=0, i \neq j}^N \frac{\theta^\lambda-\theta_i^\lambda}{\theta_j^\lambda-\theta_i^\lambda}, \quad 0 \leq j \leq N,
	$$
where $\theta_0<\theta_1<\cdots<\theta_{N-1}<\theta_N$ are zeros of the fractional $Jacobi$ polynomials $J_{N+1}^{\alpha, \beta, \lambda}(\theta)$ and $F_{j,\lambda}\left(\theta\right)$ clearly satisfy
	$$
	F_{j,\lambda}\left(\theta_i\right)=\delta_{i j} .
	$$
We define the generalized interpolation operator $I_{N, \lambda}^{\alpha, \beta}$ as follows
	\begin{align}\label{eq_Interpolation_operator_sum}
		\qquad \qquad  I_{N, \lambda}^{\alpha, \beta} v(\theta)=\sum_{j=0}^N v\left(\theta_j\right) F_{j,\lambda}\left(\theta\right)=\sum_{j=0}^N v\left(z_j^{\frac{1}{\lambda}}\right) F_{j,1}\left(z\right)=I_{N, 1}^{\alpha, \beta}v(z^{\frac{1}{\lambda}}),\quad \theta=z^{\frac{1}{\lambda}}.
	\end{align}

\subsection{Elementary lemmas}
In this section, we will present the lemmas required for the convergence analysis in section \ref{section_Convergence Analysis}.

\begin{Lemma}\label{lemma_Gronwall}
	$($Gronwall inequality$)$ Assume that
	$$
	f(\theta) \leq g(\theta)+C \int_0^{\theta} f(\eta) d \eta, \quad 0 \leq \theta \leq 1,
	$$\\
	If $f(\theta), g(\theta)$ are non-negative integrable functions on $[0, 1]$ and $C > 0$, then there exists $L>0$ such that:
	$$
	f(\theta) \leq g(\theta)+L \int_0^{\theta} g(\eta) d \eta, \quad 0 \leq \theta \leq 1 .
	$$
\end{Lemma}

\begin{Lemma}\label{lemma_Chen_Lemma3.5_norm}$($see $\cite{spectral_Chen_WVIDEs_delay}$$)$
If $E(\theta)$ is a nonnegative integrable function which satisfies
	$$
	\begin{aligned}
		E(\theta)\leq J(\theta)+L\int_{0}^{\theta}E(\eta)d\eta,0\leq\theta\leq1,\notag
	\end{aligned}
	$$
where $J(\theta)$ is a integrable function and $L$ is a positive constant. Then the following conclusion holds
	$$
	\begin{aligned}
		&\left\|E(\theta)\right\|_{\infty}\leq C\left\|J(\theta)\right\|_{\infty} ,\notag\\
		&\left\|E(\theta)\right\|_{0,\omega^{\alpha,\beta,1}}\leq C\left\|J(\theta)\right\|_{0,\omega^{\alpha,\beta,1}}.\notag
	\end{aligned}
	$$
\end{Lemma}

\begin{Lemma}\label{lemma_kappa_norm}$($ see $\cite{Lemma_linear_operator_1}$,$\cite{Lemma_linear_operator_2}$$)$
	If $r$ is a nonnegative integer and real number $\kappa \in(0,1)$, there exists a linear operator $\mathcal{T}_N$ that maps $C^{r, \kappa}(I)$ to $P_N^1(I)$, such that
	$$
	\left\|v-\mathcal{T}_N v\right\|_{\infty} \leq C_{r, \kappa} N^{-(r+\kappa)}\|v\|_{r, \kappa}, \quad v \in C^{r, \kappa}(I),
	$$
	$C_{r, \kappa}$ is a constant that may depend on $r$ and $\kappa$. For the linear weakly singular integral operators $\mathcal{K}_i$ $($where $i = 1,2$$)$ defined in the previous section:
	\begin{align}
		&(\mathcal{K}_1 v)(\theta)=\theta^{-\gamma}\int_0^{\theta}(\theta-\eta)^{-\mu}\eta^{\mu+\gamma-1} K_1(\theta, \eta) v(\eta) d \eta,\\
		&(\mathcal{K}_2 v)(\theta)=(\varepsilon \theta)^{-\gamma}\int_0^{\varepsilon \theta}(\varepsilon \theta-\eta)^{-\mu}\eta^{\mu+\gamma-1} K_2(\theta, \eta) v(\eta) d \eta.\label{eq_Operator}
	\end{align}
	$K_i$ belongs to $C(I \times I)$ and $K_i(\theta, \theta) \neq 0$ for all $\theta \in I$. For any $0<\kappa<1-\mu$, we will prove that $\mathcal{K}_i$ are linear operators mapping from $C(I)$ to $C^{0, \kappa}(I)$.
\end{Lemma}

\begin{Lemma}\label{lemma_J6org}
When $0<\kappa<1-\mu$, for any function $v \in C(I)$ and $K_i \in C(I \times I)$ with $K_i(\cdot, \eta) \in C^{0, \kappa}(I),(i=1,2)$, there is
	$$
	\begin{aligned}
		\notag \frac{|(\mathcal{K}_i v)(\theta_1)-(\mathcal{K}_i v)(\theta_2)|}{|\theta_1-\theta_2|^{\kappa}} \leq C \max _{\theta_1 \in I}|v(\theta_1)|, \quad \forall \theta_1 ,\theta_2 \in I, \theta_1 \neq \theta_2 .
	\end{aligned}
	$$
	Thus it can be inferred that
	$$
	\begin{aligned}
		\|\mathcal{K}_i v\|_{0, \kappa} \leq C\|v\|_{\infty}, \quad 0<\kappa<1-\mu .\notag
	\end{aligned}
	$$
\begin{proof}
For the sake of discussion, we will only prove the case of $\mathcal{K}_2 v$, as the case of $\mathcal{K}_1 v$ is analogous. Without loss of generality, we can assume that $0 \leq \theta_2 < \theta_1 \leq 1$. Let $K_2$ denote the maximum value over $\eta \in I$. We define $C$ as a constant and denote the Beta function as $B(\cdot, \cdot)$. The same symbol will be used in the subsequent numerical experiments.
	\begin{align}
		&\frac{\left|\left(\mathcal{K}_2 v\right)(\theta_1)-\left(\mathcal{K}_2 v\right)(\theta_2)\right|}{|\theta_1-\theta_2|^{\kappa}}\notag\\
		&=\varepsilon^{-\gamma}\frac{\left|\int_0^{\varepsilon\theta_2} \frac{(\varepsilon \theta_2-\eta)^{-\mu}}{\theta_2^{\gamma}} \eta^{\mu+\gamma-1} K_2(\theta_2, \eta) v(\eta) \, d\eta \right.
			-\left.\int_0^{\varepsilon\theta_1} \frac{(\varepsilon \theta_1-\eta)^{-\mu}}{\theta_1^{\gamma}} \eta^{\mu+\gamma-1} K_2(\theta_1, \eta) v(\eta) \, d\eta\right| }{|\theta_1-\theta_2|^{\kappa}}\notag \\
		&\leq\|v\|_{\infty}|\theta_1-\theta_2|^{-\kappa}\varepsilon^{-\gamma}\left| \int_0^{\varepsilon\theta_2}\frac{(\varepsilon \theta_2-\eta)^{-\mu}}{\theta_2^{\gamma}} \eta^{\mu+\gamma-1} K_2(\theta_2, \eta) \right.-\left.\frac{(\varepsilon \theta_1-\eta)^{-\mu}}{\theta_1^{\gamma}} \eta^{\mu+\gamma-1} K_2(\theta_1, \eta)\, d\eta\right|  \notag \\ 
		&+\|v\|_{\infty}|\theta_1-\theta_2|^{-\kappa}\varepsilon^{-\gamma}\int_{\varepsilon\theta_2}^{\varepsilon\theta_1} \frac{(\varepsilon \theta_1-\eta)^{-\mu}}{\theta_1^{\gamma}} \eta^{\mu+\gamma-1}\left|K_2(\theta_1, \eta)\right| \, d\eta \notag \\
		&=\varepsilon^{-\gamma}(M_1+M_2) \notag \\
		&\leq \varepsilon^{-\gamma}(M_1^{(1)}+M_2^{(1)}+M_2).
	\end{align}
where
	\begin{align}
		&M_1^{(1)}=\|v\|_{\infty}\left(\theta_1-\theta_2\right)^{-\kappa}
		\left| 
		\int_0^{\varepsilon\theta_2}\left(\frac{(\varepsilon \theta_2-\eta)^{-\mu}}{\theta_2^{\gamma}} \eta^{\mu+\gamma-1}\right.
		-\left.\frac{(\varepsilon \theta_1-\eta)^{-\mu}}{\theta_1^{\gamma}} \eta^{\mu+\gamma-1}\right) K_2(\theta_2, \eta)\, d\eta
		\right|,\notag\\
		&M_1^{(2)}=\|v\|_{\infty}\left(\theta_1-\theta_2\right)^{-\kappa}
		\int_0^{\varepsilon\theta_2}
		\frac{\left(\varepsilon \theta_1-\eta\right)^{-\mu}}{\theta_1^{\gamma}}\eta^{\mu+\gamma-1}
		\left|
		K_2\left(\theta_2,\eta\right)-K_2\left(\theta_1,\eta\right)
		\right|
		d\eta,\notag\\
		&M_{2}=\|v\|_{\infty}\left(\theta_1-\theta_2\right)^{-\kappa}\int_{\varepsilon\theta_2}^{\varepsilon\theta_1}\frac{\left(\varepsilon \theta_1-\eta\right)^{-\mu}}{\theta_1^{\gamma}} \eta^{\mu+\gamma-1}\left|K_2\left(\theta_1, \eta\right)\right| d \eta\notag.
	\end{align}
	Next, we analyze $M_1^{(1)}, M_1^{(2)}, M_{2}$ individually, that is
	\begin{align}
		M_1^{(1)} 
		&\leq K_2\|v\|_{\infty}\left(\theta_1-\theta_2\right)^{-\kappa}\left| \int_{0}^{\varepsilon\theta_2}\frac{\left(\varepsilon \theta_2-\eta\right)^{-\mu}}{\theta_2^{\gamma}} \eta^{\mu+\gamma-1} \, d\eta \right. -\left. \int_{0}^{\varepsilon\theta_1}\frac{\left(\varepsilon \theta_1-\eta\right)^{-\mu}}{\theta_1^{\gamma}} \eta^{\mu+\gamma-1} \, d\eta \right. \notag\\
		&\quad + \left. \int_{\varepsilon\theta_2}^{\varepsilon\theta_1}\frac{\left(\varepsilon \theta_1-\eta\right)^{-\mu}}{\theta_1^{\gamma}} \eta^{\mu+\gamma-1} \, d\eta \right| \notag\\
		&\leq C\|v\|_{\infty}\left(\theta_1-\theta_2\right)^{-\kappa}\int_{\varepsilon\theta_2}^{\varepsilon\theta_1}{\left(\varepsilon \theta_1-\eta\right)^{-\mu}}\frac{1}{\theta_1^{\gamma}}\eta^{\mu+\gamma-1} \, d\eta \notag\\
		&\leq C\|v\|_{\infty}\left(\theta_1-\theta_2\right)^{-\kappa}\varepsilon^{\gamma}\left(\theta_1-\theta_2\right)^{1-\mu}
		\left.\int_{0}^{1}\left(1-\xi\right)^{-\mu}\frac{1}{\theta_1^{\gamma}}\left(\theta_2+\left(\theta_1-\theta_2\right)\xi\right)^{\mu+\gamma-1}\, d \xi\right.\notag
	\end{align}
	\begin{align}
		&\qquad\leq C\|v\|_{\infty}\left(\theta_1-\theta_2\right)^{-\kappa}\varepsilon^{\gamma}\left(\theta_1-\theta_2\right)^{1-\mu}
		\left.\int_{0}^{1}\left(1-\xi\right)^{-\mu}\frac{1}{\theta_1^{\gamma}}\theta_1^{\mu+\gamma-1}\, d \xi\right.\notag\\
		&\qquad\leq C\|v\|_{\infty}\left(\theta_1-\theta_2\right)^{1-\mu-\kappa}\varepsilon^{\gamma}
		\left.\int_{0}^{1}\left(1-\xi\right)^{-\mu}\frac{1}{\theta_1^{\mu+\gamma-1}}\theta_1^{\mu+\gamma-1}d \xi\right.\notag\\
		&\qquad\leq C\|v\|_{\infty}\varepsilon^{\gamma}B(1-\mu,1)\notag\\
		&\qquad\leq C\varepsilon^{\gamma}\|v\|_{\infty},\label{eq_M1(1)}
	\end{align}
where $\mu+\gamma-1 \leq \gamma$. When \eqref{eq_C_rk} satisfies $r=0$, it can be inferred similarly
	\begin{align}
		M_1^{(2)}=&\|v\|_{\infty}
		\int_0^{\varepsilon\theta_2}
		\frac{\left(\varepsilon \theta_1-\eta\right)^{-\mu}}{\theta_1^{\gamma}}\eta^{\mu+\gamma-1}
		\frac{\left|
			K_2\left(\theta_2,\eta\right)-K_2\left(\theta_1,\eta\right)
			\right|}{\left(\theta_1-\theta_2\right)^{\kappa}}
		d\eta\notag\\
		&\leq \underset{\eta\in I}{\max}\|K_2(\cdot,\eta)\|_{0,\kappa}\|v\|_{\infty}
		\int_0^{\varepsilon\theta_1}
		\frac{\left(\varepsilon \theta_1-\eta\right)^{-\mu}}{\theta_1^{\gamma}}\eta^{\mu+\gamma-1}
		d\eta\notag\\
		&\leq \underset{\eta\in I}{\max}\|K_2(\cdot,\eta)\|_{0,\kappa}\|v\|_{\infty}\varepsilon^{\gamma}B(1-\mu,\mu+\gamma)\notag\\
		&\leq C\varepsilon^{\gamma}\|v\|_{\infty}.\label{eq_M1(2)} 
	\end{align}
It follows from the proof process for $M_1^{(1)}$
	\begin{align}
		M_2&\leq C\varepsilon^{\gamma}\|v\|_{\infty}B(1-\mu,\mu)\notag\\
		&\leq C\varepsilon^{\gamma}\|v\|_{\infty}.\label{eq_M2}
	\end{align}
It is clear from \eqref{eq_M1(1)}-\eqref{eq_M2} that \eqref{eq_C_rk} holds, and it is obvious that $\mathcal{K}_i$ are linear operators. Then we complete the proof.
\end{proof}
\end{Lemma}

\begin{Lemma}\label{lemma_J6}
While $0<\kappa<1-\mu$, for any $v(\theta) \in C(I), K_i\in C(I \times I)$, and $K_i(\cdot, \eta) \in C^{0, \kappa}(I),i=1,2$, there exists a constant $C$ such that
	\begin{align}\label{eq_J6_1}
		\frac{\left|(\mathcal{K}_i v)\left(\theta_1^{\frac{1}{\lambda}}\right)-(\mathcal{K}_i v)\left(\theta_2^{\frac{1}{\lambda}}\right)\right|}{|\theta_1-\theta_2|^\kappa} \leq C \max _{\theta_1 \in I}|v(\theta_1)|, \quad \forall \theta_1, \theta_2 \in I, \theta_1\neq \theta_2 .
	\end{align}
which is equivalent with 
	\begin{align}\label{eq_J6}
		\left\|(\mathcal{K}_i v)\left(\theta_1^{\frac{1}{\lambda}}\right)\right\|_{0, \kappa} \leq C\|v\|_{\infty}.
	\end{align}
\begin{proof}
From Lemma \ref{lemma_J6org}, we can obtain
	\begin{align}\label{eq_J6_2}
		\frac{\left|(\mathcal{K}_i v)\left(\theta_1^{\frac{1}{\lambda}}\right)-(\mathcal{K}_i v)\left(\theta_2^{\frac{1}{\lambda}}\right)\right|}{|\theta_1^{\frac{1}{\lambda}}-\theta_2^{\frac{1}{\lambda}}|^\kappa} \leq C \max _{\theta_1 \in I}|v(\theta_1)|, \quad \forall \theta_1, \theta_2 \in I, \theta_1\neq \theta_2,
	\end{align}
	
When $0<\lambda \leq 1$, we have 
	\begin{align}
		\left|\theta_1^{\frac{1}{\lambda}}-\theta_2^{\frac{1}{\lambda}}\right|^\kappa \sim O\left(|\theta_1-\theta_2|^\kappa\right) \text { or }\left|\theta_1^{\frac{1}{\lambda}}-\theta_2^{\frac{1}{\lambda}}\right|^\kappa \sim o\left(|\theta_1-\theta_2|^\kappa\right) \text {,}\notag
	\end{align}
	
Then it can be derived that
	\begin{align}\label{eq_J6_3}
		\frac{\left|(\mathcal{K}_i v)\left(\theta_1^{\frac{1}{\lambda}}\right)-(\mathcal{K}_i v)\left(\theta_2^{\frac{1}{\lambda}}\right)\right|}{|\theta_1-\theta_2|^\kappa}
		\leq C \frac{\left|(\mathcal{K}_i v)\left(\theta_1^{\frac{1}{\lambda}}\right)-(\mathcal{K}_i v)\left(\theta_2^{\frac{1}{\lambda}}\right)\right|}{|\theta_1^{\frac{1}{\lambda}}-\theta_2^{\frac{1}{\lambda}}|^\kappa}.
	\end{align}
	
Combing \eqref{eq_J6_1},\eqref{eq_J6_2},\eqref{eq_J6_3},we verify \eqref{eq_J6} and finish the proof.
\end{proof}
\end{Lemma}

\begin{Lemma}\label{lemma_Hardy}$($$\cite{spectal_TangTao}$ Generalized Hardy's inequality$)$ For all measurable functions $g \geq 0$, weight functions $x$ and $y$, $1 < p \leq q < \infty
	$, the generalized Hardy's inequality can be expressed as follows:
	$$
	\left(\int_a^b\left|\left(\mathcal{M} g\right)(t)\right|^q x(t) d t\right)^{1 / q} \leq C\left(\int_a^b\left|g(x)\right|^p y(t) d t\right)^{1 / p}
	$$ 
if and only if 
	$$
	\sup _{a<t<b}\left(\int_t^b x(s) d s\right)^{1 / q}\left(\int_a^t y^{1-p_{0}}(s) d s\right)^{1 / p_{0}}<\infty, \quad p_{0}=\frac{p}{p-1},
	$$\\
where the operator $\mathcal{M}$ is defined as 
	$$
	(\mathcal{M} g)(t)=\int_a^t \rho(x, \tau) g(\tau) d \tau.
	$$
with $\rho(x, s)$ being a given kernel,  $-\infty \leq a < b \leq \infty$, .
 \end{Lemma}

In order to prove the interpolation error estimate in the $L^{\infty}$-norm and $L_{\omega^{\alpha,\beta,\lambda}}^{2}$-norm, we should introduce the following lemmas:
\begin{Lemma}\label{lemma_Interpolation_infinity}$($see $\cite{Muntz} $$)$
$I_{N, \lambda}^{\alpha, \beta}$ is the interpolation operator of fractional $Jacobi$ polynomials, when $-1<\alpha, \beta \leq-\frac{1}{2}$, for any $0 \leq l \leq m \leq N+1$, we have 
	$$
	\begin{aligned}\label{eq_Interpolation_infinity}
	\left\|v-I_{N, \lambda}^{\alpha, \beta} v\right\|_{\infty} \leq C N^{1 / 2-m}\left\|\partial_\theta^m v\left(\theta^{\frac{1}{\lambda}}\right)\right\|_{0, \omega^{\alpha+m, \beta+m, 1}},\quad \forall v\left(\theta^{\frac{1}{\lambda}}\right) \in B_{\alpha, \beta}^{m, 1}(I), m \geq 1.
	\end{aligned}
	$$ 
\end{Lemma}

\begin{Lemma}\label{lemma_Interpolation_L^2}$($see $\cite{Muntz} $$)$
It holds for $\forall v\left(\theta^{\frac{1}{\lambda}}\right) \in B_{\alpha, \beta}^{m, 1}(I), m \geq 1$, then
	\begin{align}\label{eq_Interpolation_L^2}
		\left\|v-I_{N, \lambda}^{\alpha, \beta} v\right\|_{0,\omega^{\alpha,\beta,1}} \leq C N^{-m}\left\|\partial_\theta^m v\left(\theta^{\frac{1}{\lambda}}\right)\right\|_{0, \omega^{\alpha+m, \beta+m, 1}}.
	\end{align}
\end{Lemma}

\begin{Lemma}\label{lemma_Lesbegue_constant}$($see $\cite{Muntz} $$)$
Let $\left\{F_{j,\lambda}\left(\theta\right)\right\}_{j=0}^N$ be the generalized Lagrange interpolation basis functions
associated with the Gauss points of the fractional $Jacobi$ polynomials $J_{N+1}^{\alpha, \beta, \lambda}(x)$. $Lesbegue$ constant  is presented as 
	\begin{align}\label{eq_Lesbegue_constant}
		\left\|I_{N, \lambda}^{\alpha, \beta}\right\|_{\infty}:=\max _{x \in I} \sum_{i=0}^N\left|F_{j,\lambda}\left(\theta\right)\right|= \begin{cases}O(\log N), & -1<\alpha, \beta \leq-\frac{1}{2}, \\ O\left(N^{\gamma+\frac{1}{2}}\right), & \gamma=\max (\alpha, \beta), \text {otherwise}.
		\end{cases}
	\end{align}
\end{Lemma}

\begin{Lemma}\label{lemma_Ii1Ii2}$($see $\cite{Muntz} $$)$
For $\forall v \in B_{\alpha, \beta}^{m, 1}(I), m \geq 1$,$\forall \phi \in \mathcal{P}_N^1(I)$, there are some conclusions as follows
	\begin{align}\label{eq_con&dis}
		\left|(v, \phi)_{\omega^{\alpha, \beta, 1}}-(v, \phi)_{N, \omega^{\alpha, \beta, 1}}\right| \leq C N^{-m}\left\|\partial_\theta^m v\right\|_{0, \omega^{\alpha+m, \beta+m, 1}}\|\phi\|_{0, \omega^{\alpha, \beta, 1}},
	\end{align}
where
	\begin{align}
		&(v, \phi)_{\omega^{\alpha, \beta, 1}}=\int_0^1v(\theta)\phi(\theta)\theta^{\alpha}(1-\theta)^{\beta}d\theta\label{eq_inner_continue},\\
		&(v, \phi)_{N, \omega^{\alpha, \beta, 1}}=\sum_{k=0}^{N}v(\theta_k)\phi(\theta_k)\omega_k.\label{eq_inner_discrete}
	\end{align}
\end{Lemma}

\begin{Lemma}\label{lemma_E6E7_L^2}$($see $\cite{Muntz} $$)$
For any bounded function $v(\theta)$ defined on $I$, there exists a constant $C$ independent of $v$ such that
	\begin{align}\label{eq_J1J2_L^2}
		\qquad\qquad\qquad\qquad\qquad\qquad
		\sup _N\left\|I_{N, \lambda}^{\alpha, \beta} v\right\|_{0,\omega^{\alpha, \beta, \lambda}} \leq C\|v\|_{\infty} .
	\end{align}	
\end{Lemma}

\section{Numerical Scheme} \label{section_Numerical Scheme}
For the given positive integer $N$, we will introduce fractional $Jacobi-Gauss$ quadrature nodes as collocation points, denoted by $\lbrace \theta_j \rbrace_{j=0}^{N}$, with corresponding weights $\lbrace \omega_j \rbrace_{j=0}^{N}$. The form of the weight function is given in \eqref{eq_wight_equation},with the space $P_N^1$ defined in the previous section. For the generalized interpolation operator $I_{N, \lambda}^{\alpha, \beta}$ defined in \eqref{eq_Interpolation_operator_sum}, we have
	$$
	\begin{aligned}
		I_{N, \lambda}^{\alpha, \beta}v\left(\theta_i\right)=v\left(\theta_i\right),0 \leq i \leq N,\notag
	\end{aligned}
	$$
Via the change of variable $\tau=\varepsilon s$, \eqref{eq_prime} and \eqref{eq_origin} can be rewritten as
	\begin{align}
		&y^{\prime}(t)= p_1{(t)} y(t)+q_1{(t)} y(\varepsilon t)+g_1{(t)}+t^{-\gamma} \int_0^t(t-s)^{-\mu} 	s^{\mu+\gamma-1} K_1(t, s) y(s) d s \notag\\ 
		&\qquad+t^{-\gamma} \int_0^t(t-s)^{-\mu}s^{\mu+\gamma-1} K_2(t, \varepsilon s) y(\varepsilon s) d s,\\
		&y\left(0\right)=y_0.
	\end{align}
For the sake of the theory of orthogonal polynomials, we make the transformation $t=T\theta$. Meanwhile, a linear transformation $s=T\eta$ will be happened in order to transfer the integral interval $[0,T\theta]$ to $[0,\theta]$, then they become
	\begin{align}
		&\varphi^{\prime}(\theta)=\tilde{p}_1(\theta) \varphi(\theta)+\tilde{q}_1(\theta)\varphi(\varepsilon \theta)+\tilde{g}_1(\theta)+\theta^{-\gamma} \int_0^\theta(\theta-\eta)^{-\mu} \eta^{\mu+\gamma-1} \bar{K}_1(\theta, \eta) \varphi(\eta) d \eta \notag\\
		&\quad\qquad+\theta^{-\gamma} \int_0^\theta(\theta-\eta)^{-\mu}\eta^{\mu+\gamma-1} \bar{K}_2(\theta, \varepsilon \eta)\varphi(\varepsilon \eta) d \eta\label{eq_phi_prime}\\
		&\varphi(0)=\varphi_0=y_0\label{eq_phi_origin}.
	\end{align}
where
	\begin{align}
		&\tilde{p}_1(\theta)=Tp_1(T\theta),\tilde{q}_1(\theta)=Tq_1(T\theta),\tilde{g}_1(\theta)=Tg_1(T\theta), \varphi(\theta)=y(T\theta),\notag\\
		&\bar{K}_1(\theta, \eta)=TK_1(T\theta,T\eta),\bar{K}_2(\theta, \varepsilon \eta)=TK_2(T\theta, T\varepsilon \eta).\notag
	\end{align}
\eqref{eq_phi_origin} can be turned into
\begin{align}
	&\varphi(\theta)=\varphi_0+\int_0^{\theta}\varphi^{\prime}(\eta)d\eta.\label{eq_initial_integral}
\end{align}
Firstly, \eqref{eq_phi_prime} and \eqref{eq_initial_integral} hold at the collocation points $\lbrace \theta_i \rbrace_{i=0}^{N}$:
\begin{align}
	&\varphi^{\prime}(\theta_i)=\tilde{p}_1(\theta_i) \varphi(\theta_i)+\tilde{g}_1(\theta_i)\varphi(\varepsilon \theta_i)+\tilde{g}_1(\theta_i)+\theta_i^{-\gamma} \int_0^{\theta_i}(\theta_i-\eta)^{-\mu} \eta^{\mu+\gamma-1} \bar{K}_1(\theta_i, \eta) \varphi(\eta) d \eta ,\notag\\
	&\quad\qquad+\theta_i^{-\gamma} \int_0^{\theta_i}(\theta_i-\eta)^{-\mu}\eta^{\mu+\gamma-1} \bar{K}_2(\theta_i, \varepsilon \eta)(\varepsilon \eta) d \eta,\label{eq_thetai_collocation_prime}\\
	&\varphi(\theta_i)=\varphi_0+\int_0^{\theta_i}\varphi^{\prime}(\eta)d\eta,\label{eq_thetai_collocation_origin}\\
	&\varphi(\varepsilon\theta_i)=\varphi_0+\varepsilon\int_0^{\theta_i}\varphi^{\prime}(\varepsilon\eta)d\eta.\label{eq_thetai_collocation_eorigin}
\end{align}
Let $\eta=\theta_i\xi^{\frac{1}{\lambda}}= \eta_i(\xi)$ and transfer the integration interval from $[0,\theta_i]$ to a fixed interval $[0,1]$,
	\begin{align}
		&\varphi^{\prime}(\theta_i)=\tilde{p}_1(\theta_i) \varphi(\theta_i)+\tilde{q}_1(\theta_i)\varphi(\varepsilon \theta_i)+\tilde{g}_1(\theta_i)+
		\int_0^1(1-\xi)^{-\mu}\xi^{\frac{\mu+\gamma}{\lambda}-1}\widetilde{K}_1(\theta_i, \eta_i(\xi))\varphi( \eta_i(\xi))d\xi\notag\\
		&\quad\qquad+	\int_0^1(1-\xi)^{-\mu}\xi^{\frac{\mu+\gamma}{\lambda}-1}\widetilde{K}_2(\theta_i,\varepsilon \eta_i(\xi))\varphi(\varepsilon \eta_i(\xi))d\xi,\label{eq_collocation_prime}\\
		&\varphi\left(\theta_i\right)=\varphi_0+\int_0^{\theta_i} \varphi^{\prime}(\eta) d\eta\notag \\
		&\quad\qquad=\varphi_0+\frac{\theta_i}{\lambda} \int_0^1 \xi^{\frac{1}{\lambda}-1} \varphi^{\prime}\left( \eta_i(\xi)\right) d \xi \label{eq_collocation_origin},\\
		&\varphi\left(\varepsilon\theta_i\right)=\varphi_0+\frac{\varepsilon\theta_i}{\lambda} \int_0^1 \xi^{\frac{1}{\lambda}-1} \varphi^{\prime}\left( \varepsilon\eta_i(\xi)\right) d \xi .\label{eq_collocation_eorigin}
	\end{align}\\
where we set
	\begin{align}
		& \widetilde{K}_1\left(\theta_i,  \eta_i(\xi)\right)=\frac{1}{\lambda} \frac{\left(1-\xi^{\frac{1}{\lambda}}\right)^{-\mu}}{(1-\xi)^{-\mu}}\bar{K}_1\left(\theta_i,  \eta_i(\xi)\right),\notag\\
		& \widetilde{K}_2\left(\theta_i, \varepsilon \eta_i(\xi)\right)=\frac{1}{\lambda} \frac{\left(1-\xi^{\frac{1}{\lambda}}\right)^{-\mu}}{(1-\xi)^{-\mu}}\bar{K}_2\left(\theta_i,\varepsilon \eta_i(\xi)\right).\notag
	\end{align}
	
For the two integral terms in \eqref{eq_collocation_prime}, apply the $(N+1)$-point fractional $Jacobi-Gauss$ quadrature formula to estimate them, denoted by the nodes $\lbrace \xi_k\rbrace_{k=0}^{N}$ and the weights $\lbrace \omega_k\rbrace_{k=0}^{N}$, where $\alpha=-\mu,\beta=\frac{\mu+\gamma}{\lambda}-1$ are the related parameters. 
For \eqref{eq_collocation_origin} and \eqref{eq_collocation_eorigin}, denote $\lbrace \hat{\xi}_k\rbrace_{k=0}^{N}$ and $\lbrace \hat{\omega}_k\rbrace_{k=0}^{N}$ as the fractional $Jacobi-Gauss$ quadrature nodes and weights with the parameters $\alpha=0,\beta=\frac{1}{\lambda}-1$ respectively, we will obtain
	\begin{align}
	&\int_0^1(1-\xi)^{-\mu}\xi^{\frac{\mu+\gamma}{\lambda}-1}\widetilde{K}_1(\theta_i, \eta_i(\xi))\varphi( \eta_i(\xi))d\xi \approx \sum_{k=0}^N \widetilde{K}_1\left(\theta_i,  \eta_i\left(\xi_k\right)\right) \varphi\left( \eta_i\left(\xi_k\right)\right) \omega_k\label{eq_discrete_K1},\\
	&\int_0^1(1-\xi)^{-\mu}\xi^{\frac{\mu+\gamma}{\lambda}-1}\widetilde{K}_2(\theta_i,\varepsilon \eta_i(\xi))\varphi(\varepsilon \eta_i(\xi))d\xi \approx \sum_{k=0}^N \widetilde{K}_2\left(\theta_i, \varepsilon  \eta_i\left(\xi_k\right)\right) \varphi\left(\varepsilon  \eta_i\left(\xi_k\right)\right) \omega_k\label{eq_discrete_K2},\\
	&\frac{\theta_i}{\lambda} \int_0^1 \xi^{\frac{1}{\lambda}-1} \varphi^{\prime}\left( \eta_i(\xi)\right) d \xi\approx
	\sum_{k=0}^{N}\frac{\theta_i}{\lambda}\varphi^{\prime}\left( \eta_i(\hat{\xi}_k)\right)\hat{\omega}_k\label{eq_discrete_origin},\\
	&\frac{\varepsilon\theta_i}{\lambda} \int_0^1 \xi^{\frac{1}{\lambda}-1} \varphi^{\prime}\left( \varepsilon\eta_i(\xi)\right) d \xi\approx
	\sum_{k=0}^{N}\frac{\varepsilon\theta_i}{\lambda}\varphi^{\prime}\left(\varepsilon \eta_i(\hat{\xi}_k)\right)\hat{\omega}_k\label{eq_discrete_eorigin}.
	\end{align}
	
Combining the fractional $Jacobi$ collocation method, we can seek $\varphi_{N}\left(\theta\right)$ and $\varphi_{N}^{\ast}\left(\theta\right)$ such that $\lbrace \varphi_i\rbrace_{i=0}^{N},\lbrace \varphi_i^{\ast}\rbrace_{i=0}^{N}$ and $\lbrace v_i\rbrace_{i=0}^{N}$ satisfy the following collocation equations
	\begin{align}
		&\varphi_i^{\ast}=\tilde{p}_1\left(\theta_i\right) \varphi_i+\tilde{q}_1\left(\theta_i\right) v_i  +\tilde{g}_1(\theta_i)+\sum_{j=0}^N \varphi_j \left(\sum_{k=0}^N \widetilde{K}_1\left(\theta_i,  \eta_i\left(\xi_k\right)\right) F_{j,\lambda}\left( \eta_i\left(\xi_k\right)\right) \omega_k\right)\notag\\
		& \qquad+\sum_{j=0}^N \varphi_j\left(\sum_{k=0}^N \widetilde{K}_2\left(\theta_i, \varepsilon \eta_i\left(\xi_k\right)\right) F_{j,\lambda}\left(\varepsilon \eta_i\left(\xi_k\right)\right) \omega_k\right),\label{eq_collocation*_prime}\\
		&\varphi_i=\varphi_0+\sum_{j=0}^N  \varphi_j^*\left(\sum_{k=0}^N\frac{\theta_i}{\lambda} F_{j,\lambda}\left( \eta_i\left(\hat{\xi}_k\right)\right) \hat{\omega}_k\right),\label{eq_collocation*_origin}\\
		&v_i=\varphi_0+ \sum_{j=0}^N\varphi_j^*\left(\sum_{k=0}^N\frac{\varepsilon\theta_i}{\lambda}F_{j,\lambda}\left(\varepsilon \eta_i\left(\hat{\xi}_k\right)\right) \hat{\omega}_k\right)\label{eq_collocation*_eorigin}.
	\end{align}
After defining $\varphi_i,\varphi_i^{\ast}$ and $v_i$ to approxiamate the funtion value $\varphi\left(\theta_i\right),\varphi^{\prime}(\theta_i)$ and $\varphi\left(\varepsilon\theta_i\right)$ respectively, $\varphi\left(\theta\right),\varphi^{\prime}(\theta)$ are approxiated by the generalized $Lagrange$ interpolation polynomials according to \eqref{eq_Interpolation_operator_sum}, so we denote them as:
\begin{align}
	\varphi^{\prime}\left(\theta\right)\approx \varphi_{N}^{\ast}\left(\theta\right)=\sum_{j=0}^{N}\varphi_j^{\ast}F_{j,\lambda}(\theta),
	\varphi\left(\theta\right)\approx \varphi_{N}\left(\theta\right)=\sum_{j=0}^{N}\varphi_jF_{j,\lambda}(\theta).
\end{align}
where $\varphi_{N}^{\ast}\left(\theta\right)$ is not the exact derive of $\varphi_{N}\left(\theta\right)$.\\
We define $U_N^{\ast}=\lbrace\varphi_0^{\ast},\varphi_1^{\ast},\cdots,\varphi_N^{\ast}\rbrace^{T},
U_N=\lbrace\varphi_0,\varphi_1,\cdots,\varphi_N\rbrace^{T},
V_N=\lbrace v_0,v_1,\cdots,v_N\rbrace^{T}
$ so as to obtain the matrix form of \eqref{eq_collocation*_prime}-\eqref{eq_collocation*_eorigin}
	\begin{align}
		&U_N^{\ast}=(P+C+D)U_N+QV_N+G\label{eq_collocation*_matrix_prime},\\
		&U_N=U_0+EU_N^{\ast}\label{eq_collocation*_matrix_origin},\\
		&V_N=U_0+HU_N^{\ast}\label{eq_collocation*_matrix_eorigin}.
	\end{align}
where $P=diag\lbrace\tilde{p}_1\left(\theta_0\right),\tilde{p}_1\left(\theta_1\right),\cdots,\tilde{p}_1\left(\theta_N\right)\rbrace, Q=diag\lbrace\tilde{q}_1\left(\theta_0\right),\tilde{q}_1\left(\theta_1\right),\cdots,\tilde{q}_1\left(\theta_N\right)\rbrace,G=\\ \lbrace\tilde{g}_1\left(\theta_0\right),\tilde{g}_1\left(\theta_1\right),\cdots,\tilde{g}_1\left(\theta_N\right)\rbrace^{T},U_0=\lbrace\varphi_0,\varphi_0,\cdots,\varphi_0\rbrace^{T}$. The elements of $C,D,E,H$ are as follows:
	\begin{align}
		&C_{ij}=\sum_{k=0}^N \widetilde{K}_1\left(\theta_i,  \eta_i\left(\xi_k\right)\right) F_{j,\lambda}\left( \eta_i\left(\xi_k\right)\right) \omega_k,\notag\\
		&D_{ij}=\sum_{k=0}^N \widetilde{K}_2\left(\theta_i, \varepsilon \eta_i\left(\xi_k\right)\right) F_{j,\lambda}\left(\varepsilon \eta_i\left(\xi_k\right)\right)\omega_k,\notag\\
		&E_{ij}=\sum_{k=0}^N\frac{\theta_i}{\lambda} F_{j,\lambda}\left( \eta_i\left(\hat{\xi}_k\right)\right) \hat{\omega}_k,\notag\\
		&H_{ij}=\sum_{k=0}^N \frac{\varepsilon\theta_i}{\lambda}F_{j,\lambda}\left(\varepsilon \eta_i\left(\hat{\xi}_k\right)\right) \hat{\omega}_k.\notag
	\end{align}
The values of  $\lbrace \varphi_i\rbrace_{i=0}^{N}$ and $\lbrace \varphi_i^{\ast}\rbrace_{i=0}^{N}$ can be derived by solving the system of \eqref{eq_collocation*_matrix_prime}-\eqref{eq_collocation*_matrix_eorigin}, then we can get the numerical solutions accordingly.
\begin{remark}\label{remark_integral_equal}
	Since $F_{j,\lambda}(\theta)(j=0,1,\cdots,N)$ are fractional Jacobi polynomials of degree not exceeding $N$, $F_{j,\lambda}( \eta_i(\xi))$ are $j$-th Jacobi polynomials with respect to $\xi$ due to the definition of $F_{j,\lambda}( \theta)$. Thus, there is a relationship between the following integral and quadrature formula:
	\begin{align}
		\int_0^{\theta_i} \varphi_{N}^{\ast}(\eta) d\eta
		=&\int_0^{\theta_i}\sum_{j=0}^N\varphi_j^{\ast}F_{j,\lambda}(\eta)d\eta
		=\frac{\theta_i}{\lambda}\int_0^{1}\xi^{\frac{1}{\lambda}-1}\sum_{j=0}^N\varphi_j^{\ast}F_{j,\lambda}\left( \eta_i(\xi)\right)d\xi\notag\\
		=&\sum_{j=0}^N  \varphi_j^*\left(\sum_{k=0}^N\frac{\theta_i}{\lambda} F_{j,\lambda}\left( \eta_i\left(\hat{\xi}_k\right)\right) \hat{\omega}_k\right)
		=\left(\frac{\theta_i}{\lambda}, \varphi_N^{\ast}\left( \eta_i(\cdot)\right)\right) _{N,\omega^{0, \frac{1}{\lambda}-1,1}},\label{eq_integral_equal}
	\end{align}
that is also because of $\varphi_N\left( \eta_i\left(\xi\right)\right)=\sum_{j=0}^{N}{\varphi_j}^{\ast}F_{j,\lambda}\left( \eta_i\left(\xi\right)\right)\in \mathcal{P}_N^1$.  The case with $\varepsilon$ can be similarly discussed.
\end{remark}

\section{Convergence Analysis}\label{section_Convergence Analysis}
In this section, we will continue to analyse the numerical scheme in Section \ref{section_Numerical Scheme} in order to derive the convergence results more intuitively. Then exponential convergence will be rigorously proved,  i.e., the spectral accuracy can be showed for the proposed approximations. \\
First of all, based on \eqref{eq_inner_continue}, rewrite \eqref{eq_collocation_prime}-\eqref{eq_collocation_eorigin} into the form of continuous inner products
		\begin{align}
		&\varphi^{\prime}(\theta_i)=\tilde{p}_1(\theta_i) \varphi(\theta_i)+\tilde{q}_1(\theta_i)\varphi(\varepsilon \theta_i)+\tilde{g}_1(\theta_i)+
		\left(\widetilde{K}_1(\theta_i, \eta_i(\cdot)), \varphi\left( \eta_i(\cdot)\right)\right)_ {\omega^{-\mu,\frac{\mu+\gamma}{\lambda}-1,1}}\notag\\
		&\qquad\quad+	\left(\widetilde{K}_2(\theta_i,\varepsilon \eta_i(\cdot)), \varphi\left(\varepsilon \eta_i(\cdot)\right)\right)_ {\omega^{-\mu,\frac{\mu+\gamma}{\lambda}-1,1}},\label{eq_inner_continue_prime}\\
		&\varphi\left(\theta_i\right)=\varphi_0+\left(\frac{\theta_i}{\lambda}, \varphi^{\prime}\left( \eta_i(\cdot)\right)\right) _{\omega^{0, \frac{1}{\lambda}-1,1}}\label{eq_inner_continue_origin},\\
		&\varphi\left(\varepsilon\theta_i\right)=\varphi_0+\left(\frac{\varepsilon\theta_i}{\lambda}, \varphi^{\prime}\left(\varepsilon \eta_i(\cdot)\right)\right) _{\omega^{0, \frac{1}{\lambda}-1,1}}.\label{eq_inner_continue_eorigin}
	\end{align}
Derived from \eqref{eq_inner_discrete}, \eqref{eq_collocation*_prime}-\eqref{eq_collocation*_eorigin} can be changed into discrete inner products
		\begin{align}
		&\varphi_i^{\ast}=\tilde{p}_1\left(\theta_i\right) \varphi_i+\tilde{q}_1\left(\theta_i\right) v_i  +\tilde{g}_1(\theta_i)+	\left(\widetilde{K}_1(\theta_i, \eta_i(\cdot)), \varphi_N\left( \eta_i(\cdot)\right)\right)_ {N,\omega^{-\mu,\frac{\mu+\gamma}{\lambda}-1,1}}\notag\\
		&\qquad+\left(\widetilde{K}_2(\theta_i,\varepsilon \eta_i(\cdot)), \varphi_N\left(\varepsilon \eta_i(\cdot)\right)\right)_ {N,\omega^{-\mu,\frac{\mu+\gamma}{\lambda}-1,1}},\label{eq_inner_discrete_prime}\\
		&\varphi_i=\varphi_0+\left(\frac{\theta_i}{\lambda}, \varphi_N^{\ast}\left( \eta_i(\cdot)\right)\right) _{N,\omega^{0, \frac{1}{\lambda}-1,1}},\label{eq_inner_discrete_origin}\\
		&v_i= \varphi_0+\left(\frac{\varepsilon\theta_i}{\lambda}, \varphi_{N}^{\ast}\left(\varepsilon \eta_i(\cdot)\right)\right) _{N,\omega^{0, \frac{1}{\lambda}-1,1}} .\label{eq_inner_discrete_eorigin}
	\end{align}
Add on both sides of the equation \eqref{eq_inner_discrete_prime} by 
	\begin{align}
		\theta_i^{-\gamma} \int_0^{\theta_i}(\theta_i-\eta)^{-\mu} \eta^{\mu+\gamma-1} \bar{K}_1(\theta_i, \eta) \varphi_N(\eta) d \eta=\left(\widetilde{K}_1(\theta_i, \eta_i(\cdot)), \varphi_N\left( \eta_i(\cdot)\right)\right)_ {\omega^{-\mu,\frac{\mu+\gamma}{\lambda}-1,1}},
	\end{align}
 and 
 \begin{align}
 	\theta_i^{-\gamma} \int_0^{\theta_i}(\theta_i-\eta)^{-\mu}\eta^{\mu+\gamma-1} \bar{K}_2(\theta_i, \varepsilon \eta)\varphi_N(\varepsilon \eta) d \eta=\left(\widetilde{K}_2(\theta_i,\varepsilon \eta_i(\cdot)), \varphi_N\left(\varepsilon \eta_i(\cdot)\right)\right)_ {\omega^{-\mu,\frac{\mu+\gamma}{\lambda}-1,1}},
 \end{align}
 as previously derived.
Because of Remark \ref{remark_integral_equal}, \eqref{eq_inner_discrete_origin} and \eqref{eq_inner_discrete_eorigin} with \eqref{eq_inner_discrete_prime} are equipped with
		\begin{align}
		&\varphi_i^{\ast}=\tilde{p}_1\left(\theta_i\right) \varphi_i+\tilde{q}_1\left(\theta_i\right) v_i
		+\tilde{g}_1(\theta_i)
		+\theta_i^{-\gamma} \int_0^{\theta_i}(\theta_i-\eta)^{-\mu} \eta^{\mu+\gamma-1} \bar{K}_1(\theta_i, \eta) \varphi_N(\eta) d \eta ,\notag\\
		&\qquad+\theta_i^{-\gamma} \int_0^{\theta_i}(\theta_i-\eta)^{-\mu}\eta^{\mu+\gamma-1} \bar{K}_2(\theta_i, \varepsilon \eta)\varphi_N(\varepsilon \eta) d \eta-I_{i,1}-I_{i,2},\label{eq_thetai_numerical_prime}\\
		&\varphi_i=\varphi_0+\int_0^{\theta_i} \varphi_{N}^{\ast}(\eta) d\eta,\label{eq_thetai_numerical_origin}\\
		&v_i=\varphi_0+\varepsilon\int_0^{\theta_i} \varphi_{N}^{\ast}(\varepsilon\eta) d\eta.\label{eq_thetai_numerical_eorigin}
	\end{align}
where we set
	\begin{align}
		&I_{i,1}=\left(\widetilde{K}_1(\theta_i, \eta_i(\cdot)), \varphi_N\left( \eta_i(\cdot)\right)\right)_ {\omega^{-\mu,\frac{\mu+\gamma}{\lambda}-1,1}}-\left(\widetilde{K}_1(\theta_i, \eta_i(\cdot)), \varphi_N\left( \eta_i(\cdot)\right)\right)_ {N,\omega^{-\mu,\frac{\mu+\gamma}{\lambda}-1,1}},\label{eq_Ii1}\\
		&I_{i,2}=\left(\widetilde{K}_2(\theta_i,\varepsilon \eta_i(\cdot)), \varphi_N\left(\varepsilon \eta_i(\cdot)\right)\right)_ {\omega^{-\mu,\frac{\mu+\gamma}{\lambda}-1,1}}-	\left(\widetilde{K}_2(\theta_i,\varepsilon \eta_i(\cdot)), \varphi_N\left(\varepsilon \eta_i(\cdot)\right)\right)_ {N,\omega^{-\mu,\frac{\mu+\gamma}{\lambda}-1,1}}.\label{eq_Ii2}
	\end{align}
They can be developed by Lemma \ref{lemma_Ii1Ii2} that
	\begin{align}
		& \left|I_{i, 1}\right| \leq C N^{-m}\left\|\partial_\theta^m \widetilde{K}_1\left(\theta_i,  \eta_i(\cdot)\right)\right\|_ {0,\omega^{m-\mu,{m+\frac{\mu+\gamma}{\lambda}}-1,1}}
		\left\|\varphi_N\left( \eta_i\left(\cdot\right)\right)\right\|_ {0,\omega^{-\mu,\frac{\mu+\gamma}{\lambda}-1,1}},\label{eq_E1_Ii1_absolute}\\
		& \left|I_{i, 2}\right| \leq C N^{-m}\left\|\partial_\theta^m \widetilde{K}_2\left(\theta_i, \varepsilon \eta_i(\cdot)\right)\right\|_ {0,\omega^{m-\mu,{m+\frac{\mu+\gamma}{\lambda}}-1,1}}
		\left\|\varphi_N\left(\varepsilon \eta_i\left(\cdot\right)\right)\right\|_ {0,\omega^{-\mu,\frac{\mu+\gamma}{\lambda}-1,1}}.\label{eq_E2_Ii2_absolute}
	\end{align}
We now subtract \eqref{eq_thetai_numerical_prime} from \eqref{eq_thetai_collocation_prime},  \eqref{eq_thetai_numerical_origin} from \eqref{eq_thetai_collocation_origin}, and \eqref{eq_thetai_numerical_eorigin} from \eqref{eq_thetai_collocation_eorigin} to deduce the error equations
	\begin{align}
		&\varphi^{\prime}(\theta_i)-\varphi_i^{\ast}=\tilde{p}_1(\theta_i) \left(\varphi(\theta_i)-\varphi_i\right)+\tilde{q}_1(\theta_i)\left(\varphi(\varepsilon \theta_i)-v_i\right)\notag\\
		&\qquad\qquad\qquad+\theta_i^{-\gamma} \int_0^{\theta_i}(\theta_i-\eta)^{-\mu} \eta^{\mu+\gamma-1} \bar{K}_1(\theta_i, \eta) e(\eta) d \eta\notag\\
		&\qquad\qquad\qquad+\theta_i^{-\gamma} \int_0^{\theta_i}(\theta_i-\eta)^{-\mu}\eta^{\mu+\gamma-1} \bar{K}_2(\theta_i, \varepsilon \eta)e(\varepsilon \eta) d \eta+I_{i,1}+I_{i,2},\label{eq_subtraction1_prime}\\
		&\varphi(\theta_i)-\varphi_i=\int_0^{\theta_i} e^{\ast}(\eta) d\eta,\label{eq_subtraction1_origin}\\
		&\varphi(\varepsilon \theta_i)-v_i=\varepsilon\int_0^{\theta_i} e^{\ast}(\varepsilon\eta) d\eta.\label{eq_subtraction1_eorigin}
	\end{align}
where $e(\theta)=\varphi(\theta)-\varphi_N(\theta),e^{\ast}(\theta)=\varphi^{\prime}(\theta)-\varphi_N^{\ast}(\theta)$. Substituting \eqref{eq_subtraction1_origin} and \eqref{eq_subtraction1_eorigin} into \eqref{eq_subtraction1_prime} yields
	\begin{align}
		\varphi^{\prime}(\theta_i)-\varphi_i^{\ast}=&\tilde{p}_1(\theta_i) \int_0^{\theta_i} e^{\ast}(\eta) d\eta+\varepsilon\tilde{q}_1(\theta_i)\int_0^{\theta_i} e^{\ast}(\varepsilon\eta) d\eta\notag\\
		&+\theta_i^{-\gamma} \int_0^{\theta_i}(\theta_i-\eta)^{-\mu} \eta^{\mu+\gamma-1} \bar{K}_1(\theta_i, \eta) e(\eta) d \eta\notag\\
		&+\theta_i^{-\gamma} \int_0^{\theta_i}(\theta_i-\eta)^{-\mu}\eta^{\mu+\gamma-1} \bar{K}_2(\theta_i, \varepsilon \eta)e(\varepsilon \eta) d \eta+I_{i,1}+I_{i,2}.\label{eq_subtraction2_prime}
	\end{align}
After multiplying $F_{i,\lambda}(\theta)$ on both sides of \eqref{eq_subtraction2_prime}, \eqref{eq_subtraction1_origin} and \eqref{eq_subtraction1_eorigin}, then summing up from $i=0$ to $N$, it changes into
	\begin{align}
		I_{N, \lambda}^{\alpha, \beta} \varphi^{\prime}(\theta)-\varphi_N^*(\theta)=
		& I_{N, \lambda}^{\alpha, \beta}\left(\tilde{p}_1\left(\theta\right) \int_0^\theta e^*(\eta) d \eta\right)
		+ \varepsilon I_{N, \lambda}^{\alpha, \beta}\left(\tilde{q}_1\left(\theta\right) \int_0^\theta e^*(\varepsilon \eta) d \eta\right)\notag\\
		& +I_{N, \lambda}^{\alpha, \beta}\left(\theta^{-\gamma} \int_0^\theta\left(\theta-\eta\right)^{-\mu} \eta^{\mu+\gamma-1} \bar{K}_1(\theta, \eta) e(\eta) d \eta\right) \notag\\
		&+I_{N, \lambda}^{\alpha, \beta}\left(\theta^{-\gamma} \int_0^\theta\left(\theta-\eta\right)^{-\mu} \eta^{\mu+\gamma-1} \bar{K}_2(\theta,\varepsilon \eta)e(\varepsilon \eta) d \eta\right) \notag\\
		&+ \sum_{i=0}^N I_{i, 1} F_{i,\lambda}(\theta)+\sum_{i=0}^N I_{i, 2} F_{i,\lambda}(\theta)\label{eq_sum1_prime}\\
		I_{N, \lambda}^{\alpha \beta} \varphi(\theta)-\varphi_N(\theta)= & I_{N, \lambda}^{\alpha, \beta}\left(\int_0^\theta e^*(\eta) d \eta\right)\label{eq_sum1_origin},\\
		I_{N, \lambda}^{\alpha \beta} \varphi(\varepsilon\theta)-\varphi_N(\varepsilon\theta)= & I_{N, \lambda}^{\alpha, \beta}\left(\varepsilon\int_0^\theta e^*(\varepsilon\eta) d \eta\right)\label{eq_sum1_eorigin}.
	\end{align}
Adding and subtracting $\varphi^{\prime}(\theta)$ to left side of \eqref{eq_sum1_prime}, $\varphi(\theta)$ to left side of \eqref{eq_sum1_origin} and $\varphi(\varepsilon\theta)$ to left side of \eqref{eq_sum1_eorigin}, respectively yield the following results:
	\begin{align}
		e^{\ast}(\theta)=
		&\tilde{p}_1\left(\theta\right) \int_0^\theta e^*(\eta) d \eta
		+ \varepsilon \tilde{q}_1\left(\theta\right) \int_0^\theta e^*(\varepsilon \eta) d \eta+\theta^{-\gamma} \int_0^\theta\left(\theta-\eta\right)^{-\mu} \eta^{\mu+\gamma-1} \bar{K}_1(\theta, \eta) e(\eta) d \eta \notag\\
		&+\theta^{-\gamma} \int_0^\theta\left(\theta-\eta\right)^{-\mu} \eta^{\mu+\gamma-1} \bar{K}_2(\theta,\varepsilon \eta)e(\varepsilon \eta) d\eta
		+\sum_{p=1}^{7}E_p(\theta),\label{eq_sum*_prime}\\
		e(\theta)=&\int_0^\theta e^*(\eta) d \eta+E_8(\theta)+E_9(\theta),\label{eq_sum*_origin}\\
		e(\varepsilon\theta)=&\varepsilon\int_0^\theta e^*(\varepsilon\eta) d \eta+E_8(\varepsilon\theta)+E_9(\varepsilon\theta).\label{eq_sum*_eorigin}
	\end{align}
where
	\begin{align}
		&E_1(\theta)=\sum_{i=0}^N I_{i, 1} F_{i,\lambda}(\theta),  
		E_2(\theta)=\sum_{i=0}^N I_{i, 2} F_{i,\lambda}(\theta),
		E_3(\theta)= \varphi^{\prime}(\theta)-I_{N, \lambda}^{\alpha, \beta} \varphi^{\prime}(\theta),\notag\\
		&E_4(\theta)= I_{N, \lambda}^{\alpha, \beta}\left(\tilde{p}_1\left(\theta\right) \int_0^\theta e^*(\eta) d \eta\right)-\tilde{p}_1\left(\theta\right) \int_0^\theta e^*(\eta) d \eta,\notag\\
		&E_5(\theta)=\varepsilon I_{N, \lambda}^{\alpha, \beta}\left(\tilde{q}_1\left(\theta\right) \int_0^\theta e^*(\varepsilon\eta) d \eta\right)- \varepsilon \tilde{q}_1\left(\theta\right) \int_0^\theta e^*(\varepsilon \eta) d \eta,\notag\\
		&E_6(\theta)=I_{N, \lambda}^{\alpha, \beta}\left(\theta^{-\gamma} \int_0^\theta\left(\theta-\eta\right)^{-\mu} \eta^{\mu+\gamma-1} \bar{K}_1(\theta, \eta) e(\eta) d \eta\right)\notag\\
		&\qquad\qquad-\theta^{-\gamma} \int_0^\theta\left(\theta-\eta\right)^{-\mu} \eta^{\mu+\gamma-1} \bar{K}_1(\theta, \eta) e(\eta) d \eta ,\notag
	\end{align}
	\begin{align}
		&E_7(\theta)=I_{N, \lambda}^{\alpha, \beta}\left(\theta^{-\gamma} \int_0^\theta\left(\theta-\eta\right)^{-\mu} \eta^{\mu+\gamma-1} \bar{K}_2(\theta,\varepsilon \eta)e(\varepsilon \eta) d \eta\right)\notag\\
		&\qquad\qquad-\theta^{-\gamma} \int_0^\theta\left(\theta-\eta\right)^{-\mu} \eta^{\mu+\gamma-1} \bar{K}_2(\theta,\varepsilon \eta)e(\varepsilon \eta) d\eta,\notag\\
		&E_8(\theta)=\varphi(\theta)-I_{N, \lambda}^{\alpha \beta} \varphi(\theta),
		E_9(\theta)= I_{N, \lambda}^{\alpha, \beta}\left(\int_0^\theta e^*(\eta) d \eta\right)-\int_0^\theta e^*(\eta) d \eta.\notag
	\end{align}
\begin{theorem}\label{theorem_Gronwall}
The error of $e^{\ast}(\theta)$ and $e(\theta)$ can be bounded by $E_p(\theta)(p=1,2,\cdots,9)$, which means
	\begin{align}
		\|{e}^{*}(\theta)\|_{\infty} &\leq C\sum_{p=1}^9\left\|E_p(\theta)\right\|_{\infty},\label{theorem_normcontrol_e*_infty}\\
	\|e(\theta)\|_{\infty} &\leq C \left(\left\|e^*(\theta)\right\|_{\infty}+\left\|E_8(\theta)\right\|_{\infty}+\left\|E_9(\theta)\right\|_{\infty}\right).\label{theorem_normcontrol_e_infty}
	\end{align}
	\begin{align}
		\left\|e^*(\theta)\right\|_{0,\omega^{\alpha,\beta,1}} 
		&\leq C\left(\sum_{p=1}^{9}\left\|E_p(\theta)\right\|_{0,\omega^{\alpha,\beta,1}}+\left\|E_8(\theta)\right\|_{\infty}+\left\|E_9(\theta)\right\|_{\infty} \right)\label{theorem_normcontrol_e*_w}\\
		\left\|e(\theta)\right\|_{0,\omega^{\alpha,\beta,1}}  &\leq C
		\left(\left\|e^*(\theta)\right\|_{0,\omega^{\alpha,\beta,1}}+\left\|E_8(\theta)\right\|_{0,\omega^{\alpha,\beta,1}}+\left\|E_9(\theta)\right\|_{0,\omega^{\alpha,\beta,1}} \right)\label{theorem_normcontrol_e_w}
	\end{align}
where $e(\theta)=\varphi(\theta)-\varphi_N(\theta),e^{\ast}(\theta)=\varphi^{\prime}(\theta)-\varphi_N^{\ast}(\theta)$. 
\end{theorem}
\begin{proof}
Substituting \eqref{eq_sum*_origin} and \eqref{eq_sum*_eorigin} into \eqref{eq_sum*_prime} yields an inequality concerning $e^{\ast}(\theta)$ and $e(\theta)$:
	\begin{align}
		e^*(\theta)= & \tilde{p}_1(\theta) \int_0^\theta e^*(\eta) d \eta+e \tilde{q}_1(\theta) \int_0^\theta e^*(\varepsilon \eta) d \eta   +\sum_{p=1}^7 E_p(\theta) \notag\\
		& +\theta^{-\gamma} \int_0^\theta(\theta-\eta)^{-\mu}\eta^{\mu+\gamma-1} \bar{K}_1(\theta, \eta) e(\eta) d \eta+\theta^{-\gamma} \int_0^\theta(\theta-\eta)^{-\mu} \eta^{\mu+\gamma-1} \bar{K}_2(\theta, \varepsilon \eta) e(\varepsilon \eta) d \eta\notag \\
		= & \tilde{p}_1(\theta) \int_0^\theta e^*(\eta) d \eta+\tilde{q}_1(\theta) \int_0^{\varepsilon \theta} e^*(\eta) d \eta\notag \\
		& +\theta^{-\gamma} \int_0^\theta(\theta-\eta)^{-\mu} \eta^{\mu+\gamma-1} \bar{K}_1(\theta, \eta)\left(\int_0^\eta e^*(\sigma)\right) d \eta\notag\\
		&+\theta^{-\gamma} \int_0^\theta(\theta-\eta)^{-\mu}\eta^{\mu+\gamma-1}  \bar{K}_2(\theta,\varepsilon\eta)\left(\int_0^{\varepsilon \sigma} e^*(\sigma) d \sigma\right)d\eta+J(\theta)\notag 
	\end{align}
	\begin{align}
		= & \tilde{p}_1(\theta) \int_0^\theta e^*(\eta) d \eta+\tilde{q}_1(\theta) \int_0^{\varepsilon \theta} e^*(\eta) d \eta\notag \\
		& +\theta^{-\gamma} \int_0^\theta\left(\int_\eta^\theta(\theta-\sigma)^{-\mu}\sigma^{\mu+\gamma-1}  \bar{K}_1(\theta, \sigma) d \sigma\right) e^*(\eta) d \eta \notag\\
		& +\theta^{-\gamma} \int_0^{\varepsilon \theta}\left(\int_{\frac{\eta}{\varepsilon}}^\theta(\theta-\sigma)^{-\mu}\sigma^{\mu+\gamma-1}  \bar{K}_2(\theta, \varepsilon\sigma) d \sigma\right) e^*(\eta) d \eta+J(\theta)\notag\\
		=&\theta^{-\gamma} \int_0^\theta\left(\int_\eta^\theta(\theta-\sigma)^{-\mu}\sigma^{\mu+\gamma-1}  \bar{K}_1(\theta, \sigma) d \sigma+\theta^{\gamma}\tilde{p}_1(\theta)\right) e^*(\eta) d \eta \notag\\
		& +\theta^{-\gamma} \int_0^{\varepsilon \theta}\left(\int_{\frac{\eta}{\varepsilon}}^\theta(\theta-\sigma)^{-\mu}\sigma^{\mu+\gamma-1}  \bar{K}_2(\theta, \varepsilon\sigma) d \sigma+\theta^{\gamma}\tilde{q}_1(\theta)\right) e^*(\eta) d \eta+J(\theta).\label{eq_Gronwall_origin}
	\end{align}
where
	\begin{align}
		J(\theta)=&\theta^{-\gamma} \int_0^\theta(\theta-\eta)^{-\mu} \sigma^{\mu+\gamma-1} \bar{K}_1(\theta, \eta)\left(E_8(\eta)+E_9(\eta)\right) d \eta\notag\\
		&+\theta^{-\gamma} \int_0^\theta(\theta-\eta)^{-\mu}\sigma^{\mu+\gamma-1}  \bar{K}_2(\theta,\varepsilon\eta)\left(E_8(\varepsilon\eta)+E_9(\varepsilon\eta)\right) d\eta
		+\sum_{p=1}^7 E_p(\theta).\label{eq_Gronwall_J}
	\end{align}
Now introduce the area $D:\lbrace(\theta,\sigma):0 \leq \sigma \leq \theta,\theta \in [0,1]\rbrace$,where we denote $\underset{(\theta,\sigma) \in D}{\max}\left|\bar{K}_1(\theta,\sigma)\right|=\bar{K}_1^{max}$, 
$\underset{(\theta,\sigma) \in D}{\max}\left|\bar{K}_2(\theta,\varepsilon\sigma)\right|=\bar{K}_2^{max}$, 
$\underset{\theta \in [0,1]}{\max}\left|\tilde{p}_1(\theta)\right|=p^{max}$, $\underset{\theta \in [0,1]}{\max}\left|\tilde{q}_1(\theta)\right|=q^{max}$, then we can obtain
	\begin{align}
		&\left|\theta^{-\gamma}\left(\int_\eta^\theta(\theta-\sigma)^{-\mu} \sigma^{\mu+\gamma-1} \bar{K}_1(\theta, \sigma) d \sigma+\theta^{\gamma}\tilde{p}_1(\theta)\right)\right| \notag\\
		\leq &\left|\theta^{-\gamma}\left(\int_0^\theta(\theta-\sigma)^{-\mu} \sigma^{\mu+\gamma-1} \bar{K}_1(\theta, \sigma) d \sigma+\theta^{\gamma}\tilde{p}_1(\theta)\right)\right| \notag\\
		\leq& \left|B(1-\mu,\mu+\gamma)\bar{K}_1^{max}+\tilde{p}_1(\theta)\right|\notag\\
		\leq&\left|B(1-\mu,\mu+\gamma)\bar{K}_1^{max}\right|+p^{max}\notag\\
		\equiv &C_1.\label{eq_Gronwall_C1}
	\end{align}
	\begin{align}
		&\left|\theta^{-\gamma}\left(\int_{\frac{\eta}{\varepsilon}}^\theta(\theta-\sigma)^{-\mu}\sigma^{\mu+\gamma-1}  \bar{K}_2(\theta, \varepsilon\sigma) d \sigma+\theta^{\gamma}\tilde{q}_1(\theta)\right)\right|\notag\\
		\leq&\bar{K}_2^{max}\left|\theta^{-\gamma}\left(\int_{\frac{\eta}{\varepsilon}}^\theta(\theta-\sigma)^{-\mu}\sigma^{\mu+\gamma-1}d \sigma\right)\right|+q^{max}\notag\\
		\leq&\bar{K}_2^{max}\left|\theta^{-\gamma}\left(\int_{0}^\theta(\theta-\sigma)^{-\mu}\sigma^{\mu+\gamma-1}d \sigma\right)\right|+q^{max}\notag\\
		\leq &\bar{K}_2^{max}B(1-\mu, \mu+\gamma)+q^{max} \notag\\
		\equiv&C_2.\label{eq_Gronwall_C2}
	\end{align}
where $0 \leq \frac{\eta}{\varepsilon} \leq \theta$.
According to \eqref{eq_Gronwall_C1} and \eqref{eq_Gronwall_C2}, take the absolute value of both sides of \eqref{eq_Gronwall_origin}. By the triangle inequality, we can obtain
	\begin{align}
		\left|e^*(\theta)\right| & \leq C_1 \int_0^\theta\left|e^*(\eta)\right| d \eta+C_2 \int_0^{\varepsilon \theta}\left|e^*(\eta)\right| d \eta+|J(\theta)| \notag\\ 
		& \leq\left(C_1+C_2\right) \int_0^\theta\left|e^*(\eta)\right| d \eta+|J(\theta)|.\label{eq_Gronwall_e*}
	\end{align}
It follows from Lemma \ref{lemma_Chen_Lemma3.5_norm} that
	\begin{align}
		\|{e}^{*}(\theta)\|_{\infty} \leq C\|J(\theta)\|_{\infty}.
	\end{align}
Based on the relationship between $e(\theta)$ and $e^*(\theta)$ in \eqref{eq_sum*_origin}, we can deduce
	\begin{align}
		\|e(\theta)\|_{\infty} \leq \|{e}^{*}(\theta)\|_{\infty}+\left\|E_8(\theta)\right\|_{\infty}+\left\|E_9(\theta)\right\|_{\infty} .\label{e*_e}
	\end{align}
It can be determined from \eqref{eq_Gronwall_J} that
	\begin{align}
		\|J(\theta)\|_{\infty} \leq C \sum_{p=1}^9\left\|E_p(\theta)\right\|_{\infty}.
	\end{align}
In conclusion,
	\begin{align}
		&\|{e}^{*}(\theta)\|_{\infty} \leq C\sum_{p=1}^9\left\|E_p(\theta)\right\|_{\infty},\\
		&\|e(\theta)\|_{\infty} \leq C \left(\left\|e^*(\theta)\right\|_{\infty}+\left\|E_8(\theta)\right\|_{\infty}+\left\|E_9(\theta)\right\|_{\infty}\right).
	\end{align}
By virtue of \eqref{eq_sum*_origin}, \eqref{eq_Gronwall_J}, \eqref{eq_Gronwall_e*}, Lemma \ref{lemma_Chen_Lemma3.5_norm}, and Lemma \ref{lemma_Hardy}, we can derive the following estimates
	\begin{align}
		\left\|e^*(\theta)\right\|_{0,\omega^{\alpha,\beta,1}}  &\leq C\left( \sum_{p=1}^{9}\left\|E_p(\theta)\right\|_{0,\omega^{\alpha,\beta,1}}
		+\left\|E_8(\varepsilon\eta)\right\|_{0,\omega^{\alpha,\beta,1}}+\left\|E_8(\varepsilon\eta)\right\|_{0,\omega^{\alpha,\beta,1}}\right)\notag\\
		&\leq C\left(\sum_{p=1}^{9}\left\|E_p(\theta)\right\|_{0,\omega^{\alpha,\beta,1}}+\left\|E_8(\theta)\right\|_{\infty}+\left\|E_9(\theta)\right\|_{\infty} \right)\\
		\left\|e(\theta)\right\|_{0,\omega^{\alpha,\beta,1}}  &\leq C
		\left(\left\|e^*(\theta)\right\|_{0,\omega^{\alpha,\beta,1}}+\left\|E_8(\theta)\right\|_{0,\omega^{\alpha,\beta,1}}+\left\|E_9(\theta)\right\|_{0,\omega^{\alpha,\beta,1}} \right).\label{eq_e_e*}
	\end{align}
\end{proof}
Accordingly we will deduce convergence analysis in $L^{\infty}$-norm.

\subsection{Error estimate in $L^{\infty}$-norm}
\begin{theorem}\label{theorem_infty_norm}
Let $\varphi(\theta)$ be an exact solution of the VIDEs \eqref{eq_phi_prime} and \eqref{eq_phi_origin}, assuming it is sufficiently smooth.	
	$$
	\begin{aligned}
		\varphi_{N}^{\ast}\left(\theta\right)=\sum_{j=0}^{N}\varphi_j^{\ast}F_{j,\lambda}(\theta),
		\varphi_{N}\left(\theta\right)=\sum_{j=0}^{N}\varphi_jF_{j,\lambda}(\theta).\notag
	\end{aligned}
	$$
$\lbrace\varphi_j^*\rbrace_{j=0}^N,\lbrace\varphi_j\rbrace_{j=0}^N$ can be seen in \eqref{eq_collocation*_prime} and \eqref{eq_collocation*_origin}.
If $ \varphi\left(\theta^{\frac{1}{\lambda}}\right) \in B_{\alpha, \beta}^{m, 1}(I), \partial_\theta \varphi\left(\theta^{\frac{1}{\lambda}}\right) \in B_{\alpha, \beta}^{m, 1}(I), \\
\tilde{p}_1(\theta), \tilde{q}_1(\theta), \tilde{g}_1(\theta) \in C^m(I), \bar{K}_1(\theta, \eta), \bar{K}_2(\theta, \varepsilon \eta) \in C^m(I \times I),m \geq 1,\text{ and }-1<\alpha, \beta \leq-\frac{1}{2},$$0<\mu<1,$ where $\mu$ is associted with the weakly singular kernel and $0 < \kappa < 1-\mu,$
then we have
	\begin{align}
		\left\|e^*(\theta)\right\|_{\infty}\leq &CN^{-m}\Bigl\{\log N\bigl(\mathcal{K}^*\left\|\varphi(\theta)\right\|_{\infty}
		+N^{\frac{1}{2}-\kappa}\left\|\partial_\theta^m \varphi\left(\theta^{\frac{1}{\lambda}}\right)\right\|_{0, \omega^{\alpha+m,\beta+m,1}}\bigr)+N^{\frac{1}{2}}\Phi\Bigr\},\label{theorem_e*_infty}\\
		\left\|e(\theta)\right\|_{\infty}\leq&CN^{-m}\Bigl\{\log 	N\bigl(\mathcal{K}^*\left\|\varphi(\theta)\right\|_{\infty}
		+N^{\frac{1}{2}-\kappa}\left\|\partial_\theta^m 	\varphi\left(\theta^{\frac{1}{\lambda}}\right)\right\|_{0, \omega^{\alpha+m,\beta+m,1}}\bigr)
		+N^{\frac{1}{2}}\Phi\Bigr\}.\label{theorem_e_infty}
	\end{align}
where a sufficiently large positive integer $N$ is provided and
	\begin{align}
		\mathcal{K}^*=& \max\limits_{0\leq i\leq N}\left\|\partial_\theta^m \widetilde{K}_1\left(\theta_i, \eta_i(\cdot)\right)\right\|_ {0,\omega^{m-\mu,{m+\frac{\mu+\gamma}{\lambda}}-1,1}}+ \max\limits_{0\leq i\leq N}\left\|\partial_\theta^m \widetilde{K}_2\left(\theta_i, \varepsilon \eta_i(\cdot)\right)\right\|_ {0,\omega^{m-\mu,{m+\frac{\mu+\gamma}{\lambda}}-1,1}},\notag\\
		\Phi=&\left\|\partial_\theta^{m+1} \varphi\left(\theta^{\frac{1}{\lambda}}\right)\right\|_{0, \omega^{\alpha+m,\beta+m,1}}+\left\|\partial_\theta^m \varphi\left(\theta^{\frac{1}{\lambda}}\right)\right\|_{0, \omega^{\alpha+m,\beta+m,1}}.\notag
	\end{align}
\end{theorem}
\begin{proof}
As discussed in \eqref{eq_Ii1},\eqref{eq_Ii2},\eqref{e*_e}, Lemma \ref{lemma_Lesbegue_constant} and Lemma \ref{lemma_Ii1Ii2}, we can obtain
	\begin{align}
		\left\|E_1(\theta)\right\|_{\infty}
		&\leq C\max\limits_{0\leq i\leq N}\left|I_{i, 1}\right| \max\limits_{\theta \in I}\sum_{j=0}^{N}\left|F_{j,\lambda}(\theta)\right|\notag\\
		&\leq C N^{-m}\log N \max\limits_{0\leq i\leq N}\left\|\partial_\theta^m \widetilde{K}_1\left(\theta_i,  \eta_i(\cdot)\right)\right\|_ {0,\omega^{m-\mu,{m+\frac{\mu+\gamma}{\lambda}}-1,1}}
		\left\|\varphi_N\left( \eta_i\left(\cdot\right)\right)\right\|_{0,\omega^{-\mu,\frac{\mu+\gamma}{\lambda}-1,1}},\notag\\
		&\leq C N^{-m}\log N \max\limits_{0\leq i\leq N}\left\|\partial_\theta^m \widetilde{K}_1\left(\theta_i,  \eta_i(\cdot)\right)\right\|_ {0,\omega^{m-\mu,{m+\frac{\mu+\gamma}{\lambda}}-1,1}}
		\left(\left\|\varphi(\theta)\right\|_{\infty}+\left\|e(\theta)\right\|_\infty\right)\notag\\
		&\leq C N^{-m}\log N \max\limits_{0\leq i\leq N}\left\|\partial_\theta^m \widetilde{K}_1\left(\theta_i,  \eta_i(\cdot)\right)\right\|_ {0,\omega^{m-\mu,{m+\frac{\mu+\gamma}{\lambda}}-1,1}}\notag\\
		&\quad\left(\left\|\varphi(\theta)\right\|_{\infty}+\|{e}^{*}(\theta)\|_{\infty}+\left\|E_8(\theta)\right\|_{\infty}+\left\|E_9(\theta)\right\|_{\infty}\right),\label{eq_E1_infty}
	\end{align}
	\begin{align}
		\left\|E_2(\theta)\right\|_{\infty}
		&\leq C\max\limits_{0\leq i\leq N}\left|I_{i, 2}\right| \max\limits_{\theta \in I}\sum_{j=0}^{N}\left|F_{j,\lambda}(\theta)\right|\notag\\
		&\leq C N^{-m}\log N \max\limits_{0\leq i\leq N}\left\|\partial_\theta^m \widetilde{K}_2\left(\theta_i, \varepsilon \eta_i(\cdot)\right)\right\|_ {0,\omega^{m-\mu,{m+\frac{\mu+\gamma}{\lambda}}-1,1}}
		\left\|\varphi_N\left(\varepsilon \eta_i\left(\cdot\right)\right)\right\|_ {0,\omega^{-\mu,\frac{\mu+\gamma}{\lambda}-1,1}}\notag\\
		&\leq C N^{-m}\log N \max\limits_{0\leq i\leq N}\left\|\partial_\theta^m \widetilde{K}_2\left(\theta_i, \varepsilon \eta_i(\cdot)\right)\right\|_ {0,\omega^{m-\mu,{m+\frac{\mu+\gamma}{\lambda}}-1,1}}
		\left(\left\|\varphi(\theta)\right\|_{\infty}+\left\|e(\theta)\right\|_{\infty}\right)\notag\\
		&\leq C N^{-m}\log N \max\limits_{0\leq i\leq N}\left\|\partial_\theta^m \widetilde{K}_2\left(\theta_i, \varepsilon \eta_i(\cdot)\right)\right\|_ {0,\omega^{m-\mu,{m+\frac{\mu+\gamma}{\lambda}}-1,1}}\notag\\
		&\quad\left(\left\|\varphi(\theta)\right\|_{\infty}+\|{e}^{*}(\theta)\|_{\infty}+\left\|E_8(\theta)\right\|_{\infty}+\left\|E_9(\theta)\right\|_{\infty}\right).\label{eq_E2_infty}
	\end{align}
Applying the integration error estimate from Lemma \ref{lemma_Interpolation_infinity},
	\begin{align}
		&\left\|E_3(\theta)\right\|_{\infty} \leq C N^{\frac{1}{2}-m}\left\|\partial_\theta^{m+1} \varphi\left(\theta^{\frac{1}{\lambda}}\right)\right\|_{0, \omega^{\alpha+m,\beta+m,1}},\label{eq_E3_infity}\\
		&\left\|E_8(\theta)\right\|_{\infty}\leq C N^{\frac{1}{2}-m}\left\|\partial_\theta^m \varphi\left(\theta^{\frac{1}{\lambda}}\right)\right\|_{0, \omega^{\alpha+m,\beta+m,1}}.\label{eq_E8_infty}
	\end{align}

$\int_0^{\theta^{\frac{1}{\lambda}}} e^*(\eta) d \eta=\int_0^{\theta^{\frac{1}{\lambda}}}\left(\varphi^{\prime}(\eta)-\varphi_N^*(\eta)\right) d \eta,$
we set $ \varphi\left(\theta^{\frac{1}{\lambda}}\right) \in B_{\alpha, \beta}^{1,1}(I) \text { , then } \int_0^{\theta^{\frac{1}{\lambda}}} \varphi^{\prime}(\eta) d \eta \in B_{\alpha, \beta}^{1,1}(I)$. 
We have $\int_0^{\theta^{\frac{1}{\lambda}}} \varphi_N^*(\eta) d \eta \in B_{\alpha, \beta}^{1,1}(I)$ from  $\varphi_N^*\left(\theta^{\frac{1}{\lambda}}\right) \in B_{\alpha, \beta}^{1,1}(I)$, then $\left\|E_4(\theta)\right\|_{\infty}$ can be estimated by Lemma \ref{lemma_Interpolation_infinity} setting $m=1$,
		\begin{align}
		\left\|E_4(\theta)\right\|_{\infty}
		& \leq\left\|\left(I_{N, \lambda}^{a, \beta}-I\right) \tilde{p}_1(\theta) \int_0^\theta e^*(\eta) d \eta\right\|_{\infty}\notag \\
		& \leq C N^{-\frac{1}{2}}\left\|\partial_\theta\left( \tilde{p}_1(\theta^{\frac{1}{\lambda}}) \int_0^{\theta^{\frac{1}{\lambda}}} e^*(\eta) d \eta\right)\right\|_{\infty} \notag\\
		&\leq C N^{-\frac{1}{2}}\left\|\frac{1}{\lambda} \theta^{\frac{1}{\lambda}-1}\partial_\theta\tilde{p}_1(\theta^{\frac{1}{\lambda}})\int_0^{\theta^{\frac{1}{\lambda}}} e^*(\eta) d \eta    \right\|_{\infty} 
		+ C N^{-\frac{1}{2}}\left\|\frac{1}{\lambda} \theta^{\frac{1}{\lambda}-1} e^*\left(\theta^{\frac{1}{\lambda}}\right)+\int_0^{\theta^{\frac{1}{\lambda}}} e^*(\eta) d \eta\right\|_{\infty}\notag\\
		& \leq C N^{-\frac{1}{2}}\left\|\theta^{\frac{1}{\lambda}}e^*(\theta)\right\|_{\infty}+C N^{-\frac{1}{2}}\left\|\frac{1}{\lambda} \theta^{\frac{1}{\lambda}-1} e^*\left(\theta^{\frac{1}{\lambda}}\right)+\int_0^{\theta^{\frac{1}{\lambda}}} e^*(\eta) d \eta\right\|_{\infty} \notag\\
		&\leq C N^{-\frac{1}{2}}\left(\left\|e^*\left(\theta^{\frac{1}{\lambda}}\right)\right\|_{\infty}+
		\left\|e^*(\theta)\right\|_{\infty}\right)
		\leq C N^{-\frac{1}{2}}\left\|e^*(\theta)\right\|_{\infty}.\label{eq_E4_infty}
	\end{align}
The same approach can be used to estimate $\left\|E_5(\theta)\right\|_{\infty},\left\|E_9(\theta)\right\|_{\infty}$, respectively 
	\begin{align}
		&\left\|E_5(\theta)\right\|_{\infty}\leq C N^{-\frac{1}{2}}\left\|e^*(\theta)\right\|_{\infty},\label{eq_E5_infty}\\
		&\left\|E_9(\theta)\right\|_{\infty}\leq C N^{-\frac{1}{2}}\left\|e^*(\theta)\right\|_{\infty}.\label{eq_E9_infty}
	\end{align}
When 
	$\mu+\gamma \geq 1 \text {, } \bar{K}_2(\theta, \varepsilon\eta) \in C^m(I \times I) \text {, so there exists } \xi \in(0, \theta). 
\text { Performing a Taylor expansion on it }$\\
can yield
	$\bar{K}_2(\theta, \varepsilon\eta)=\bar{K}_2(0,0)+\hat{K}_2(\theta, \varepsilon\eta), \\$
	\begin{align}
		\hat{K}_2(\theta, \varepsilon\eta)=&\sum_{k=1}^{m-1} \frac{\partial_{(\varepsilon\eta)}^k \bar{K}_2(0,0)}{k!}(\varepsilon\eta)^k+\frac{\partial_{(\varepsilon\eta)}^m \bar{K}_2\left(\theta, \varepsilon_1\right)}{m!}(\varepsilon\eta)^m\notag\\
		&+\sum_{k=1}^{m-1} \frac{\partial_\theta^k \bar{K}_2(0,0)}{k!}\theta^k+\frac{\partial_\theta^m \bar{K}_2\left(\varepsilon_2,0\right)}{m!}\theta^m
	\end{align}
where $\varepsilon_1, \varepsilon_2 \in (0,1)$, then
	\begin{align}
		E_7(\theta)&=\left(I_{N, \lambda}^{\alpha, \beta}-I\right) \int_0^\theta \frac{(\theta-\eta)^{-\mu}}{\theta^{\gamma}} \eta^{\mu+\gamma-1}\left(\bar{K}_2(0,0)+\hat{K}_2(\theta, \varepsilon\eta)\right) e(\varepsilon\eta) d \eta \notag\\
		& =\left(I_{N, \beta}^{\alpha, \beta}-I\right) \int_0^\theta \frac{(\theta-\eta)^{-\mu}}{\theta^{\gamma}} \eta^{\mu+\gamma-1} \bar{K}_2(0,0)\left(\varphi(\varepsilon\eta)-\varphi_N(\varepsilon\eta)\right) d \eta \notag\\
		& \quad+\left(I_N^{\alpha, \beta}-I\right) \int_0^\theta \frac{(\theta-\eta)^{-\mu}}{\theta^{\gamma}} \eta^{\mu+\gamma-1} \hat{K}_2(\theta,\varepsilon \eta) e(\varepsilon\eta) d \eta \notag\\
		& =E_{7,1}(\theta)+E_{7,2}(\theta)+E_{7,3}(\theta).\notag
	\end{align}
Based on Lemmas \ref{lemma_J6}, \ref{lemma_Hardy} and \ref{lemma_Interpolation_infinity}, error estimates are made for $E_{7,1}(\theta), E_{7,2}(\theta), E_{7,3}(\theta)$ respectively.
Since $\int_0^\theta \frac{(\theta-\eta)^{-\mu}}{\theta^{\gamma}} \eta^{ \mu+\gamma-1} d \eta=B(1-\mu, \mu+\gamma),$
therefore, by the mean value theorem for integrals, it is known that
	$ \int_0^\theta \frac{(\theta-\eta)^{-\mu}}{\theta^{\gamma}} \eta^{\mu+\gamma-1} \bar{K}_2(0,0)\varphi_N(\varepsilon\eta)d \eta \in P_N^\lambda(I)$,
so \begin{align}E_{7, 2}(\theta)\equiv0\label{eq_E62_infty}.\end{align}
	\begin{align}
	\left\|E_{7,1}(\theta)\right\|_{\infty} &\leq C N^{\frac{1}{2}-m}\left\|\partial_{\theta}^m \int_0^{\theta^{\frac{1}{\lambda}}}\left(\theta^{\frac{1}{\lambda}}\right)^{-\gamma}\left(\theta^{\frac{1}{\lambda}}-\eta\right)^{-\mu} \eta^{\mu+\gamma-1}\bar{K}_2(0,0) \varphi(\varepsilon\eta)d \eta\right\|_{0, \omega^{\alpha+m, \beta+m, 1}}\notag\\
	&\overset{(1)}{\leq} C N^{\frac{1}{2}-m}\left\| \frac{1}{\lambda}\int_{0}^{1}\left(1-q^{\frac{1}{\lambda}}\right)^{-\mu}q^{\frac{\mu+\gamma}{\lambda}-1}\bar{K}_2(0,0)\partial_{\theta}^m\varphi\left(\varepsilon(\theta q)^{\frac{1}{\lambda}}\right)d q
	\right\|_{0, \omega^{\alpha+m, \beta+m, 1}}\notag\\
	&\overset{(2)}{\leq} C N^{\frac{1}{2}-m}\left\| \frac{1}{\lambda}\int_{0}^{1}\left(1-q^{\frac{1}{\lambda}}\right)^{-\mu}q^{\frac{\mu+\gamma}{\lambda}-1}\bar{K}_2(0,0)\varepsilon^{\lambda m}q^m\partial_{\hat{q}}^m\varphi\left((\hat{q})^{\frac{1}{\lambda}}\right)d q
	\right\|_{0, \omega^{\alpha+m, \beta+m, 1}}\notag\\
	&\overset{(3)}{\leq}C N^{\frac{1}{2}-m}\left\|\partial_\theta^m \varphi\left(\theta^{\frac{1}{\lambda}}\right)\right\|_{0, \omega^{\alpha+m,\beta+m,1}}.\label{eq_E61_infty}
	\end{align}
where $(1)\text{ Variable transformation }\eta=(\theta q)^{\frac{1}{\lambda}}$,$(2)\text{ Variable transformation }\hat{q}=\varepsilon^{\lambda}\theta q$, $(3)$ Follow from Lemma \ref{lemma_Hardy}, let
	$$
	\begin{aligned}
		\mathcal{\hat{K}}_2e(\theta)=\int_0^\theta \frac{(\theta-\eta)^{-\mu}}{\theta^{\gamma}} \eta^{\mu+\gamma-1} \mathcal{\hat{K}}_2(\theta, \varepsilon\eta) e(\varepsilon\eta) d \eta\notag.
	\end{aligned}
	$$
It is transformed from \eqref{eq_Operator}. So by virtue of \eqref{eq_Interpolation_operator_sum}, Lemmas \ref{lemma_kappa_norm}, \ref{lemma_J6} and \ref{lemma_Lesbegue_constant}, for any $\kappa \in (0,1-\mu)$ and $N$ sufficiently large, the desired estimates can be 
	\begin{align}
		\left\|E_{7,3}(\theta)\right\|_{\infty} & =\max _{z^{\frac{1}{\lambda}} \in I}\left|\left(I_{N, \lambda}^{\alpha \beta}-I\right)\left(\mathcal{\hat{K}}_2 e\right)\left(z^{\frac{1}{\lambda}}\right)\right| \notag\\ 
		& =\left\|\left(I_{N, 1}^{\alpha,\beta}-I\right)\left(\mathcal{\hat{K}}_2 e\right)\left(z^{\frac{1}{\lambda}}\right)\right\|_{\infty} \notag\\ 
		& =\left\|\left(I_{N,1}^{\alpha \cdot \beta}-I\right)\left(I-\mathcal{T}_N\right)\mathcal{\hat{K}}_2 e\left(z^{\frac{1}{\lambda}}\right)\right\|_{\infty} \notag\\ 
		&=\left(\left\|I_{N,1}^{\alpha \cdot \beta}\right\|_{\infty}+1\right)\left\|\left(I-\mathcal{T}_N\right)\mathcal{\hat{K}}_2 e\left(z^{\frac{1}{\lambda}}\right)\right\|_{\infty}\notag\\ 
		& \leq C N^{-\kappa} \log N\left\|\mathcal{\hat{K}}_2 e\left(z^{\frac{1}{\lambda}}\right)\right\|_{0,\kappa} \notag\\ 
		& \leq C N^{-\kappa} \log N\|e(\theta)\|_{\infty}.\label{eq_E63_infty}
	\end{align}
So on, the estimate for $\left\|E_{7}(\theta)\right\|_{\infty}$ is provided under \eqref{eq_E62_infty}-\eqref{eq_E63_infty}
	\begin{align}
		\left\|E_{7}(\theta)\right\|_{\infty}&=C\sum_{k=0}^{3}\left\|E_{7,k}(\theta)\right\|_{\infty}\notag\\
		&\leq C N^{\frac{1}{2}-m}\left\|\partial_\theta^m \varphi\left(\theta^{\frac{1}{\lambda}}\right)\right\|_{0, \omega^{\alpha+m,\beta+m,1}}+ C N^{-\kappa} \log N\|e(\theta)\|_{\infty}\notag\\
		&\leq C N^{\frac{1}{2}-m}\left\|\partial_\theta^m \varphi\left(\theta^{\frac{1}{\lambda}}\right)\right\|_{0, \omega^{\alpha+m,\beta+m,1}}\notag\\
		&\quad+CN^{-\kappa} \log N\left(\|{e}^{*}(\theta)\|_{\infty}+\left\|E_8(\theta)\right\|_{\infty}+\left\|E_9(\theta)\right\|_{\infty}\right),\label{eq_E6_infty}
	\end{align}
Similarly, let $\widetilde{q}=\theta q$. It can be obtained that
		\begin{align}
		\left\|E_{6}(\theta)\right\|_{\infty}=&C\sum_{k=0}^{3}\left\|E_{6,k}(\theta)\right\|_{\infty}\notag\\
		\leq& C N^{\frac{1}{2}-m}\left\|\partial_\theta^m \varphi\left(\theta^{\frac{1}{\lambda}}\right)\right\|_{0, \omega^{\alpha+m,\beta+m,1}}+ C N^{-\kappa} \log N\|e(\theta)\|_{\infty}\notag\\
		\leq& C N^{\frac{1}{2}-m}\left\|\partial_\theta^m \varphi\left(\theta^{\frac{1}{\lambda}}\right)\right\|_{0, \omega^{\alpha+m,\beta+m,1}}\notag\\
		&+C N^{-\kappa} \log N\left(\|{e}^{*}(\theta)\|_{\infty}+\left\|E_8(\theta)\right\|_{\infty}+\left\|E_9(\theta)\right\|_{\infty}\right).\label{eq_E7_infty}
	\end{align}
In conclusion, from \eqref{eq_E1_infty} to \eqref{eq_E7_infty} together with \eqref{theorem_normcontrol_e*_infty} and \eqref{theorem_normcontrol_e_infty}, we can obtain the estimates for $\left\|e^*(\theta)\right\|_{\infty}$ and $\left\|e(\theta)\right\|_{\infty}$ as follows
		\begin{align}
		\left\|e^*(\theta)\right\|_{\infty}\leq &CN^{-m}\Bigl\{\log N\bigl(\mathcal{K}^*\left\|\varphi(\theta)\right\|_{\infty}
		+N^{\frac{1}{2}-\kappa}\left\|\partial_\theta^m \varphi\left(\theta^{\frac{1}{\lambda}}\right)\right\|_{0, \omega^{\alpha+m,\beta+m,1}}\bigr)+N^{\frac{1}{2}}\Phi\Bigr\},\notag\\
		\left\|e(\theta)\right\|_{\infty}\leq&CN^{-m}\Bigl\{\log N\bigl(\mathcal{K}^*\left\|\varphi(\theta)\right\|_{\infty}
		+N^{\frac{1}{2}-\kappa}\left\|\partial_\theta^m \varphi\left(\theta^{\frac{1}{\lambda}}\right)\right\|_{0, \omega^{\alpha+m,\beta+m,1}}\bigr)+N^{\frac{1}{2}}\Phi\Bigr\}.
	\end{align}
\end{proof}
Thereupon, we will derive the error estimate under the $L_{\omega^{\alpha,\beta,\lambda}}^{2}$-norm.

\subsection{Error estimate in $L_{\omega^{\alpha,\beta,\lambda}}^{2}$-norm}
\begin{theorem}\label{theorem_L2_norm}
If the hypotheses given in Theorem \ref{theorem_infty_norm} hold, then estimates of weighted $L^2$-norm behave as
	\begin{align}
		\left\|e^*(\theta)\right\|_{0,\omega^{\alpha,\beta,1}} \leq&
		CN^{-m}\Bigl\{\left(N^{-\kappa}\log N+1\right)\mathcal{K}^*\left\|\varphi(\theta)\right\|_{\infty} \notag\\
		&+\left(N^{\frac{1}{2}-\kappa}+1\right)\left\|\partial_\theta^m \varphi\left(\theta^{\frac{1}{\lambda}}\right)\right\|_{0, \omega^{\alpha+m,\beta+m,1}}
		+\left(N^{\frac{1}{2}-\kappa}+1\right)\Phi\Bigr\},\label{theorem_e*_L2}\\
		\left\|e(\theta)\right\|_{0,\omega^{\alpha,\beta,1}} \leq&
		 CN^{-m}\Bigl\{\left(N^{-\kappa}\log N+1\right)\mathcal{K}^*\left\|\varphi(\theta)\right\|_{\infty} \notag\\
		 &+\left(N^{\frac{1}{2}-\kappa}+2\right)\left\|\partial_\theta^m \varphi\left(\theta^{\frac{1}{\lambda}}\right)\right\|_{0, \omega^{\alpha+m,\beta+m,1}}
		 +\left(N^{\frac{1}{2}-\kappa}+1\right)\Phi\Bigr\}. \label{theorem_e_L2}
	\end{align}
\end{theorem}
\begin{proof}
By Lemmas \ref{lemma_Interpolation_L^2} and \ref{lemma_E6E7_L^2}, we can establish the following error analysis for the weighted $L^2$-norm, similar to the proof process in Theorem \ref{theorem_infty_norm}:
		\begin{align}
		\left\|E_1(\theta)\right\|_{0, \omega^{\alpha, \beta, \lambda}}&=\left\|\sum_{i=0}^N I_{i, 1} F_{i,\lambda}(\theta)\right\|_{0, \omega^{\alpha, \beta,\lambda}} \leq \max _{0 \leq i \leq N}\left|I_{i, 1}\right|\notag \\
		& \leq C N^{-m} \max\limits_{0\leq i\leq N}\left\|\partial_\theta^m \widetilde{K}_1\left(\theta_i,  \eta_i(\cdot)\right)\right\|_ {0,\omega^{m-\mu,{m+\frac{1}{\lambda}}-1,1}}\left(\|e(\theta)\|_{\infty}+\|\varphi(\theta)\|_{\infty}\right) \notag\\
		&\leq C N^{-m} \max\limits_{0\leq i\leq N}\left\|\partial_\theta^m \widetilde{K}_1\left(\theta_i,  \eta_i(\cdot)\right)\right\|_ {0,\omega^{m-\mu,{m+\frac{1}{\lambda}}-1,1}}\notag\\
		&\quad\left(\|{e}^{*}(\theta)\|_{\infty}+\left\|E_8(\theta)\right\|_{\infty}+\left\|E_9(\theta)\right\|_{\infty}+\|\varphi(\theta)\|_{\infty}\right),\label{eq_E1_L2}\\
		\left\|E_2(\theta)\right\|_{0, \omega^{\alpha, \beta,\lambda}} &\leq C N^{-m} \max\limits_{0\leq i\leq N}\left\|\partial_\theta^m \widetilde{K}_2\left(\theta_i, \varepsilon \eta_i(\cdot)\right)\right\|_ {0,\omega^{m-\mu,{m+\frac{1}{\lambda}}-1,1}}\left(\|e(\theta)\|_{\infty}+\|\varphi(\theta)\|_{\infty}\right)\notag\\
		&\leq C N^{-m} \max\limits_{0\leq i\leq N}\left\|\partial_\theta^m \widetilde{K}_2\left(\theta_i, \varepsilon \eta_i(\cdot)\right)\right\|_ {0,\omega^{m-\mu,{m+\frac{1}{\lambda}}-1,1}}\notag\\
		&\quad\left(\|{e}^{*}(\theta)\|_{\infty}+\left\|E_8(\theta)\right\|_{\infty}+\left\|E_9(\theta)\right\|_{\infty}+\|\varphi(\theta)\|_{\infty}\right),\label{eq_E2_L2}
	\end{align}
	\begin{align}
		\left\|E_3(\theta)\right\|_{0, \omega^{\alpha, \beta, \lambda}} &\leq C N^{-m}\left\|\partial_\theta^{m+1} \varphi\left(\theta^{\frac{1}{\lambda}}\right)\right\|_{0, \omega^{\alpha+m,\beta+m,1}},\label{eq_E3_L2}\\
		\left\|E_8(\theta)\right\|_{0, \omega^{\alpha, \beta,\lambda}} &\leq C N^{-m} \left\|\partial_\theta^{m} \varphi\left(\theta^{\frac{1}{\lambda}}\right)\right\|_{0, \omega^{\alpha+m,\beta+m,1}},\label{eq_E8_L2}\\
		\left\|E_4(\theta)\right\|_{0, \omega^{\alpha, \beta,\lambda}} &\leq C N^{-1} \left\| e^*\left(\theta^{\frac{1}{\lambda}}\right)\right\|_{0, \omega^{\alpha, \beta,\lambda}} \leq C N^{-1}\left\| e^*(\theta)\right\|_{\infty},\label{eq_E4_L2}
	\end{align}
	\begin{align}
		\left\|E_5(\theta)\right\|_{0, w^{\alpha, \beta, \lambda}} &\leq C N^{-1}\left\|e^*(\theta)\right\|_{\infty},\label{eq_E5_L2}\\
		\left\|E_9(\theta)\right\|_{0, \omega^{\alpha, \beta,\lambda}}&\leq C N^{-1}\left\|e^*(\theta)\right\|_{\infty}\label{eq_E9_L2},\\
		\left\|E_{6,1}(\theta)\right\|_{0, \omega^{\alpha, \beta,\lambda}}&\leq C N^{-m}\left\|\partial_{\theta}^m\varphi\left(\theta^{\frac{1}{\lambda}}\right)\right\|_{0, \omega^{\alpha+m, \beta+m, 1}},\label{eq_E61_L2}
	\end{align}
	\begin{align}
		\left\|E_{6,3}(\theta)\right\|_{0, \omega^{\alpha, \beta,\lambda}}&=\left\|\left(I_{\beta, \lambda}^{\alpha, \beta}-I\right)\left(\mathcal{\hat{K}}_1 e\right)(\theta)\right\|_{0, \omega^{\alpha, \beta, \lambda}} \notag\\
		& =\left\|\left(I_{N, 1}^{\alpha, \beta}-I\right)\left(\mathcal{\hat{K}}_1 e\right)\left(z^{\frac{1}{\lambda}}\right) \right\|_{0, \omega^{\alpha, \beta, 1}}\notag\\
		& =\left\|\left(I_{N, 1}^{\alpha, \beta}-I\right)\left(\mathcal{\hat{K}}_1 e-\mathcal{T}_N \mathcal{\hat{K}}_1 e\right)\left(z^{\frac{1}{\lambda}}\right)\right\|_{0, w^{\alpha, \beta,1}} \notag\\
		& \leq \left\|\left(I_{N, 1}^{\alpha, \beta}\left(\mathcal{\hat{K}}_1e-\mathcal{T}_N \mathcal{\hat{K}}_1 e\right)\left(z^{\frac{1}{\lambda}}\right) \right\|_{0, \omega^{\alpha, \beta,1}}\right. +\left\|\left(\mathcal{\hat{K}}_1 e-\mathcal{T}_N \mathcal{\hat{K}}_1 e\right)\left(z^{\frac{1}{\lambda}}\right)\right\|_{0, \omega^{\alpha, \beta,1}}\notag\\
		& \leq C \left\|\left(\mathcal{\hat{K}}_1 e-\mathcal{T}_N \mathcal{\hat{K}}_1 e\right)\left(z^{\frac{1}{\lambda}}\right) \right\|_{\infty}\notag\\
		& \leq C N^{-\kappa} \left\| \mathcal{\hat{K}}_1 e\left(z^{\frac{1}{\lambda}}\right) \right\|_{0, \kappa} \leq C N^{-\kappa}\left\| e(\theta) \right\|_{\infty},\label{eq_E63_L2}
	\end{align}
	\begin{align}
		\left\|E_{7,3}(\theta)\right\|_{0, \omega^{\alpha, \beta,\lambda}}& \leq C N^{-\kappa}\|e(\theta)\|_{\infty},\label{eq_E73_L2}\\
		\left\|E_{6}(\theta)\right\|_{0, \omega^{\alpha, \beta,\lambda}}& \leq CN^{-m}\left\|\partial_{\theta}^m\varphi\left(\theta^{\frac{1}{\lambda}}\right)\right\|_{0, \omega^{\alpha+m, \beta+m, 1}}+ C N^{-\kappa}\notag\\
		&\quad\left(\|{e}^{*}(\theta)\|_{\infty}+\left\|E_8(\theta)\right\|_{\infty}+\left\|E_9(\theta)\right\|_{\infty}\right),\label{eq_E6_L2}\\
		\left\|E_{7}(\theta)\right\|_{0, \omega^{\alpha, \beta,\lambda}}& \leq CN^{-m}\left\|\partial_{\theta}^m\varphi\left(\theta^{\frac{1}{\lambda}}\right)\right\|_{0, \omega^{\alpha+m, \beta+m, 1}}+ C N^{-\kappa}\notag\\
		&\quad\left(\|{e}^{*}(\theta)\|_{\infty}+\left\|E_8(\theta)\right\|_{\infty}+\left\|E_9(\theta)\right\|_{\infty}\right).\label{eq_E7_L2}
	\end{align}
The error estimates for $\left\|e^*(\theta)\right\|_{0,\omega^{\alpha,\beta,1}}$ and $\left\|e(\theta)\right\|_{0,\omega^{\alpha,\beta,1}}$ are determined by \eqref{eq_E1_L2}-\eqref{eq_E7_L2},  together with \eqref{theorem_normcontrol_e*_w}, \eqref{theorem_normcontrol_e_w},  \eqref{eq_e_e*} and \eqref{eq_E8_infty}, yield
		\begin{align}
		\left\|e^*(\theta)\right\|_{0,\omega^{\alpha,\beta,1}} \leq&
		CN^{-m}\Bigl\{\left(N^{-\kappa}\log N+1\right)\mathcal{K}^*\left\|\varphi(\theta)\right\|_{\infty} \notag\\
		&+\left(N^{\frac{1}{2}-\kappa}+1\right)\left\|\partial_\theta^m \varphi\left(\theta^{\frac{1}{\lambda}}\right)\right\|_{0, \omega^{\alpha+m,\beta+m,1}}
		+\left(N^{\frac{1}{2}-\kappa}+1\right)\Phi\Bigr\},\\
		\left\|e(\theta)\right\|_{0,\omega^{\alpha,\beta,1}} \leq&
		CN^{-m}\Bigl\{\left(N^{-\kappa}\log N+1\right)\mathcal{K}^*\left\|\varphi(\theta)\right\|_{\infty} \notag\\
		&+\left(N^{\frac{1}{2}-\kappa}+2\right)\left\|\partial_\theta^m \varphi\left(\theta^{\frac{1}{\lambda}}\right)\right\|_{0, \omega^{\alpha+m,\beta+m,1}}
		+\left(N^{\frac{1}{2}-\kappa}+1\right)\Phi\Bigr\}. 
	\end{align}
\end{proof}
Finally, we have completed the theoretical proof for $L^{\infty}$-norm and the weighted $L^2$-norm.

\section{Numerical experiments}\label{section_Numerical experiments}
In this section, we will verify the numerical method and error estimates constructed in the previous section through some numerical examples. All the errors will be presented in the $L^{\infty}(I)$ and $L^{2}_{\omega^{\alpha,\beta,\lambda}}(I)$-norm in the following text, where $\alpha$ and $\beta$ are related coefficients. 
The primary objective of these numerical results is to demonstrate that the proposed method exhibits the potential to achieve high accuracy when the parameter $\lambda$ is chosen such that the functions $y\left(t^{\frac{1}{\lambda}}\right)$ and $y^{\prime}\left(t^{\frac{1}{\lambda}}\right)$ behave smoothly enough.
The left figures of the numerical experiments below show $\left\|e(\theta)\right\|_{0,\omega^{\alpha,\beta,1}}$ and $\|e(\theta)\|_{\infty}$, while the right ones show $\left\|e^*(\theta)\right\|_{0,\omega^{\alpha,\beta,1}}$ and $\|e^*(\theta)\|_{\infty}$.
\begin{example}\label{example_third_NO1}
Consider the following problem: 
	\begin{equation}
		\left\{\begin{array}{l}
			t^\gamma y^{\prime}(t)=t^{\frac{5}{3}} y(t)+t^{\frac{5}{3}} y(\varepsilon t)+g(t)+\int_0^t(t-s)^{-\mu}s^{\mu+\gamma-1} e^{s^{1-\mu}} y(s) d s\notag\\
			\qquad\qquad+\varepsilon^{-\gamma}\int_0^t(t-\tau)^{-\mu}\tau^{\mu+\gamma-1} e^{\tau^{1-\mu}} y(\tau) d \tau, \quad t \in I, \\
			y(0)=0.
		\end{array}\right.
	\end{equation}
 where $g(t)=\left(1-(1-\mu) t^{1-\mu}+t\right) e^{-t^{1-\mu}}+(1+e)B(1-\mu, \mu+\gamma-1)t-t^{\frac{8}{3}}e^{-t^{1-\mu}}-t^{\frac{5}{3}}(\varepsilon t)e^{-(\varepsilon t)^{1-\mu}}, B(\cdot, \cdot)$ for the Beta function.
 
In this example, the exact solution is $y(t)=t e^{-t^{1-\mu}}$. 
We take $\mu=\frac{1}{2}, T=1, \varepsilon=0.5$. Through transformation $t=\theta^{\frac{1}{\lambda}}$, $y\left(\theta^{\frac{1}{\lambda}}\right)$ and $y^{\prime}\left(\theta^{\frac{1}{\lambda}}\right)$ are both analytical even if $y(t)$ and $y^{\prime}(t)$ are weakly
 singular at $t = 0$. 
 
 In Figures \ref{Fig_third_NO1_0.5} and \ref{Fig_third_NO1_1}, we can see the errors and convergence results intuitively. Besides,Tables \ref{tabular_third_NO1_e} and \ref{tabular_third_NO1_e*} show the specific datas of the error bounds under different norms.
When $\lambda=1$, the namely polynomial collocation method can only arrive at algebraic convergence, with a minimum error of $10^{-3.5}$ magnitude at $N=20$. Meanwhile, exponential convergence can be easily achieved, with a minimum error of $4.91123\times 10^{-12}$ at $N=12$ when $\lambda=\frac{1}{2}$. It is obvious that fractional collocation method significantly outperforms polynomial collocation method in terms of accuracy.

		\begin{figure}
		\subfloat[$\lambda=\frac{1}{2}:\left\|e(\theta)\right\|_{0,\omega^{\alpha,\beta,1}},\|e(\theta)\|_{\infty}$]{\includegraphics[width=0.5\textwidth]{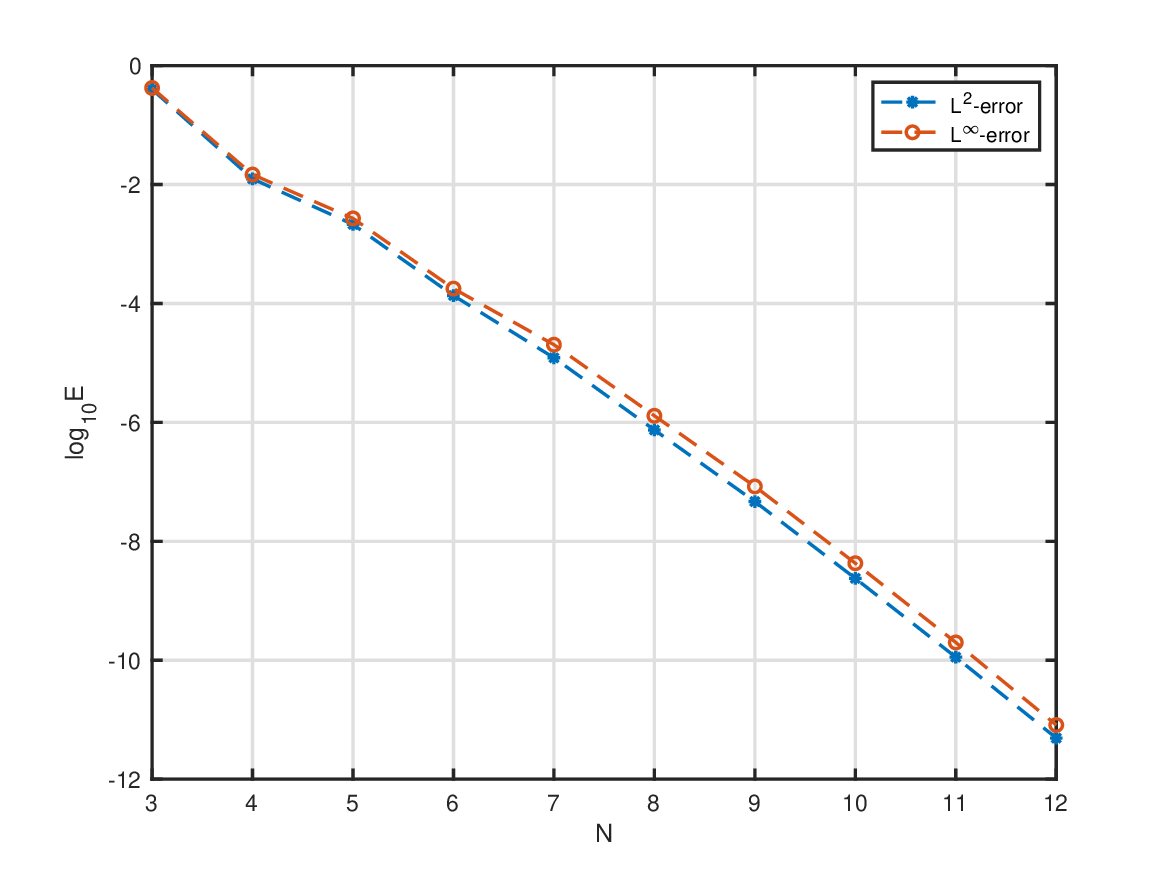}}%
		\hfill
		\subfloat[$\lambda=\frac{1}{2}:\left\|e^*(\theta)\right\|_{0,\omega^{\alpha,\beta,1}},\|e^*(\theta)\|_{\infty}$]{\includegraphics[width=0.5\textwidth]{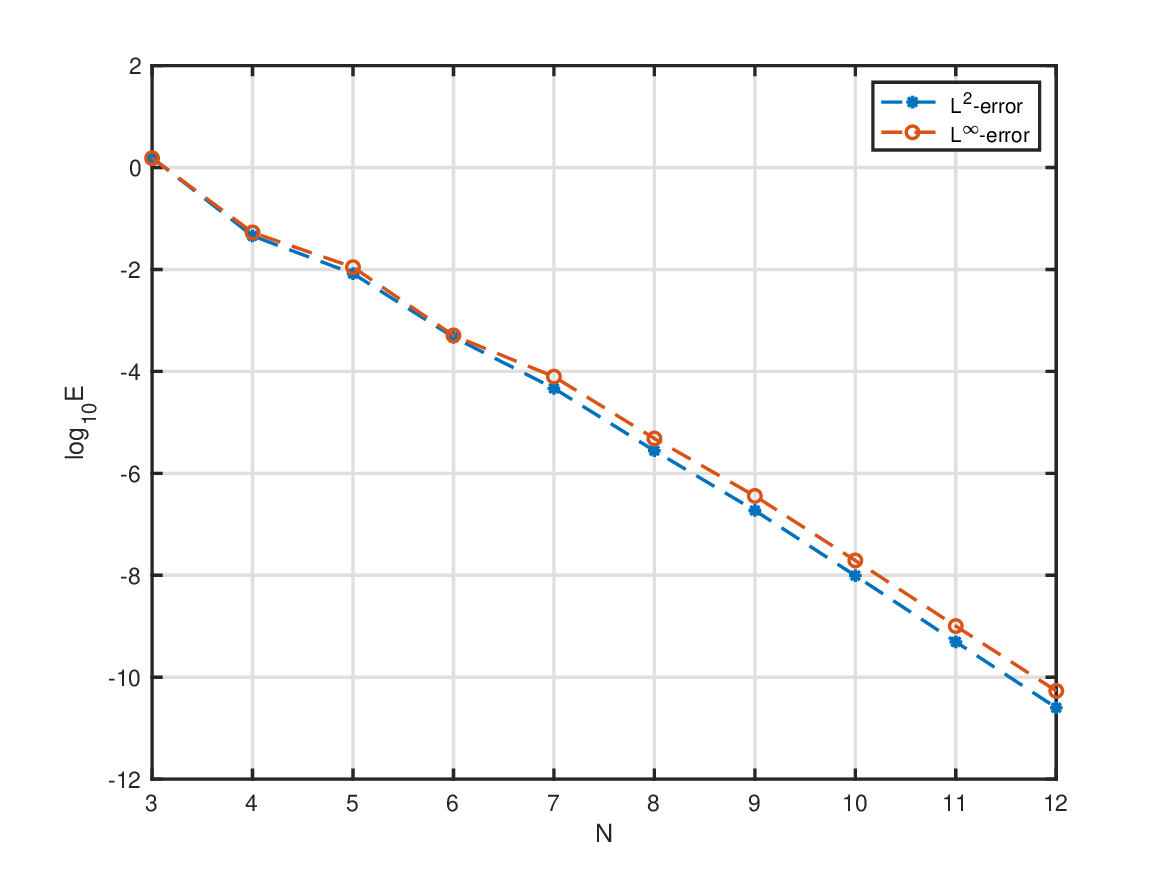}}
		\caption{\bf{Example \ref{example_third_NO1}} with $\lambda=\frac{1}{2}$}
		\label{Fig_third_NO1_0.5}
	\end{figure}
	
	\begin{figure}
		\subfloat[$\lambda=1:\left\|e(\theta)\right\|_{0,\omega^{\alpha,\beta,1}},\|e(\theta)\|_{\infty}$]{\includegraphics[width=0.5\textwidth]{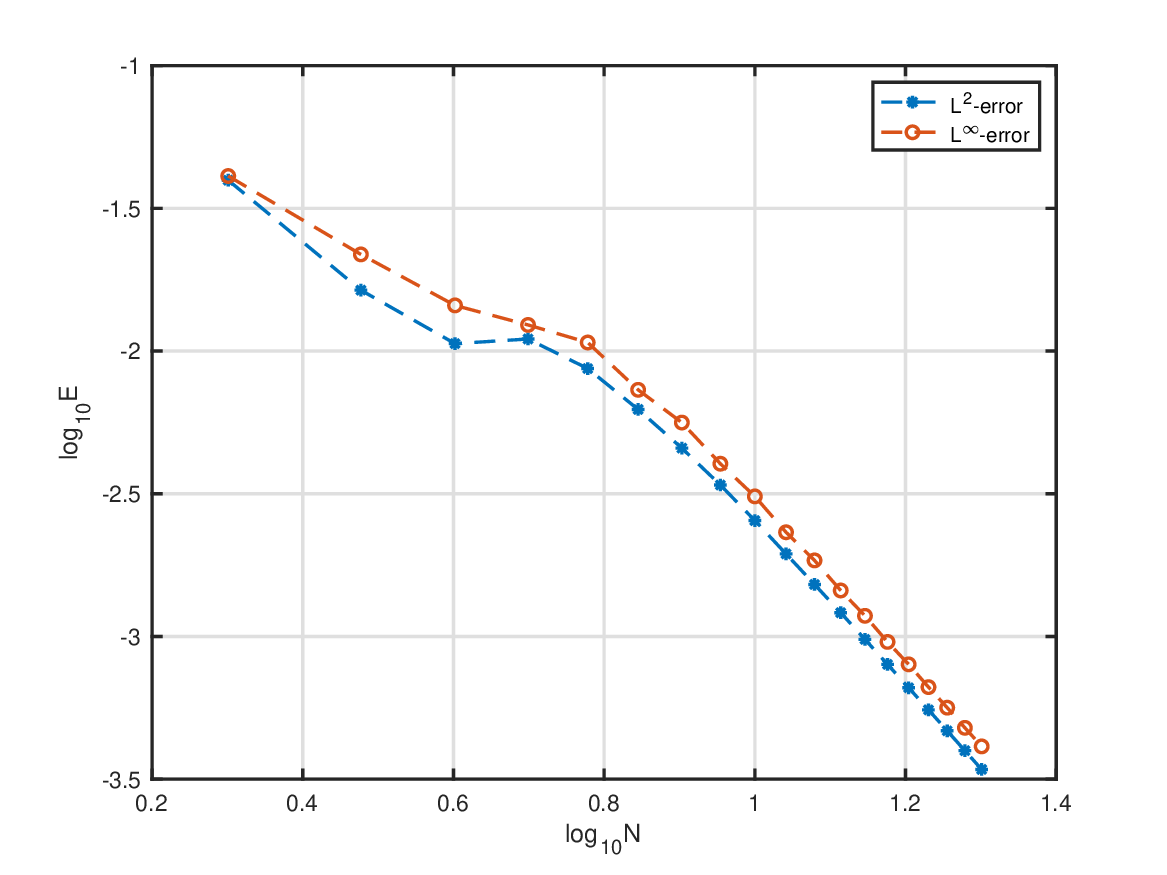}}%
		\hfill
		\subfloat[$\lambda=1:\left\|e^*(\theta)\right\|_{0,\omega^{\alpha,\beta,1}},\|e^*(\theta)\|_{\infty}$]{\includegraphics[width=0.5\textwidth]{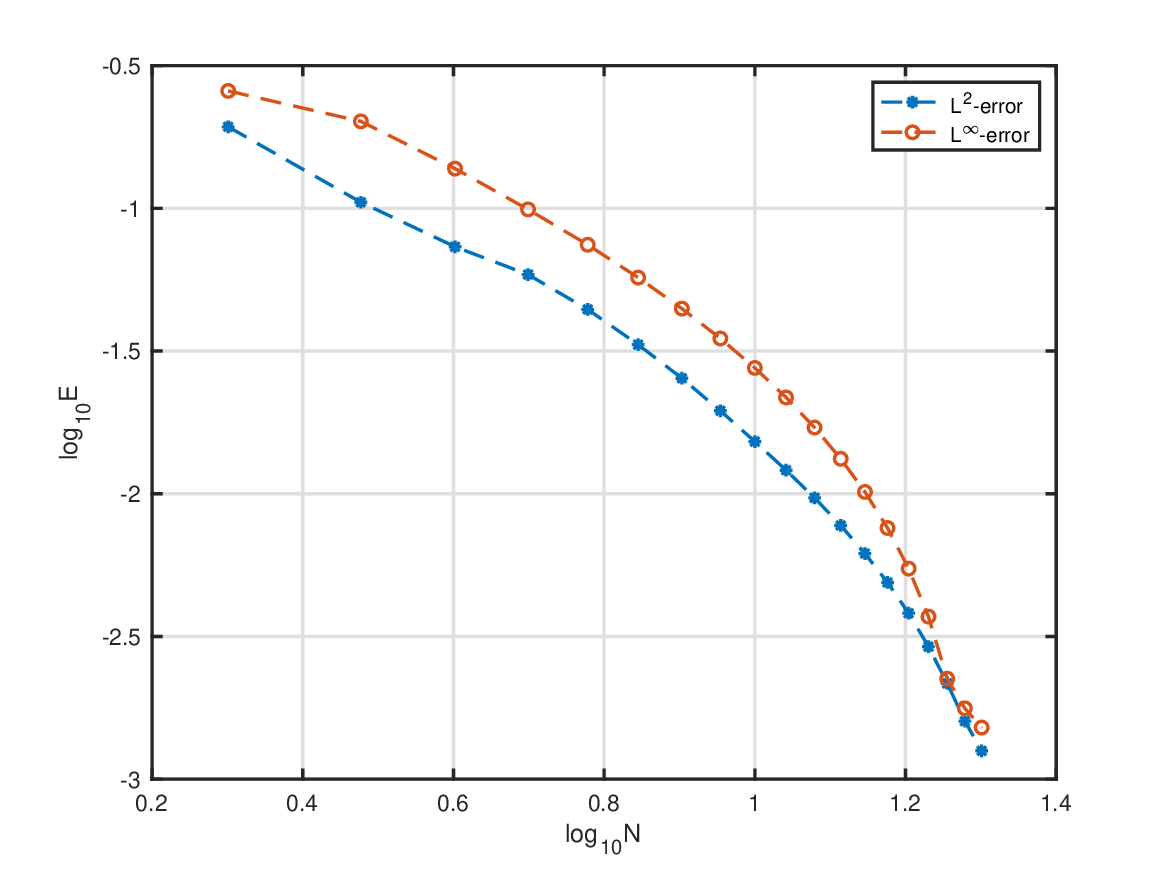}}
		\caption{\bf{Example \ref{example_third_NO1}} with $\lambda=1$}
		\label{Fig_third_NO1_1}
	\end{figure}
	
	\begin{table}[!ht]
		\centering
		\caption{\bf{Example \ref{example_third_NO1}} with $\lambda=\frac{1}{2}$: $\left\|e(\theta)\right\|_{0,\omega^{\alpha,\beta,1}}$ and $\|e(\theta)\|_{\infty}$.}
		\begin{tabular}{llllll}
			\hline$N$& 4 & 6 & 8 & 10 & 12\\
			\hline$L^2$-error  & $1.24718 \mathrm{e}-02$ & $1.35898 \mathrm{e}-04$ & $7.46163 \mathrm{e}-07$ & $2.38088 \mathrm{e}-09$ & $4.91123 \mathrm{e}-12$\\
			$L^{\infty}$-error  & $1.47322\mathrm{e}-02$ & $1.78065\mathrm{e}-04$ & $1.29258 \mathrm{e}-06$ & $4.27947\mathrm{e}-09$ &$8.15620 \mathrm{e}-12$\\
			\hline
		\end{tabular}
		\label{tabular_third_NO1_e}
	\end{table}
	
	\begin{table}[!ht]
		\centering
		\caption{\bf{Example \ref{example_third_NO1}} with $\lambda=\frac{1}{2}$: $\left\|e^*(\theta)\right\|_{0,\omega^{\alpha,\beta,1}}$ and $\|e^*(\theta)\|_{\infty}$.}
		\begin{tabular}{llllll}
			\hline$N$& 4 & 6 & 8 & 10 & 12\\
			\hline$L^2$-error  & $4.57408\mathrm{e}-02$ & $4.69971 \mathrm{e}-04$ & $2.78318 \mathrm{e}-06$ & $9.84189\mathrm{e}-09$ & $2.52306 \mathrm{e}-11$\\
			$L^{\infty}$-error& $5.32377\mathrm{e}-02$ & $5.07685\mathrm{e}-04$ & $4.82471\mathrm{e}-06$ & $1.95348 \mathrm{e}-08$ & $5.39264 \mathrm{e}-11$\\
			\hline
		\end{tabular}
		\label{tabular_third_NO1_e*}
	\end{table}
\end{example}

\begin{example}\label{example_third_NO2}
Consider the following linear VIDEs:
	\begin{equation}
	\left\{\begin{array}{l}
		t^\gamma y^{\prime}(t)=t^{\frac{5}{3}} y(t)+t^{\frac{5}{3}} y(\varepsilon t)+g(t)
		+\int_0^t \frac{\sqrt{3}}{3 \pi}(t-s)^{-\mu} s^{\mu+\gamma-1} e^s y(s) d s\notag\\
		\qquad\qquad+{\varepsilon}^{-\gamma}\int_0^{\varepsilon t} \frac{\sqrt{3}}{3 \pi}(\varepsilon t-\tau)^{-\mu} \tau^{\mu+\gamma-1} e^{\tau} y(\tau) d \tau, \quad t \in[0,T],  \\
		y(0)=0 .
	\end{array}\right.
	\end{equation}
where $g$ is a given function.

Let $\varepsilon = 0.66, \mu = \frac{1}{3}, \gamma = 1, T=\frac{1}{2}$, we set $g(t)$ such that the exact solution is $y(t) = t^{1+\mu}e^{-t}$, where $\mu \in (0,1)$, $y(t)$ has a weakly singularity at $t = 0^+$. In fact, considering the structure of the solution, we need to choose $\lambda=\frac{1}{3}$ so that $y\left(t^{\frac{1}{\lambda}}\right)$ and $y^{\prime}\left(t^{\frac{1}{\lambda}}\right)$ meet the requirements. Numerical convegence is shown in Figures \ref{Fig_third_NO2_0.5} and \ref{Fig_third_NO2_1}. Tables \ref{tabular_third_NO2_e} exhibit the errors reaching $10^{-8}$. The numerical results match the theorem analysis.
		\begin{figure}
		\subfloat[$\lambda=\frac{1}{3}:\left\|e(\theta)\right\|_{0,\omega^{\alpha,\beta,1}},\|e(\theta)\|_{\infty}$]{\includegraphics[width=0.5\textwidth]{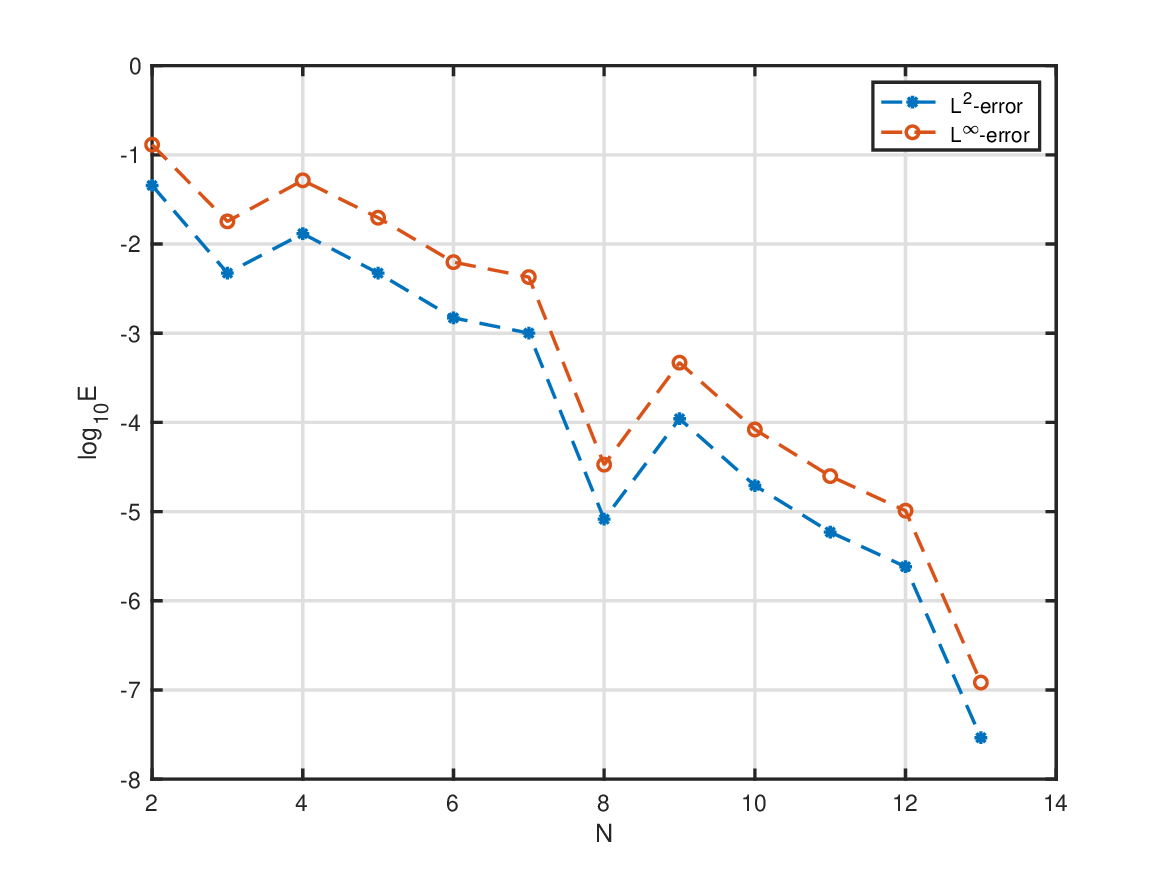}}%
		\hfill
		\subfloat[$\lambda=\frac{1}{3}:\left\|e^*(\theta)\right\|_{0,\omega^{\alpha,\beta,1}},\|e^*(\theta)\|_{\infty}$]{\includegraphics[width=0.5\textwidth]{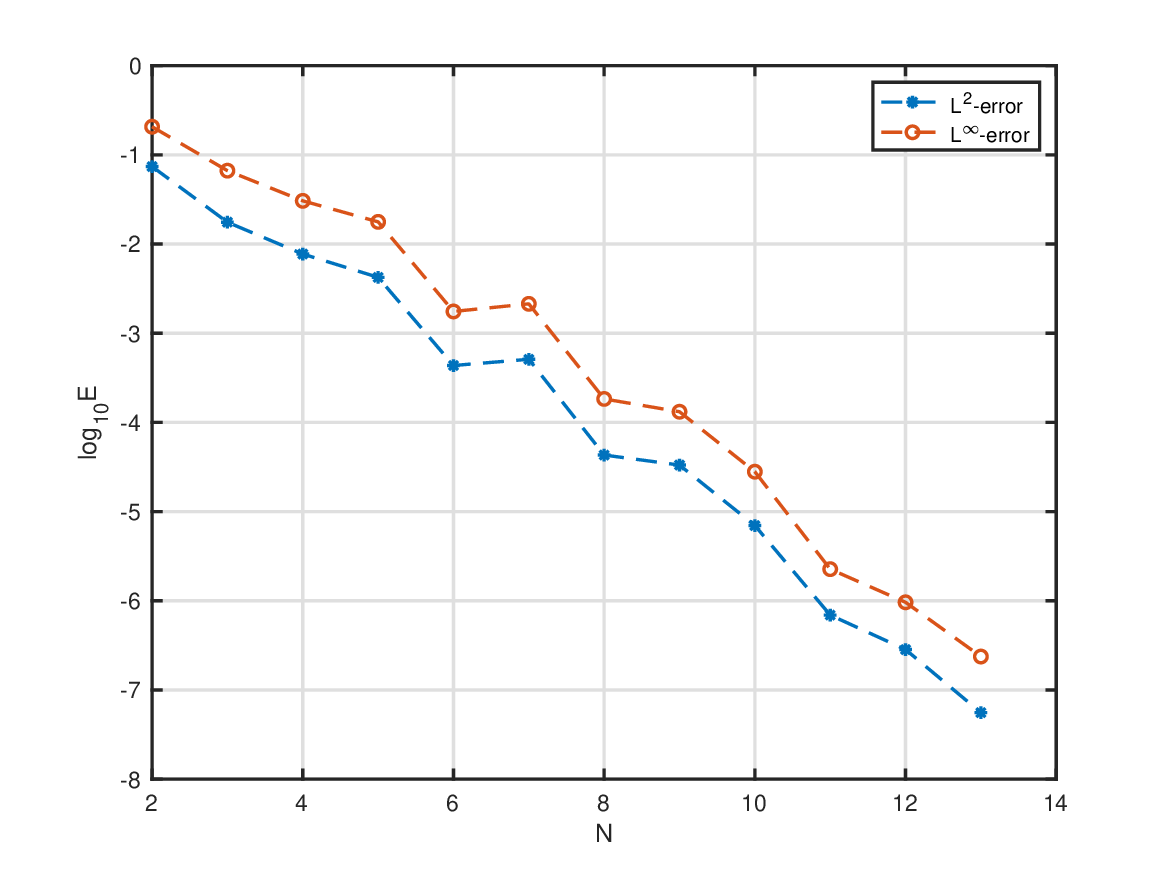}}
		\caption{\bf{Example \ref{example_third_NO2} } with $\lambda=\frac{1}{3}$}
		\label{Fig_third_NO2_0.5}
	\end{figure}
	
	\begin{figure}
		\subfloat[$\lambda=1:\left\|e(\theta)\right\|_{0,\omega^{\alpha,\beta,1}},\|e(\theta)\|_{\infty}$]{\includegraphics[width=0.5\textwidth]{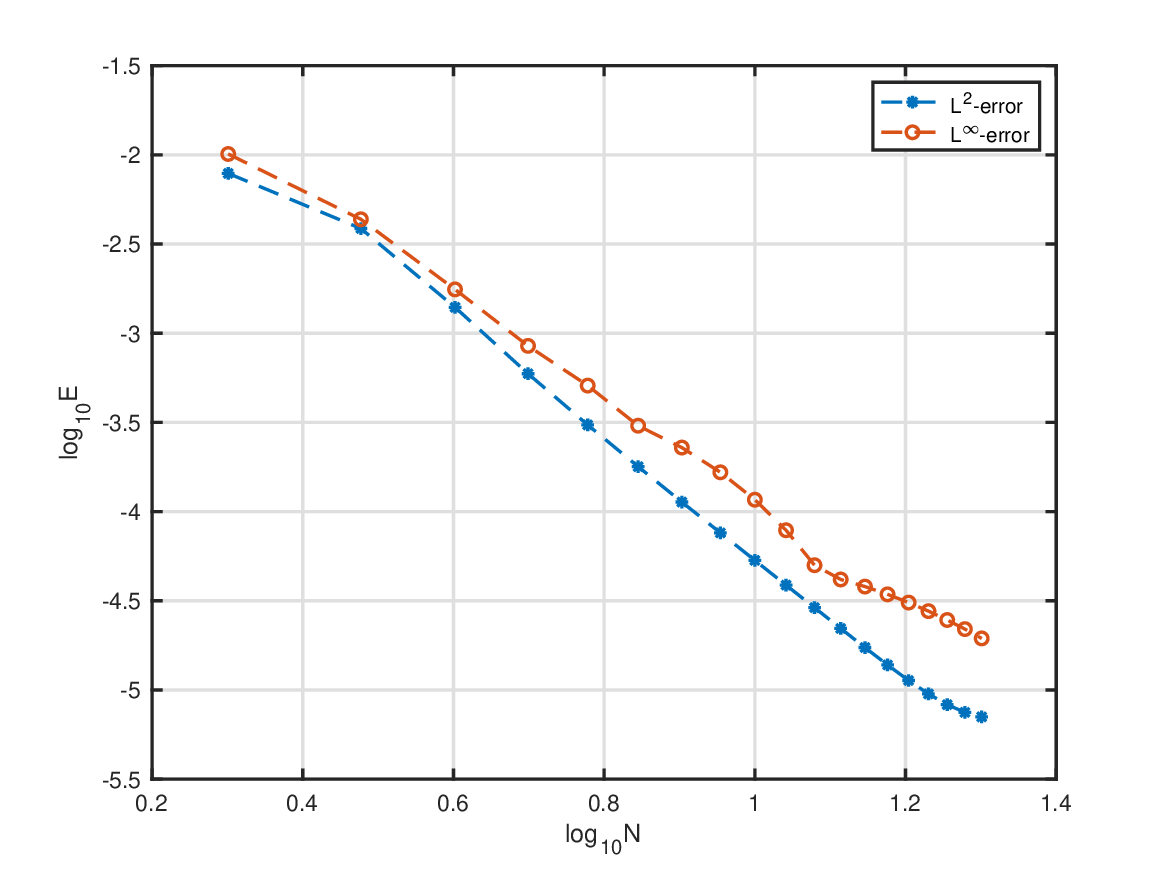}}%
		\hfill
		\subfloat[$\lambda=1:\left\|e^*(\theta)\right\|_{0,\omega^{\alpha,\beta,1}},\|e^*(\theta)\|_{\infty}$]{\includegraphics[width=0.5\textwidth]{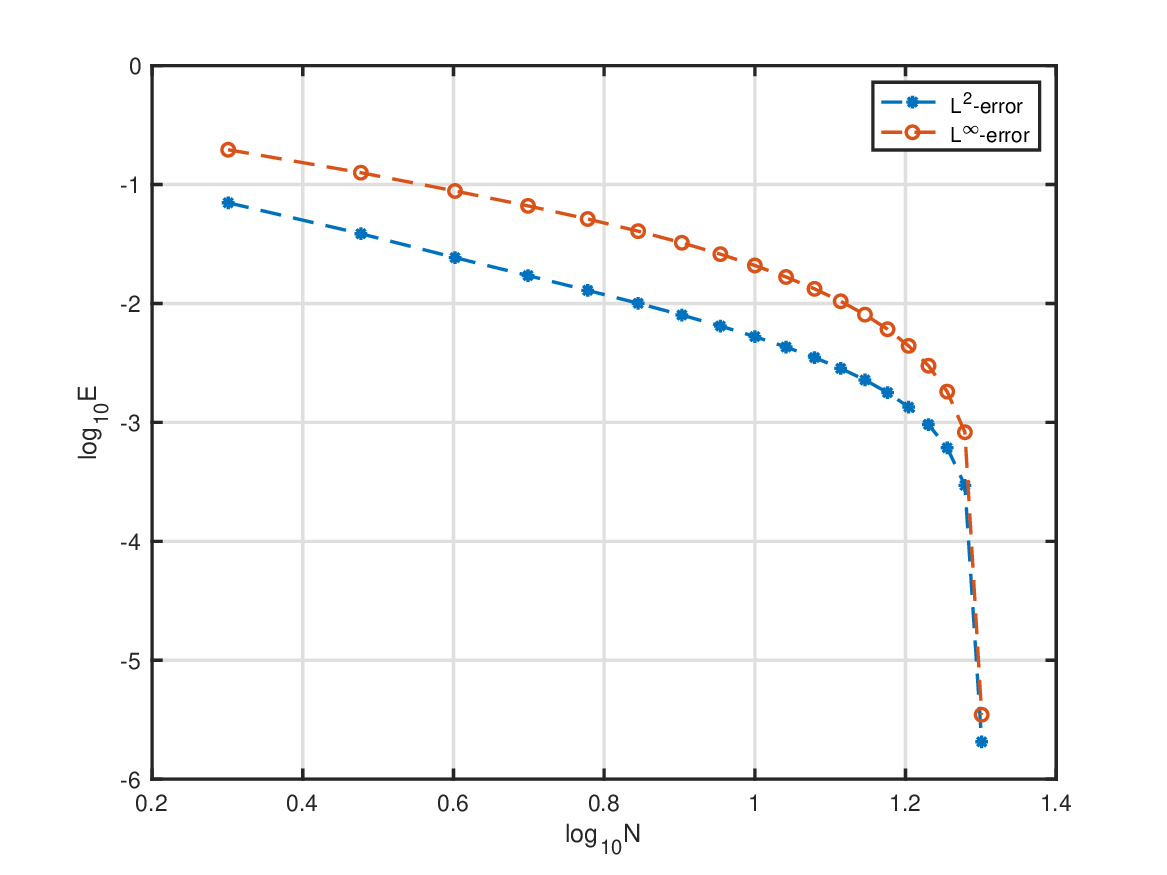}}
		\caption{\bf{Example \ref{example_third_NO2} } with $\lambda=1$}
		\label{Fig_third_NO2_1}
	\end{figure}
	
	\begin{table}[!ht]
		\centering
		\caption{\bf{Example \ref{example_third_NO2} } with $\lambda=\frac{1}{2}$: $\left\|e(\theta)\right\|_{0,\omega^{\alpha,\beta,1}}$ and $\|e(\theta)\|_{\infty}$.}
		\begin{tabular}{llllll}
			\hline$N$ & 2 & 5 & 10 & 11 & 13\\
			\hline$L^2$-error & $4.53443 \mathrm{e}-02$ & $4.71548 \mathrm{e}-03$ & $1.95915 \mathrm{e}-05$ & $5.88895 \mathrm{e}-06$ & $2.91639 \mathrm{e}-08$\\
			$L^{\infty}$-error & $1.29846\mathrm{e}-01$ & $1.97082\mathrm{e}-02$ & $8.32625\mathrm{e}-05$ & $2.49951 \mathrm{e}-05$ &$1.21192\mathrm{e}-07$\\
			\hline
		\end{tabular}
		\label{tabular_third_NO2_e}
	\end{table}
	
	\begin{table}[!ht]
		\centering
		\caption{\bf{Example \ref{example_third_NO2}} with  $\lambda=\frac{1}{2}$: $\left\|e^*(\theta)\right\|_{0,\omega^{\alpha,\beta,1}}$ and $\|e^*(\theta)\|_{\infty}$.}
		\begin{tabular}{llllll}
			\hline$N$& 2 & 5 & 10 & 11 & 13\\
			\hline$L^2$-error & $7.38681\mathrm{e}-02$ & $4.23023\mathrm{e}-03$ & $6.97779\mathrm{e}-06$ & $6.91387\mathrm{e}-07$ & $5.57349\mathrm{e}-08$\\
			$L^{\infty}$-error  & $2.06450\mathrm{e}-01$ & $1.77134\mathrm{e}-02$ & $2.79917\mathrm{e}-05$ & $2.26236\mathrm{e}-06$ &$2.37082\mathrm{e}-07$\\
			\hline
		\end{tabular}
		\label{tabular_third_NO2_e*}
	\end{table}
\end{example}

\begin{example}\label{example_third_NO3}
Continue to consider the equation in Example \ref{example_third_NO2}, with the given function $g$ be
	$$
	\begin{aligned}
		g(t)=&e^{-t}\left((t^{w_1}(1+w_1-t)+t^{w_2}(1+w_2-t))\right)
				-t^{\frac{5}{3}-\gamma}(u(t)+u(\varepsilon t))\\
			  	&-\frac{\sqrt3}{3\pi}B(1-\mu,w_1+\mu+\gamma-1)t^{1+w_1}(e^{1+w_1}+1)\\
				&-\frac{\sqrt3}{3\pi}B(1-\mu,w_2+\mu+\gamma-1)t^{1+w_2}(e^{1+w_2}+1).
	\end{aligned}
	$$
where the exact solution $y(t)=(t^{1+w_1}+t^{1+w_2})e^{-t}$ is  constructed for a more complicated situation.

As is shown in the equation, we cannot always keep $y\left(t^{\frac{1}{\lambda}}\right)$ and $y^{\prime}\left(t^{\frac{1}{\lambda}}\right)$ analytic with the choice of $\lambda$. This example can test the situation that  $y\left(t^{\frac{1}{\lambda}}\right)$ and $y^{\prime}\left(t^{\frac{1}{\lambda}}\right)$ behave not smoothly enough.

Let $\varepsilon$, $ \mu$, $ \gamma$ and $\lambda$ be the same as Example \ref{example_third_NO2}, then $T=1, w_1=\frac{1}{2}, w_2=\sqrt{2}$. Comparing the numerical results under $\lambda=0.5$ and $\lambda=1$ in Figures \ref{Fig_third_NO3_0.5} and \ref{Fig_third_NO3_1}, it is obvious to see that fractional collocation method has better convergence result than the polynomial one. 

 Besides, the convergence results in Figure \ref{Fig_third_NO3_0.5} and Table \ref{tabular_third_NO3_e} behave better than the ones in Figure \ref{Fig_third_NO2_0.5} and Table \ref{tabular_third_NO2_e} although this example is more complicated and the interval expansion, we infer the reason is that $\mu=\frac{1}{2}, \lambda=\frac{1}{2}$ is more suitable for the proposed method than $\mu=\frac{1}{3}, \lambda=\frac{1}{3}$.
		\begin{figure}
		\subfloat[$\lambda=\frac{1}{2}:\left\|e(\theta)\right\|_{0,\omega^{\alpha,\beta,1}},\|e(\theta)\|_{\infty}$]{\includegraphics[width=0.5\textwidth]{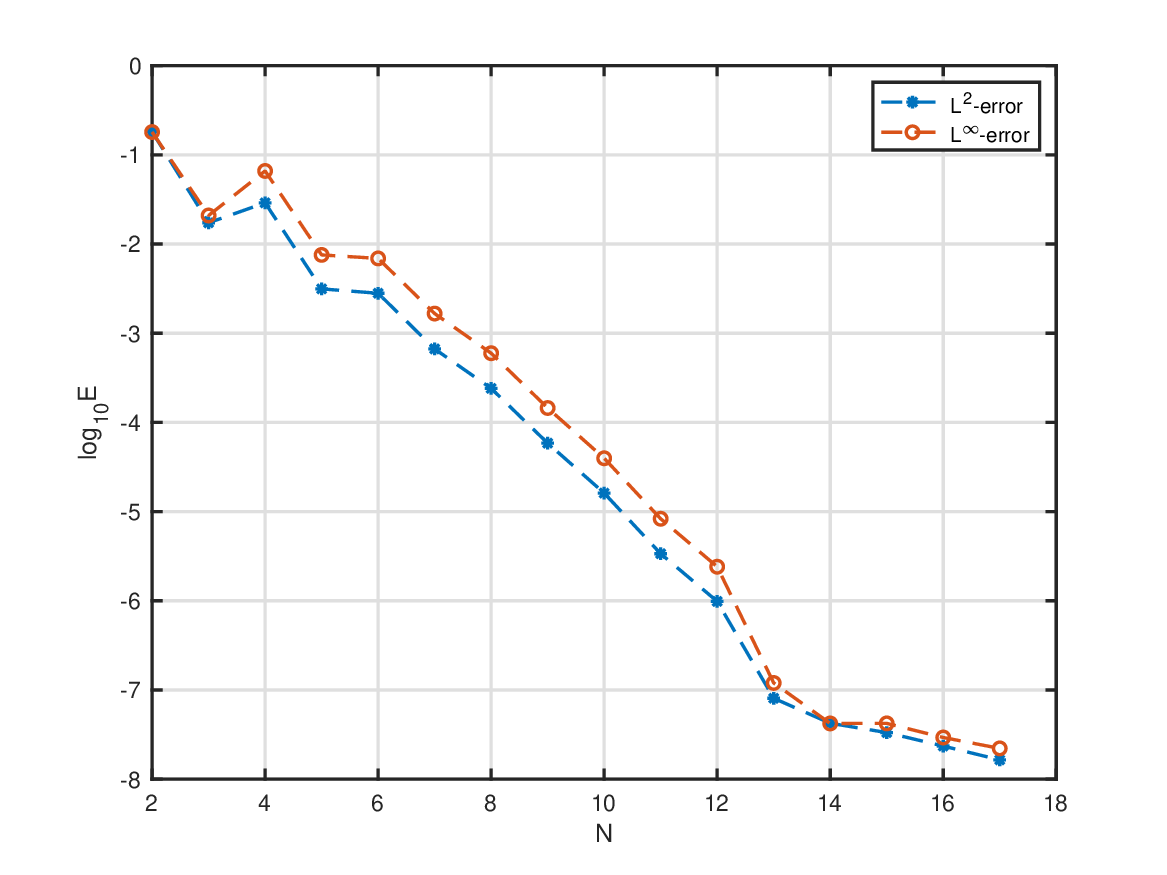}}%
		\hfill
		\subfloat[$\lambda=\frac{1}{2}:\left\|e^*(\theta)\right\|_{0,\omega^{\alpha,\beta,1}},\|e^*(\theta)\|_{\infty}$]{\includegraphics[width=0.5\textwidth]{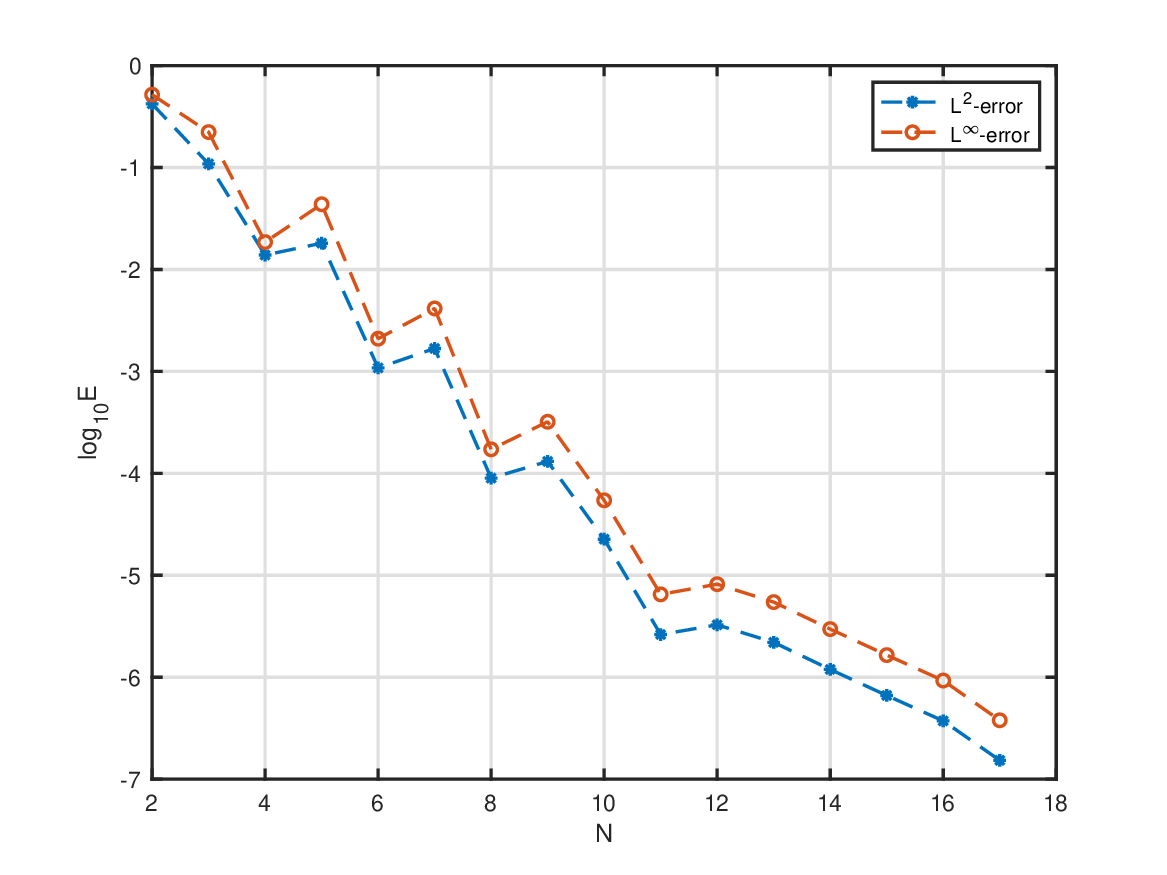}}
		\caption{\bf{Example \ref{example_third_NO3} } with $\lambda=\frac{1}{2}$}
		\label{Fig_third_NO3_0.5}
	\end{figure}
	
	\begin{figure}
		\subfloat[$\lambda=1:\left\|e(\theta)\right\|_{0,\omega^{\alpha,\beta,1}},\|e(\theta)\|_{\infty}$]{\includegraphics[width=0.5\textwidth]{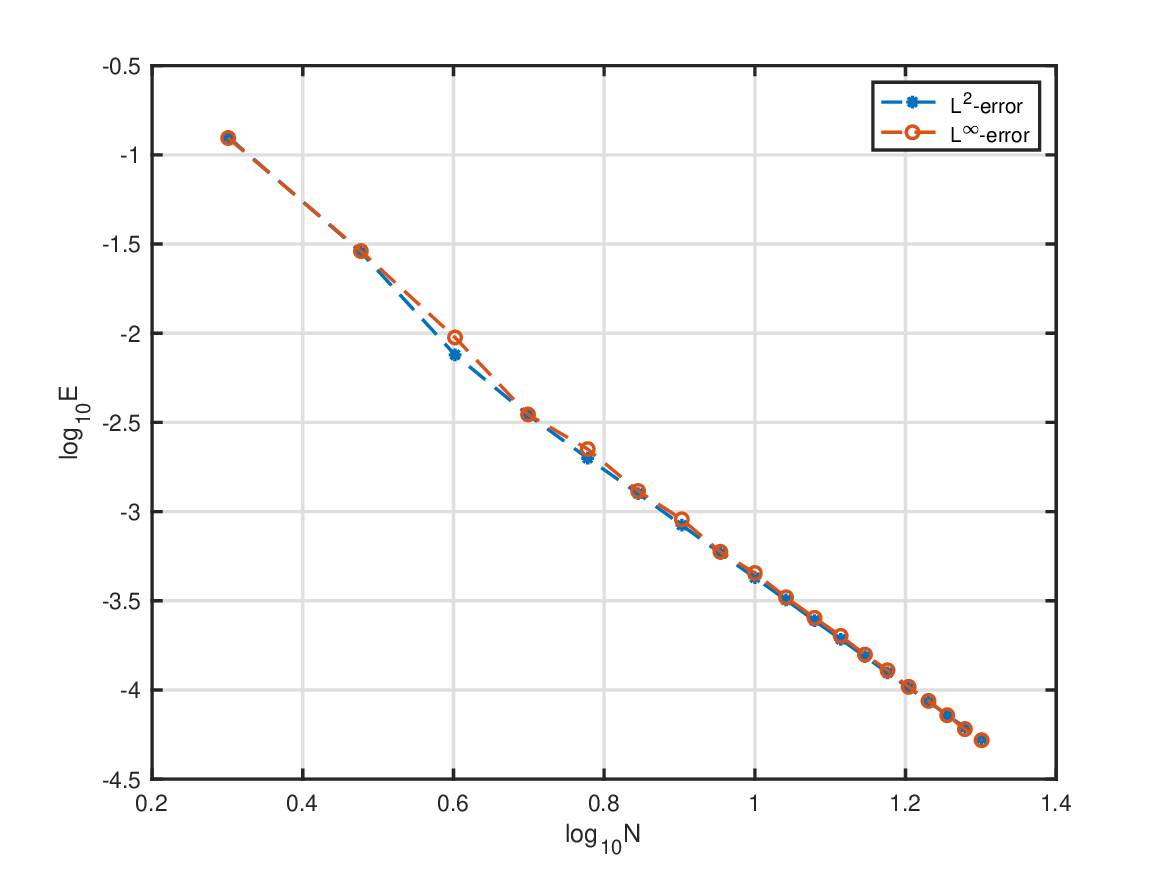}}%
		\hfill
		\subfloat[$\lambda=1:\left\|e^*(\theta)\right\|_{0,\omega^{\alpha,\beta,1}},\|e^*(\theta)\|_{\infty}$]{\includegraphics[width=0.5\textwidth]{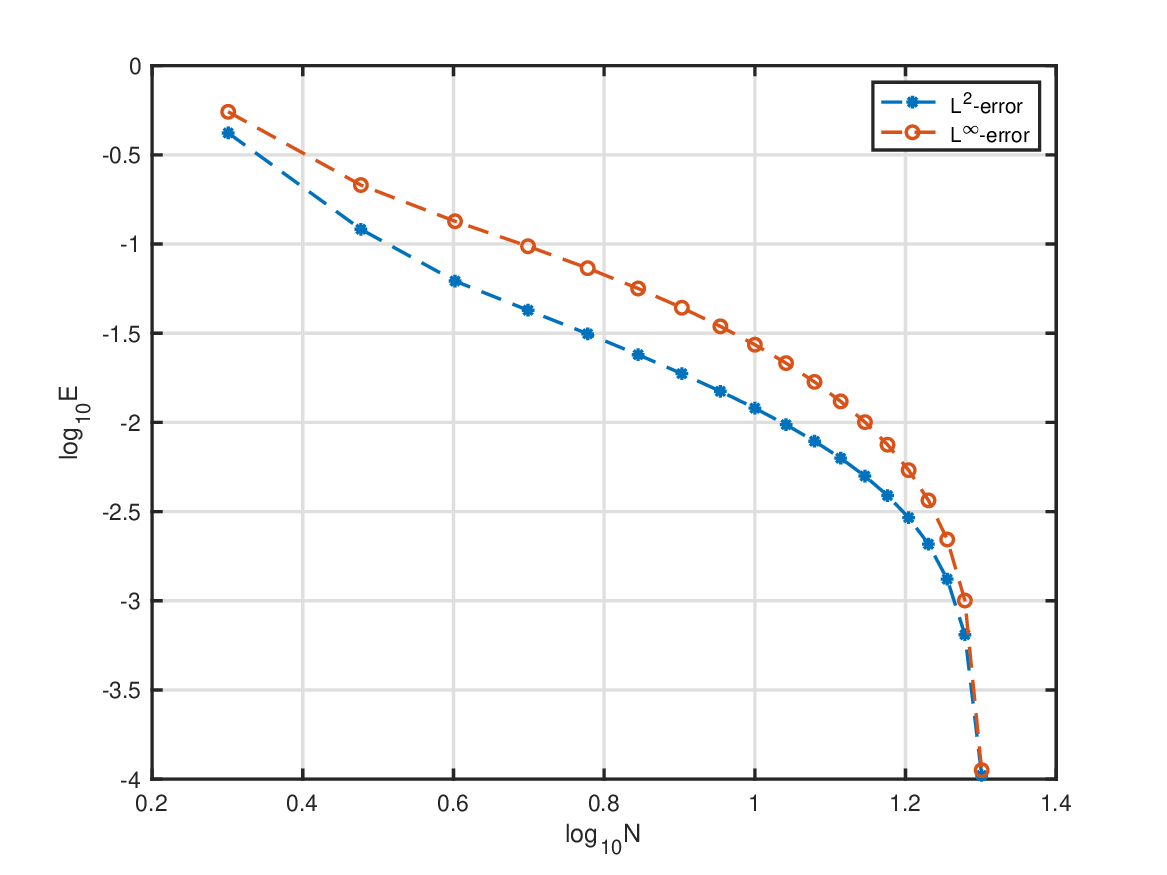}}
		\caption{\bf{Example \ref{example_third_NO3}} with $\lambda=1$}
		\label{Fig_third_NO3_1}
	\end{figure}

	\begin{table}[!ht]
		\centering
		\caption{\bf{Example \ref{example_third_NO3} } with $\lambda=\frac{1}{2}$: $\left\|e(\theta)\right\|_{0,\omega^{\alpha,\beta,1}}$ and $\|e(\theta)\|_{\infty}$.}
		\begin{tabular}{llllll}
			\hline$N$ & 7 & 9 & 11 & 13 & 17\\
			\hline$L^2$-error & $6.66271\mathrm{e}-04$ & $5.84874 \mathrm{e}-05$ & $3.37130 \mathrm{e}-06$ & $8.07408 \mathrm{e}-08$ & $1.64438 \mathrm{e}-08$\\
			$L^{\infty}$-error & $1.66020\mathrm{e}-03$ & $1.44999\mathrm{e}-04$ & $8.30929\mathrm{e}-06$ & $1.19880\mathrm{e}-07$ &$2.21299\mathrm{e}-08$\\
			\hline
		\end{tabular}
		\label{tabular_third_NO3_e}
	\end{table}
	
	\begin{table}[!ht]
		\centering
		\caption{\bf{Example \ref{example_third_NO3}} with  $\lambda=\frac{1}{2}$: $\left\|e^*(\theta)\right\|_{0,\omega^{\alpha,\beta,1}}$ and $\|e^*(\theta)\|_{\infty}$.}
		\begin{tabular}{llllll}
			\hline$N$& 7 & 10 & 13 & 15 & 17\\
			\hline$L^2$-error & $1.68075\mathrm{e}-03$ & $1.31019 \mathrm{e}-04$ & $2.62447\mathrm{e}-06$ & $2.19824\mathrm{e}-06$ & $1.53059\mathrm{e}-07$\\
			$L^{\infty}$-error  & $4.14564\mathrm{e}-03$ & $3.20731 \mathrm{e}-04$ & $6.49866\mathrm{e}-06$ & $5.46109\mathrm{e}-06$ & $3.78046\mathrm{e}-07$\\
			\hline
		\end{tabular}
		\label{tabular_third_NO3_e*}
		\end{table}
\end{example}

\begin{example}\label{example_third_NO4}
	Now we consider the problem with the unknown exact solution:
	\begin{equation}
		\left\{\begin{array}{l}
			t^{\gamma}y^{\prime}(t)=p(t)y(t)+t^{\frac{3}{2}}e^{-t} y(\varepsilon t)+\sin (2 t)-\int_0^t(t-s)^{-\mu}s^{\mu+\gamma-1}s^{1+\mu}(1+sin(ts))y(s) d s \\
			\qquad \quad+\varepsilon ^{-\gamma}\int_0^{\varepsilon t}(\varepsilon t-\tau)^{-\mu}\tau^{\mu+\gamma-1}K_2(t,\tau)y(\tau) d \tau, t \in[0,\frac{1}{2}] \\
			y(0)=3.
		\end{array}\right.
	\end{equation}
	where $\gamma=1, \varepsilon=0.5$ and $p(t)=t^{\frac{3}{2}}\cos(t), K_2(t,\tau)=\tau^{1+\mu}(1+\cos(t\tau))$. The reference "exact" solution is computed by $\lambda=\frac{1}{2}$ and $N=18$, then we choose $\lambda=\frac{1}{2}$ and $1$ for the numerical solution. We do not know the structure and properties of the exact solution, so it is much more complex to test the effectiveness and accuracy of the fractional collocation method. 
	
	The first case is $\mu=\frac{1}{2}$, as expected, the results exhibited exponential convergence when $\lambda=\frac{1}{2}$ and  algebraic convergence $\lambda=1$ respectively in Figures \ref{Fig_third_NO4_0.5}, \ref{Fig_third_NO4_1} and Tables \ref{tabular_third_NO4_e} and \ref{tabular_third_NO4_e*}.

	\begin{figure}
		\subfloat[$\lambda=\frac{1}{2},\left\|e(\theta)\right\|_{0,\omega^{\alpha,\beta,1}},\|e(\theta)\|_{\infty}$]{\includegraphics[width=0.5\textwidth]{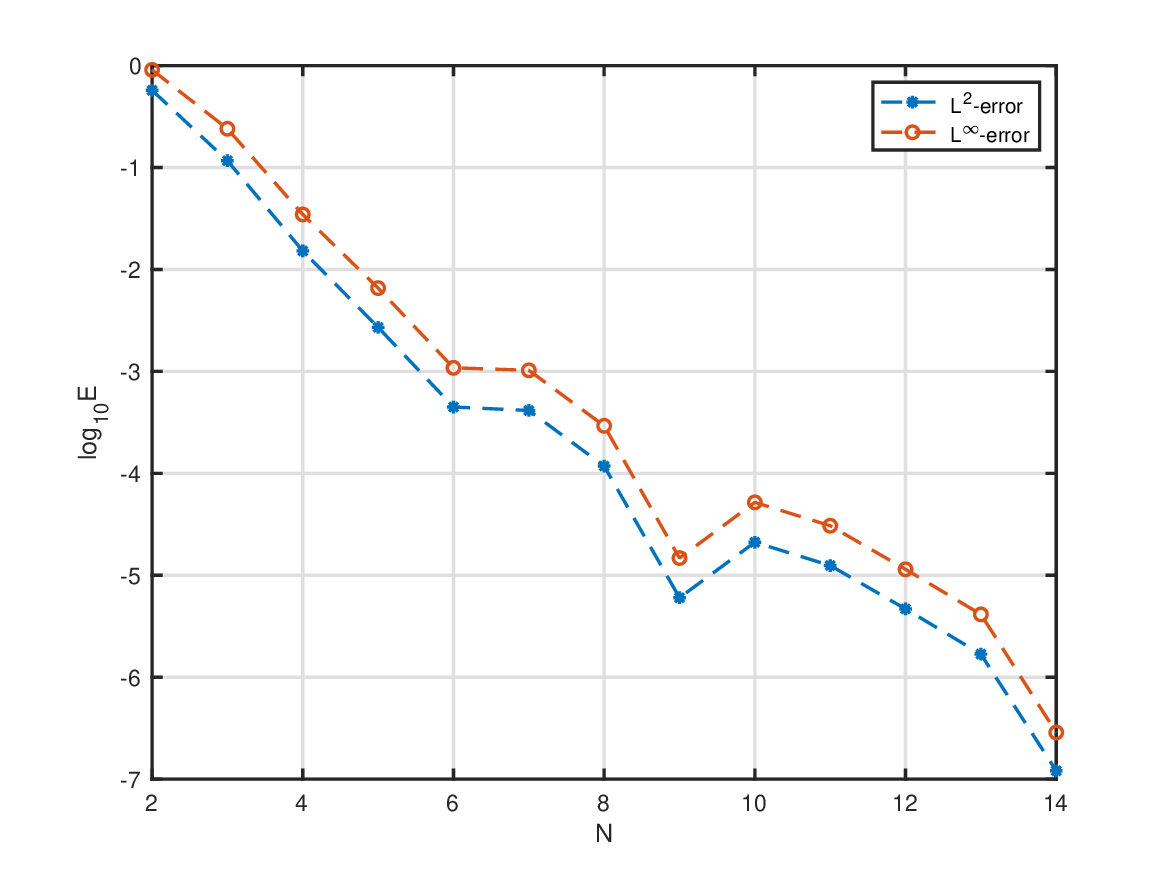}}%
		\hfill
		\subfloat[$\lambda=\frac{1}{2},\left\|e^*(\theta)\right\|_{0,\omega^{\alpha,\beta,1}},\|e^*(\theta)\|_{\infty}$]{\includegraphics[width=0.5\textwidth]{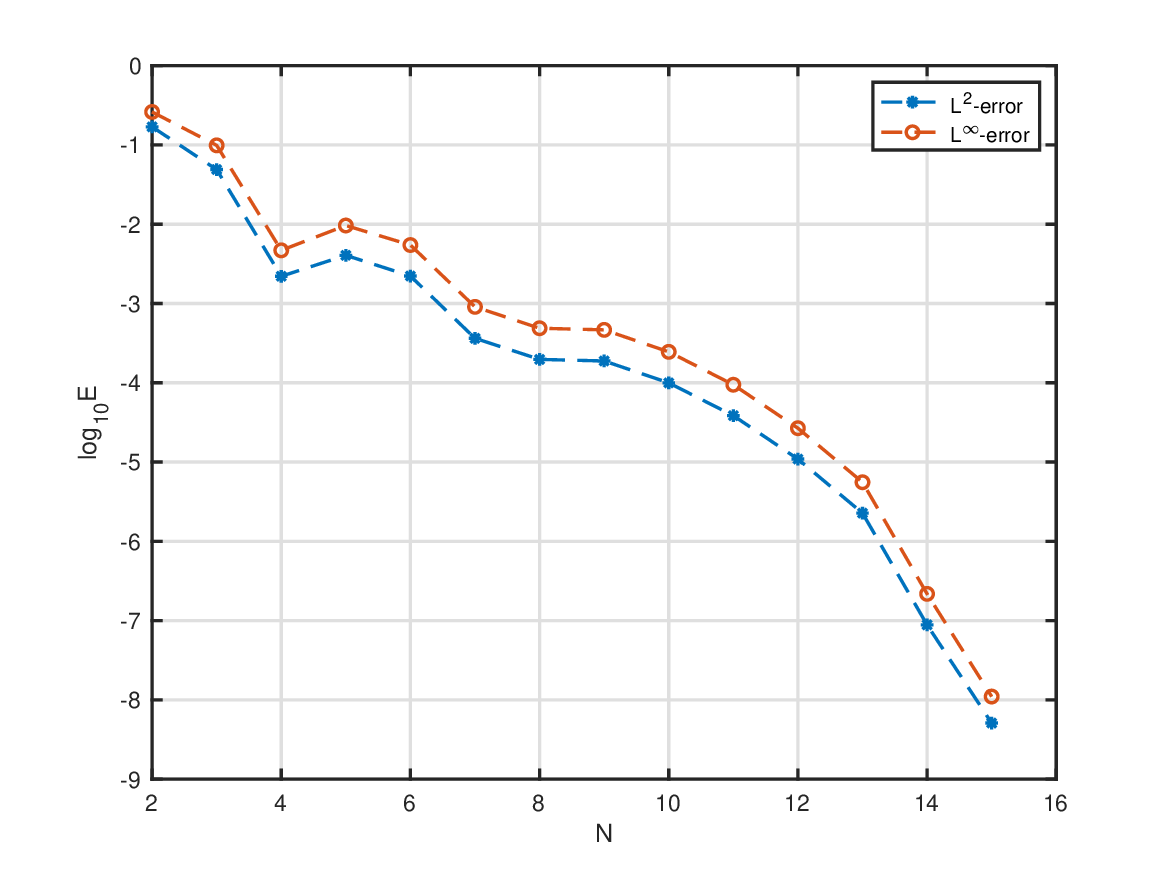}}
		\caption{\bf{Example \ref{example_third_NO4} } with $\lambda=\frac{1}{2}$}
		\label{Fig_third_NO4_0.5}
	\end{figure}
	
	\begin{figure}
		\subfloat[$\lambda=1,\left\|e(\theta)\right\|_{0,\omega^{\alpha,\beta,1}},\|e(\theta)\|_{\infty}$]{\includegraphics[width=0.5\textwidth]{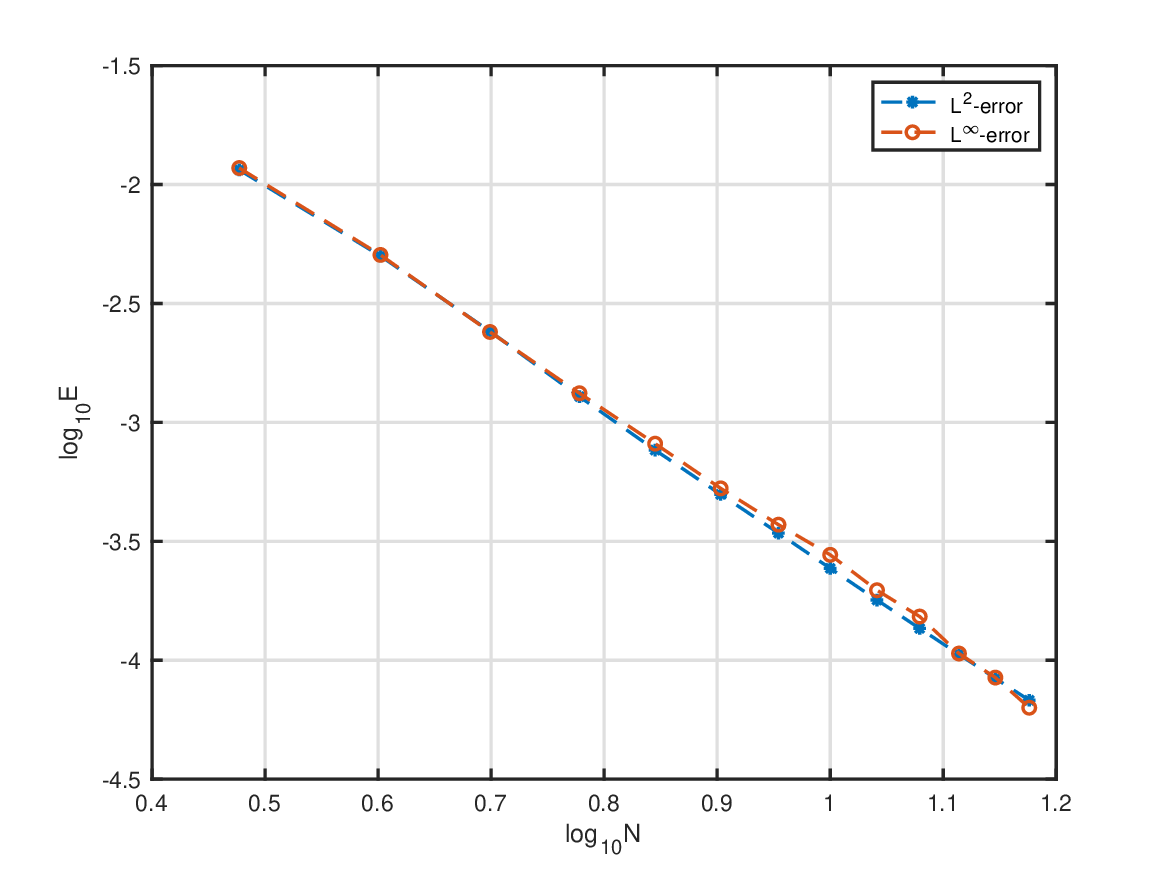}}%
		\hfill
		\subfloat[$\lambda=1,\left\|e^*(\theta)\right\|_{0,\omega^{\alpha,\beta,1}},\|e^*(\theta)\|_{\infty}$]{\includegraphics[width=0.5\textwidth]{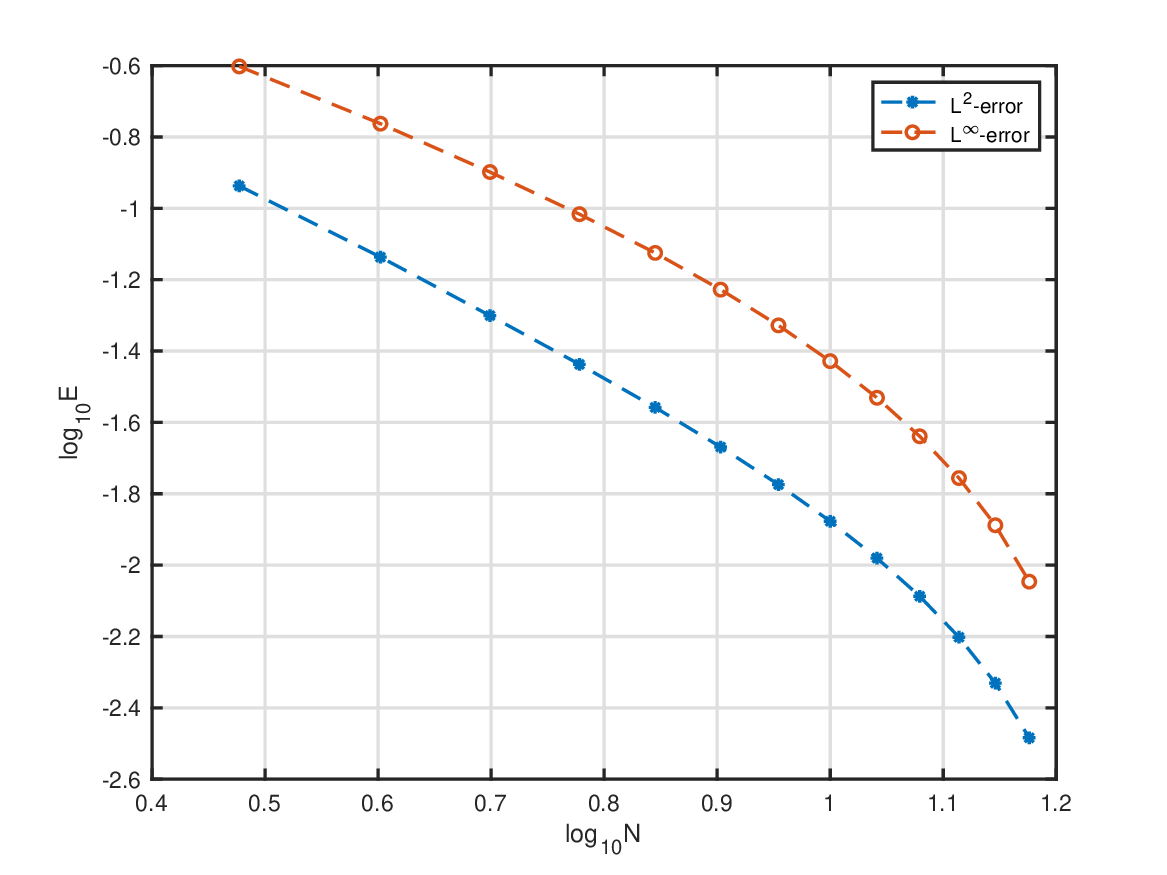}}
		\caption{\bf{Example \ref{example_third_NO4} } with $\lambda=1$}
		\label{Fig_third_NO4_1}
	\end{figure}
	
	\begin{table}[!ht]
		\centering
		\caption{\bf{Example \ref{example_third_NO4} } with $\lambda=\frac{1}{2}$: $\left\|e(\theta)\right\|_{0,\omega^{\alpha,\beta,1}}$ and $\|e(\theta)\|_{\infty}$.}
		\begin{tabular}{llllll}
			\hline$N$ & 5 & 7 & 10 & 12 & 14\\
			\hline$L^2$-error & $2.70464\mathrm{e}-03$ & $4.13840 \mathrm{e}-04$ & $2.10939 \mathrm{e}-05$ & $4.67595\mathrm{e}-06$ & $1.20788\mathrm{e}-07$\\
			$L^{\infty}$-error & $6.56633\mathrm{e}-03$ & $1.02585\mathrm{e}-03$ & $5.19343 \mathrm{e}-05$ &$1.14426\mathrm{e}-05$ &$2.87010\mathrm{e}-07$\\
			\hline
		\end{tabular}
		\label{tabular_third_NO4_e}
	\end{table}
	
	\begin{table}[!ht]
		\centering
		\caption{\bf{Example \ref{example_third_NO4}} with  $\lambda=\frac{1}{2}$: $\left\|e^*(\theta)\right\|_{0,\omega^{\alpha,\beta,1}}$ and $\|e^*(\theta)\|_{\infty}$.}
		\begin{tabular}{llllll}
			\hline$N$& 5 & 7 & 10 & 13 & 15\\
			\hline$L^2$-error & $4.04976\mathrm{e}-03$ & $3.63609\mathrm{e}-04$ & $9.97540\mathrm{e}-05$ &$2.26774\mathrm{e}-06$ &$5.10405\mathrm{e}-09$\\
			
			$L^{\infty}$-error & $9.62480\mathrm{e}-03$ & $9.06380\mathrm{e}-04$ & $2.45406\mathrm{e}-04$ &$5.56027\mathrm{e}-06$ &$1.10183\mathrm{e}-08$\\
			\hline
		\end{tabular}
		\label{tabular_third_NO4_e*}
	\end{table}
\end{example}

\begin{example}\label{example_third_NO5}
	Next, we consider another special case $\mu=2-\sqrt{2}$ where there is no $\lambda$ such that $y\left(t^{\frac{1}{\lambda}}\right)$ and $y^{\prime}\left(t^{\frac{1}{\lambda}}\right)$ are analytical. All other parameters remain unchanged. The exponential convergence result can be observed in Figure \ref{Fig_third_NO5_0.5}, by comparison, however, Figure \ref{Fig_third_NO5_1} exhibit an algebraic convergence result. The cause of this situation may be due to the fact that $y\left(t\right)$ and $y^{\prime}\left(t\right)$ are less smooth than $y\left(t^2\right)$ and $y^{\prime}\left(t^2\right)$. So fractional collocation method always converges faster than polynomial collocation method. Therefore,  it is very important to select the value of $\lambda$ to obtain better numerical errors.
	
	\begin{figure}
		\subfloat[$\lambda=\frac{1}{2},\left\|e(\theta)\right\|_{0,\omega^{\alpha,\beta,1}},\|e(\theta)\|_{\infty}$]{\includegraphics[width=0.5\textwidth]{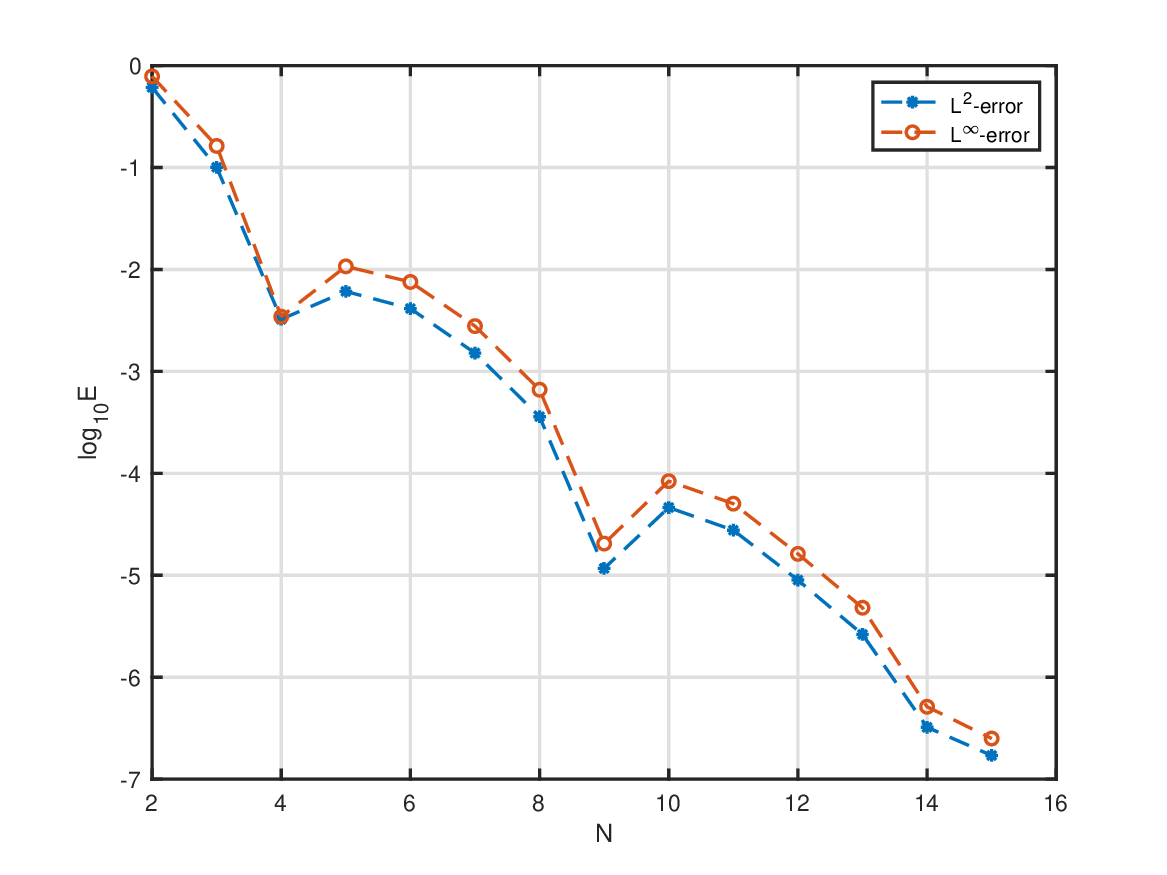}}%
		\hfill
		\subfloat[$\lambda=\frac{1}{2},\left\|e^*(\theta)\right\|_{0,\omega^{\alpha,\beta,1}},\|e^*(\theta)\|_{\infty}$]{\includegraphics[width=0.5\textwidth]{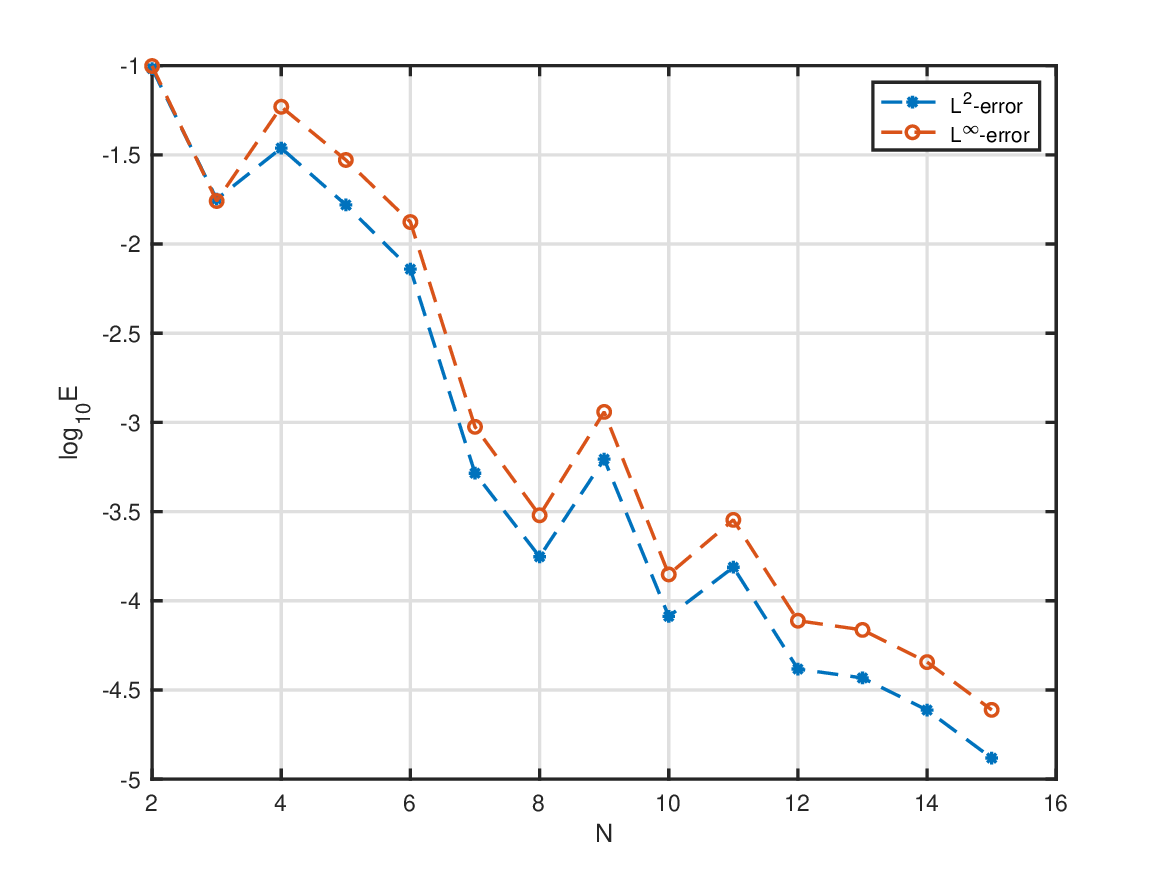}}
		\caption{\bf{Example \ref{example_third_NO5} } with $\lambda=\frac{1}{2}$}
		\label{Fig_third_NO5_0.5}
	\end{figure}
	
	\begin{figure}
		\subfloat[$\lambda=1,\left\|e(\theta)\right\|_{0,\omega^{\alpha,\beta,1}},\|e(\theta)\|_{\infty}$]{\includegraphics[width=0.5\textwidth]{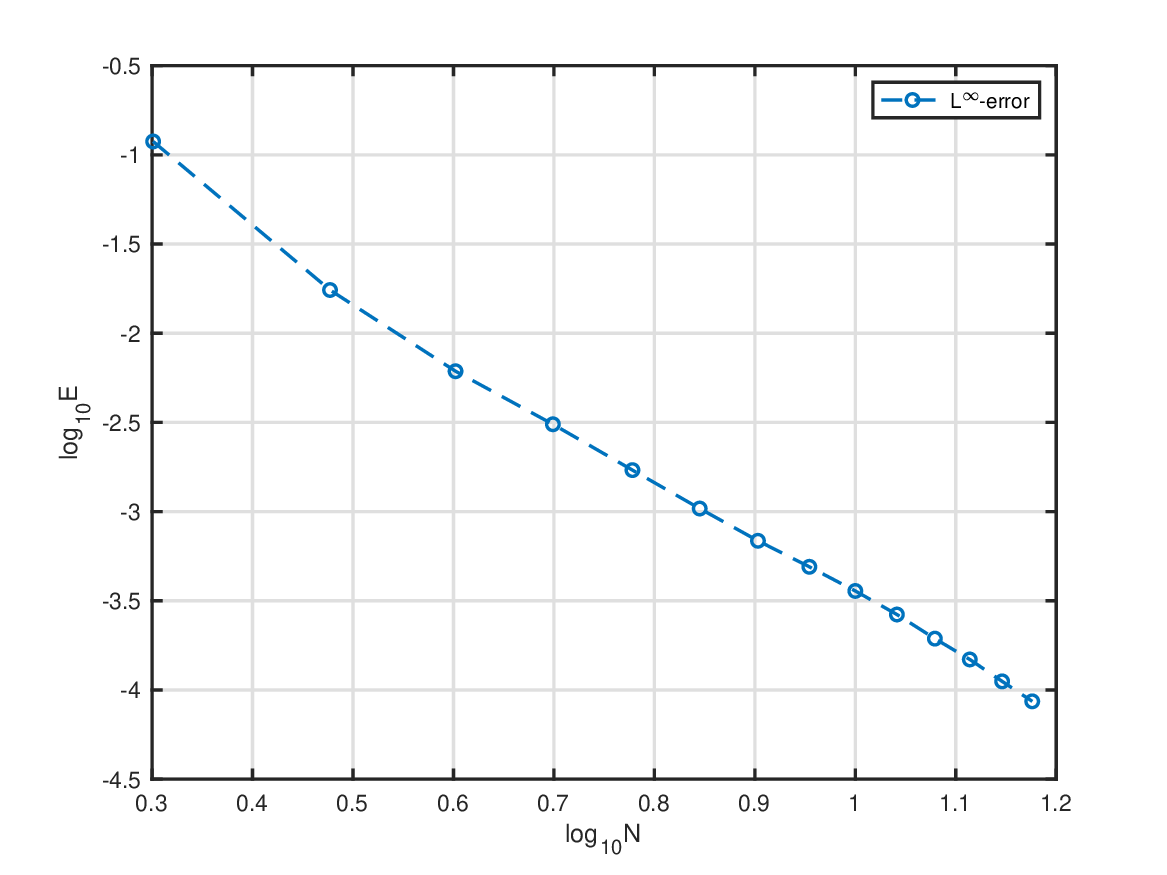}}%
		\hfill
		\subfloat[$\lambda=1,\left\|e^*(\theta)\right\|_{0,\omega^{\alpha,\beta,1}},\|e^*(\theta)\|_{\infty}$]{\includegraphics[width=0.5\textwidth]{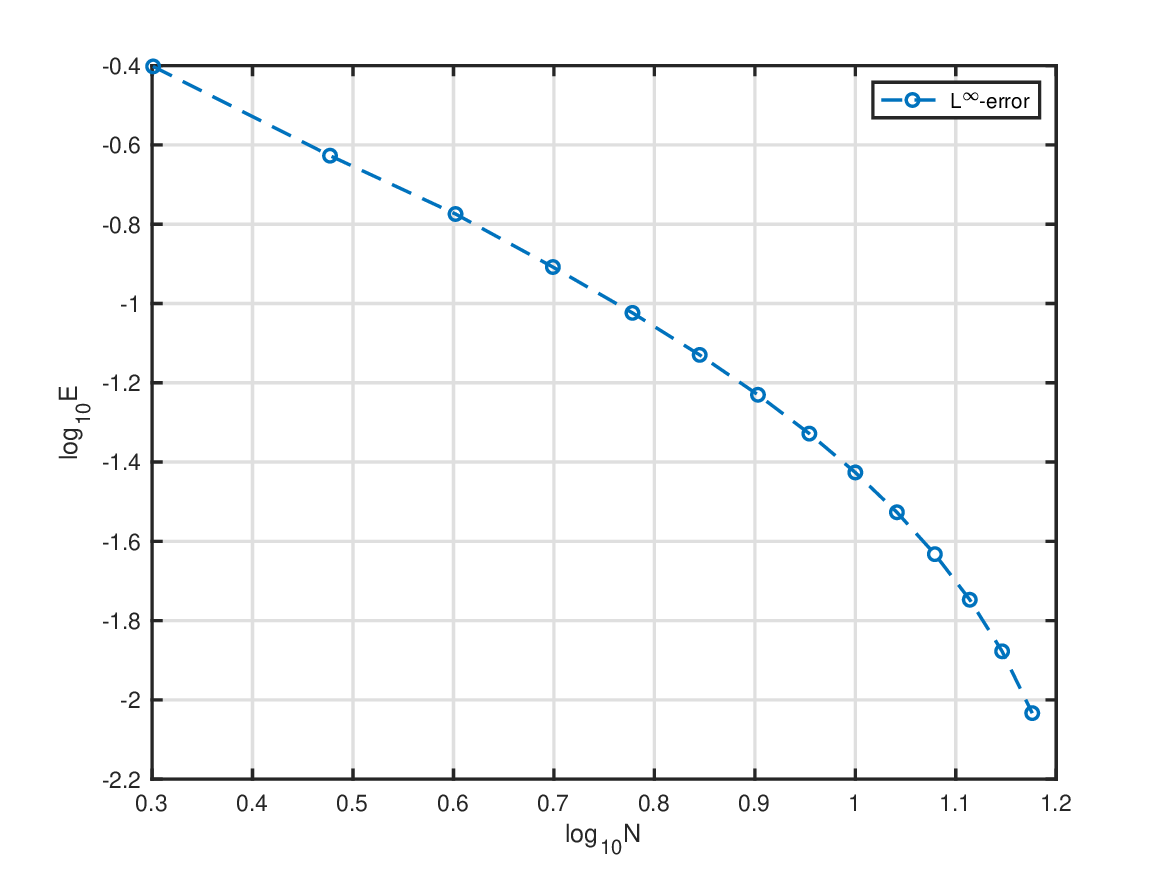}}
		\caption{\bf{Example \ref{example_third_NO5} } with $\lambda=1$}
		\label{Fig_third_NO5_1}
	\end{figure}
		\begin{table}[!ht]
		\centering
		\caption{\bf{Example \ref{example_third_NO5} } with $\lambda=\frac{1}{2}$: $\left\|e(\theta)\right\|_{0,\omega^{\alpha,\beta,1}}$ and $\|e(\theta)\|_{\infty}$.}
		\begin{tabular}{llllll}
			\hline$N$ & 5 & 8 & 10 & 13 & 15\\
			\hline$L^2$-error & $6.09366\mathrm{e}-03$ & $3.59672\mathrm{e}-04$ & $4.62133 \mathrm{e}-05$ & $2.63270\mathrm{e}-06$ & $1.70753\mathrm{e}-07$\\
			$L^{\infty}$-error & $1.07296\mathrm{e}-02$ & $6.61803\mathrm{e}-04$ & $8.38297 \mathrm{e}-05$ &$4.80641\mathrm{e}-06$ &$2.50450\mathrm{e}-07$\\
			\hline
		\end{tabular}
		\label{tabular_third_NO5_e}
	\end{table}
	
	\begin{table}[!ht]
		\centering
		\caption{\bf{Example \ref{example_third_NO5}} with  $\lambda=\frac{1}{2}$: $\left\|e^*(\theta)\right\|_{0,\omega^{\alpha,\beta,1}}$ and $\|e^*(\theta)\|_{\infty}$.}
		\begin{tabular}{llllll}
			\hline$N$& 5 & 8 & 10 & 13 & 15\\
			\hline$L^2$-error & $1.65675\mathrm{e}-02$ & $1.76491\mathrm{e}-04$ & $8.17439\mathrm{e}-05$ &$3.69433\mathrm{e}-05$ &$1.31322\mathrm{e}-05$\\
			
			$L^{\infty}$-error & $2.95843\mathrm{e}-02$ & $3.01460\mathrm{e}-04$ & $1.40496\mathrm{e}-04$ &$6.86499\mathrm{e}-05$ &$2.44657\mathrm{e}-05$\\
			\hline
		\end{tabular}
		\label{tabular_third_NO5_e*}
	\end{table}
\end{example}

\section{Conclusions}\label{section_Conclusions}
This paper presents a spectral collocation method using fractional Jacobi polynomials for solving the third-kind weakly singular Volterra integro-differential equations with proportional delays. Firstly, we elucidate the regularity conditions of the equation. The most important contribution is that we obtain the error estimates for the fractional numerical scheme. The  estimates can achieve exponential convergence when the appropriate $\lambda$ is selected to make $y(t^{\frac{1}{\lambda}})$ and $y^{\prime}(t^{\frac{1}{\lambda}})$ analytical. Finally, numerical experiments verify the theoretical proof results and the effectiveness of the proposed method.

\end{document}